\documentclass[10pt,reqno]{amsart}  
\usepackage[english]{babel}
\usepackage[p,osf]{cochineal}
\usepackage[scale=.95,type1]{cabin}
\usepackage[zerostyle=c,scaled=.94]{newtxtt}
\usepackage[T1]{fontenc} 
\usepackage[utf8]{inputenc} 
\usepackage{amsthm}

\usepackage{amsmath}
\usepackage{amssymb}
\usepackage{pgfplots} 
\pgfplotsset{compat=newest} 
\usepgfplotslibrary{fillbetween,colorbrewer}
\usepackage{amsfonts}
\usepackage{amsaddr}
\usepackage{xspace,subcaption}
\usepackage{enumitem} 
\usepackage{csquotes}
\usepackage{xcolor}
\usepackage[colorlinks]{hyperref}
\hypersetup{colorlinks, breaklinks, citecolor=teal,filecolor=black}
\usepackage[nosort,capitalise]{cleveref}
\crefname{enumi}{}{}
\crefname{equation}{}{}
\usepackage{graphicx}
\usepackage{mathtools}
\mathtoolsset{centercolon}
\usepackage{bm}
\usepackage{pgfkeys}
\usepackage{pgf,interval}
\intervalconfig{soft open fences}

\DeclareMathOperator{\Tr}{Tr}

\DeclareMathOperator{\rank}{rank}

\DeclareMathOperator{\E}{\mathbf{E}}
\DeclareMathOperator{\Prob}{\mathbf{P}}

\DeclareMathOperator{\Av}{Av}

\newcommand{\ov}{\overline}
\newcommand{\ii}{\mathrm{i}}

\newcommand{\F}{\mathbf{F}}

\ifdefined\C
\renewcommand{\C}{\mathbf{C}}
\else
\newcommand{\C}{\mathbf{C}}
\fi

\newcommand{\cI}{\mathcal{I}}
\newcommand{\un}{\underline}
\newcommand{\vx}{\bm{x}}

\newcommand{\vl}{{\bm{l}}}
\newcommand{\vu}{{\bm{u}}}
\newcommand{\vj}{{\bm{j}}}

\newcommand{\vy}{\bm{y}}

\newcommand{\wt}{\widetilde}
\newcommand{\wh}{\widehat}

\newcommand{\cE}{\mathcal{E}}
\newcommand{\av}{\mathrm{av}}
\newcommand{\iso}{\mathrm{iso}}
\newcommand{\R}{\mathbf{R}}
\newcommand{\N}{\mathbf{N}}
\newcommand{\Z}{\mathbf{Z}}

\newcommand{\cG}{\mathcal{G}}
\newcommand{\cO}{\mathcal{O}}
\newcommand{\co}{{\scriptstyle\mathcal{O}}}

\newcommand{\dif}{\operatorname{d}\!{}}
\DeclarePairedDelimiter{\braket}{\langle}{\rangle}%
\DeclarePairedDelimiter{\abs}{\lvert}{\rvert}%
\DeclarePairedDelimiter{\norm}{\lVert}{\rVert}%
\providecommand\given{}
\newcommand\SetSymbol[1][]{\nonscript\:#1\vert\allowbreak\nonscript\:\mathopen{}}
\DeclarePairedDelimiterX{\tuple}[1](){\renewcommand\given{\SetSymbol[\delimsize]}#1}
\DeclarePairedDelimiterX{\set}[1]\{\}{\renewcommand\given{\SetSymbol[\delimsize]}#1}
\DeclarePairedDelimiterXPP{\landauO}[1]{\cO}(){}{#1}
\DeclarePairedDelimiterXPP{\landauo}[1]{\co}(){}{#1}
\DeclarePairedDelimiterXPP{\landauOprec}[1]{\cO_\prec}(){}{#1}

\usepackage[giveninits=true,url=false,doi=false,isbn=false,eprint=true,datamodel=mrnumber,sorting=nty,maxcitenames=4,maxbibnames=99,backref=false,block=space,backend=biber,bibstyle=phys]{biblatex} 
\AtEveryBibitem{\clearfield{month}}
\AtEveryCitekey{\clearfield{month}} 
\renewbibmacro{in:}{}
\DeclareFieldFormat[article]{title}{\emph{#1}} 
\DeclareFieldFormat{mrnumber}{\ifhyperref{\href{http://www.ams.org/mathscinet-getitem?mr=#1}{\nolinkurl{MR#1}}}{\nolinkurl{#1}}}
\DeclareFieldFormat{pmid}{\ifhyperref{\href{https://www.ncbi.nlm.nih.gov/pubmed/#1}{\nolinkurl{PMID#1}}}{\nolinkurl{#1}}}
\DeclareFieldFormat{eprint}{\ifhyperref{\href{https://arxiv.org/abs/#1}{\nolinkurl{arXiv:#1}}}{\nolinkurl{#1}}}
\renewbibmacro*{doi+eprint+url}{%
  \iftoggle{bbx:doi}{\printfield{doi}}{}
  \newunit\newblock%
  \printfield{mrnumber}%
  \newunit\newblock%
  \printfield{pmid}%
  \newunit\newblock%
  \printfield{eprint}%
  \iftoggle{bbx:url}{\usebibmacro{url+urldate}}{}}
\newlist{genprop}{enumerate}{1}
\bibliography{bibclt}

\calclayout

\makeatletter
\newcommand\footnoteref[1]{\protected@xdef\@thefnmark{\ref{#1}}\@footnotemark}
\makeatother

\numberwithin{equation}{section} 
\newtheorem{theorem}{Theorem}[section]
\newtheorem{assumption}{Assumption}
\newtheorem{lemma}[theorem]{Lemma}
\newtheorem{proposition}[theorem]{Proposition}
\newtheorem{remark}[theorem]{Remark} 
\newtheorem{corollary}[theorem]{Corollary}

\date{\today}
\author{Giorgio Cipolloni}
\address{Princeton Center for Theoretical Science, Princeton University, Princeton, NJ 08544, USA}
\author{L\'aszl\'o Erd\H{o}s\(^\#\)}
\address{IST Austria, Am Campus 1, 3400 Klosterneuburg, Austria}
\author{Dominik Schr\"oder\(^{\ast}\)}
\address{Institute for Theoretical Studies, ETH Zurich, Clausiusstr.\ 47, 8092 Zurich, Switzerland}
\email{gc4233@princeton.edu} 
\email{lerdos@ist.ac.at}
\email{dschroeder@ethz.ch}
\thanks{\(^\#\)Supported by ERC Advanced Grant ``RMTBeyond'' No.~101020331}
\thanks{\(^\ast\)Supported by Dr.\ Max R\"ossler, the Walter Haefner Foundation and the ETH Z\"urich Foundation}
\subjclass[2010]{60B20, 15B52} 
\keywords{Local law, quantum unique ergodicity, eigenvector overlap}
\title{Rank-uniform local law for Wigner matrices}
\date{\today}

\begin{document} 

\begin{abstract} 
    We prove a general local law  for Wigner matrices
    which optimally 
    handles observables of arbitrary rank and thus
    it unifies the well-known averaged and
    isotropic local laws. 
    As an application, we prove a central limit theorem in quantum unique ergodicity  (QUE), i.e. we show
    that the quadratic forms of a general  deterministic matrix  $A$
    on the bulk eigenvectors of a Wigner matrix have approximately Gaussian fluctuation.
    For the bulk spectrum, we thus
    generalize  our previous result~\cite{2103.06730}  valid for  test matrices
    $A$ of large rank as well as  the result of Benigni and Lopatto~\cite{2103.12013} valid
    for specific small rank observables.
\end{abstract}

\thispagestyle{empty}

\maketitle

\section{Introduction}

Wigner random matrices are \(N\times N\)  random Hermitian  matrices \(W=W^*\)
with centred, independent, identically distributed \emph{(i.i.d.)} entries up to the symmetry constraint \(w_{ab} = \ov{w_{ba}}\).
Originally introduced by E. Wigner~\cite{MR77805} to study spectral gaps of large atomic nuclei, 
Wigner matrices have become  the most studied random matrix ensemble since they represent the simplest
example of a fully chaotic quantum Hamiltonian beyond the explicitly computable Gaussian case.

A key conceptual feature of Wigner matrices, as well as a fundamental technical tool to study them,
is the fact that their resolvent $G(z):= (W-z)^{-1}$, with a spectral parameter $z$ away from
the real axis, becomes asymptotically deterministic in the large $N$ limit. The limit
is the scalar matrix $m(z)\cdot I$, where $m(z) = \frac{1}{2}(-z+\sqrt{z^2-4})$ is the Stieltjes transform 
of the \emph{Wigner semicircular density}, $\rho_\mathrm{sc}(x) = \frac{1}{2\pi}\sqrt{4-x^2}$,
which is the $N\to \infty$ limit of the empirical density of the eigenvalues of $W$
under the standard normalization $\E |w_{ab}|^2= 1/N$.
The \emph{local law on optimal scale} asserts
that this limit holds even when $z$ is very close to the real axis as long as $|\Im z|\gg 1/N$.
Noticing that the imaginary part of the Stieltjes transform resolves the spectral measure
on a  scale comparable with $|\Im z|$, this condition is necessary for a deterministic limit to hold
since on scales of order $1/N$, comparable with the typical eigenvalue spacing, the resolvent is
genuinely fluctuating.

The limit $G(z)\to m(z)\cdot I$ holds in a natural appropriate topology, namely when tested against
deterministic $N\times N$ matrices $A$, i.e.\ in the form $\braket{G(z) A } \to m(z) \braket{ A}$,
where $\braket{\cdot }:=\frac{1}{N}\Tr (\cdot) $ denotes the normalized trace. 
It is essential that the test matrix $A$ is deterministic, no analogous limit can hold
if  $A$ were random and strongly correlated with $W$, 
e.g.\ if $A$ were a spectral projection of $W$. 

The first optimal local law for Wigner matrices was proven for $A=I$ in~\cite{MR2481753}, see also~\cite{MR2784665,MR2669449,MR3405612,MR4125959}
extended later to more general matrices $A$ in the form that\footnote{ 
Traditional local laws for Wigner matrices did not 
consider  a general test matrix $A$; this concept appeared later in connection with
more general random matrix ensebles, see e.g.~\cite{MR3941370}.}
\begin{equation}\label{aveloc}
    \big| \braket{(G(z)-m(z))A}\big|\le \frac{N^\xi\| A\|}{N\eta}, \qquad \eta:=|\Im z|,
\end{equation}
holds with very high probability  for any fixed $\xi>0$ if $N$ is sufficiently large.  By \emph{optimality} in this paper we always mean up to a tolerance factor $N^\xi$, this is 
a natural byproduct of our method yielding  very high probability  estimates under 
the customary moment condition, see~\eqref{moments} later\footnote{We remark that the $N^\xi$ tolerance
factor can be improved to logarithmic factors under slightly different conditions,
see e.g.~\cite{MR3405612, MR3633466,MR4125959}. }.  The estimate~\eqref{aveloc} is 
called the \emph{average local law} and it controls the error in terms of the  standard Euclidean matrix norm $\| A\|$ of $A$.
It holds for arbitrary deterministic matrices $A$ and it is also  optimal in this generality
with respect to the dependence on $A$:  e.g.\ for $A=I$ the trace
\(\braket{G-m}\) is approximately complex Gaussian with standard deviation~\cite{MR3678478}
\[
\sqrt{\E\abs{\braket{G-m}}^2}\approx \frac{\abs{m'(z)}\abs{\Im m(z)}}{N\eta\abs{m(z)}^2}\sim \frac{1}{N\eta}, 
\qquad \eta=\abs{\Im z}= N^{-\alpha}, \: \alpha\in [0, 1),
\]
but~\eqref{aveloc} is far from being
optimal when applied to matrices with small rank. Rank one matrices, $A= \vy \vx^*$, are especially important since they
give the asymptotic behaviour of resolvent matrix elements  $G_{\vx \vy}:= \langle \vx, G \vy\rangle$.
For such special test matrices, a separate \emph{isotropic local law} of the  optimal form
\begin{equation}\label{isoloc}
    \abs*{\braket{\vx, (G(z)-m(z))\vy }}\le \frac{N^\xi \rho^{1/2} \| \vx\|\|\vy\| }{\sqrt{N\eta}}, \qquad \eta=|\Im z|, \quad \rho:=\abs{\Im m(z)},
\end{equation}
has been proven; see~\cite{MR2981427} for special coordinate vectors and later~\cite{MR3103909} for general 
vectors $\vx, \vy$, as well as~\cite{MR3134604,MR3704770, MR3800833,MR3941370,MR3800833} for more general ensembles. Note that a direct application of~\eqref{aveloc} to $A= \vy \vx^*$ would give a bound of order $1/\eta$
instead of the optimal $1/\sqrt{N\eta}$ in~\eqref{isoloc} which is an unacceptable overestimate 
in the most interesting small $\eta$-regime. More generally, the average local law~\eqref{aveloc}
performs badly when $A$ has effectively small rank, i.e. if only a few eigenvalues of $A$ are comparable with 
the norm $\|A\|$ and most other eigenvalues are  much smaller or even zero.

Quite recently  we found that the average local law~\eqref{aveloc} is also suboptimal for another
class of test matrices $A$, namely for traceless matrices. In~\cite{MR4334253} we proved
that 
\begin{equation}\label{tracelessloc}
    \big| \braket{(G(z)-m(z))A}\big| =  \big| \braket{G(z)A}\big|\le \frac{N^\xi\| A\|}{N\sqrt{\eta}}, \qquad \eta=|\Im z|,
\end{equation}
for any deterministic matrix $A$ with $\braket{A}=0$, i.e. traceless observables yield
an additional $\sqrt{\eta}$ improvement in the error. The optimality of this bound for general traceless $A$ was demonstrated
by identifying the nontrivial  Gaussian fluctuation of $N\sqrt{\eta}\braket{G(z)A}$  in~\cite{2012.13218}.

While the mechanism behind the suboptimality of~\eqref{aveloc} for small rank and traceless $A$ is very different, 
their common core is that estimating the size of $A$ simply by the Euclidean norm is too crude for several
important classes of $A$. In this paper we present a local law which unifies all three local laws~\eqref{aveloc}, \eqref{isoloc}
and~\eqref{tracelessloc} by identifying the appropriate way to measure the size of $A$. Our main result 
(Theorem~\ref{theorem multi G local law}, $k=1$ case) shows that 
\begin{equation}\label{unifiedloc}
    \big| \braket{(G(z)-m(z))A}\big| \le 
    \frac{N^\xi }{N\eta} \abs{\braket{A}}+\frac{N^\xi \rho^{1/2}\braket{ |\mathring{A}|^2}^{1/2} }{N\sqrt{\eta}}, \qquad \eta=|\Im z|,
    \quad \rho =\abs{\Im m(z)},
\end{equation}
holds with very high probability, where $\mathring{A}:= A- \braket{A}$ is the traceless part of $A$. It is straightforward
to check that~\eqref{unifiedloc} implies~\eqref{aveloc}, \eqref{isoloc}
and~\eqref{tracelessloc}, moreover, it optimally interpolates between full rank and rank-one matrices $A$, hence we call~\eqref{unifiedloc} the \emph{rank-uniform local law} for Wigner matrices. Note that an optimal local law for matrices of intermediate rank  was previously unknown; 
indeed the local laws \eqref{aveloc}--\eqref{isoloc} are optimal only for  essentially full   
rank and  essentially finite rank observables, respectively.
The proof of the optimality of~\eqref{unifiedloc} follows from identifying the scale of the Gaussian fluctuation 
of its left hand side. Its standard deviation for traceless $A$ is
\begin{equation}
    \sqrt{\E |\braket{GA}|^2}  \approx \frac{\abs{m}\sqrt{\Im m}\braket{AA^\ast}^{1/2}}{N\sqrt{\eta}} \sim 
      \frac{\rho^{1/2}\braket{AA^*}^{1/2}}{N\sqrt{\eta}};
\end{equation}
this relation was established for matrices with bounded norm $\| A\|\lesssim 1$
in~\cite{MR3155024, 2012.13218}.

The key observation that traceless $A$ substantially improves the error term~\eqref{tracelessloc} 
compared with~\eqref{aveloc} was the conceptually new input behind our recent proof of the \emph{Eigenstate Thermalisation
Hypothesis} in~\cite{MR4334253} followed by the proof of the normal fluctuation
in the quantum unique ergodicity for Wigner matrices in~\cite{MR4413210}. Both results concern
the behavior of the \emph{eigenvector overlaps}, i.e.\ quantities of the form $\braket{\vu_i, A\vu_j}$, 
where $\{\vu_i\}_{i=1}^N$ are the normalized eigenvectors of $W$.
The former result stated that
\begin{equation}\label{eth}
    \big| \langle \vu_i, \mathring{A} \vu_j \rangle\big|=\big| \langle \vu_i, A \vu_j \rangle - \delta_{ij} \braket{A}\big| \le 
    \frac{N^\xi \| \mathring{A}\|}{\sqrt{N}}
\end{equation}
holds
with very high probability for any $i,j$ and for any fixed $\xi>0$.
The latter result established the
optimality of~\eqref{eth} for $i=j$  by showing that $\sqrt{N} \langle \vu_i, \mathring{A} \vu_i \rangle$
is asymptotically Gaussian when the corresponding eigenvalue  lies in the bulk
of the spectrum. The variance of $\sqrt{N} \langle \vu_i, \mathring{A} \vu_i \rangle$
was shown to be $\braket{|\mathring{A}|^2}$  in~\cite{MR4413210} but we needed to assume that 
$\braket{|\mathring{A}|^2}\ge c\| \mathring{A}\|^2$
with some fixed positive constant $c$, i.e.\ that  the rank of $\mathring{A}$ was essentially macroscopic.  

As the second main result of the current paper, we now remove this unnatural condition and show the standard Gaussianity of 
the normalized overlaps
$[N/\braket{|\mathring{A}|^2}]^{1/2} \langle \vu_i, \mathring{A} \vu_i \rangle$ for bulk indices under the  optimal and natural
condition that $\braket{|\mathring{A}|^2}\gg N^{-1} \| \mathring{A}\|^2$, which essentially  ensures
that $\mathring{A}$ is not of finite rank. 
This improvement is possible thanks to improving the dependence of the error terms in the local laws from $\|\mathring{A}\|$
to $\braket{|\mathring{A}|^2}^{1/2}$ similarly to the improvement in~\eqref{unifiedloc} over~\eqref{tracelessloc}.
We will  also need a multi-resolvent version of this improvement 
since off-diagonal overlaps  $\langle \vu_i, A \vu_j \rangle$ 
are not accessible via single-resolvent local laws;
in fact $|\langle \vu_i, A \vu_j \rangle|^2$  is intimately
related to $\braket{ \Im G(z)A \Im G(z') A^*}$ with two different spectral parameters $z, z'$, analysed
in Theorem~\ref{theorem multi G local law}. As a corollary we will show the following improvement of~\eqref{eth}
(see~Theorem~\ref{thm:eth})
\begin{equation}\label{ethnew}
    \big| \langle \vu_i, A \vu_j \rangle - \delta_{ij} \braket{A}\big| \le \frac{N^\xi \braket{|\mathring{A}|^2}^{1/2}}{\sqrt{N}}
\end{equation}
for the bulk indices. The analysis at the edge is deferred to later work.

Gaussian fluctuation of diagonal overlaps with a special low rank observable has been proven earlier.
Right after~\cite{MR4413210} was posted on the arXiv,   Benigni and Lopatto
in an independent work~\cite{MR4397177}
proved the standard Gaussian fluctuation of $[N/|S|]^{1/2}\big[\sum_{a\in S} |u_i(a)|^2 - |S|/N]$
whenever $1\ll |S|\ll N$, i.e. they considered $\langle \vu_i, \mathring{A} \vu_i \rangle$ for the special case when the matrix $A$
is the projection on coordinates from the set $S$. Their result also holds at the edge. The
condition $|S|\ll N$ requires $A$ to have small rank, 
hence it is  complementary to our old  condition $\braket{|\mathring{A}|^2}\ge c\| \mathring{A}\|^2$ from \cite{MR4413210}
for  projection operators.
The natural condition $|S|\gg 1$ is the special case of our new improved condition $\braket{|\mathring{A}|^2}\gg N^{-1} \| \mathring{A}\|^2$. In particular,  our new result covers~\cite{MR4397177}
as a special case  in the bulk and it gives a uniform treatment of all observables in full generality.

The methods of~\cite{MR4413210} and~\cite{MR4397177} are very different albeit they both 
rely on the \emph{Dyson Brownian motion (DBM)}, complemented by fairly standard \emph{Green function comparison (GFT)}
techniques. Benigni and Lopatto focused
on the joint Gaussianity of the individual eigenvector entries $u_i(a)$ (or more generally, linear functionals
$\langle q_\alpha, \vu_i\rangle$ with deterministic unit vectors $q_\alpha$) in the spirit of 
the previous quantum ergodicity results by Bourgade and Yau~\cite{MR3606475} 
operating with the so-called \emph{eigenvector moment flow} from~\cite{MR3606475} complemented by 
its "fermionic" version by Benigni~\cite{MR4242625}.
This approach becomes less effective when more 
entries  need to be controlled simultaneously and it seems to have a natural limitation at $|S|\ll N$. 

Our method viewed the eigenvector overlap $\langle \vu_i, \mathring{A} \vu_i \rangle$ and its offdiagonal version
$\langle \vu_i, \mathring{A} \vu_j \rangle$ as one unit without translating it into a sum of rank one projections
$\langle \vu_i, q_\alpha\rangle\langle q_\alpha, \vu_j\rangle$ via the spectral decomposition of $\mathring{A}$.
The corresponding flow for overlaps with arbitrary $A$,
called the \emph{stochastic eigenstate equation}, 
was introduced by Bourgade, Yau and Yin in~\cite{MR4156609} (even though they applied it
to the special case when $A$ is a projection,  their formalism is general). The analysis of this new flow 
is more involved than the eigenvector moment flow since it operates on a geometrically more
complicated higher dimensional space. However, the substantial part of this analysis has been done by Marcinek and Yau~\cite{MR4272266}
and we heavily relied on their work in our proof~\cite{MR4413210}.

We close this introduction by
commenting on our methods. 
The main novelty of the current paper is the proof of the rank-uniform local laws involving the Hilbert-Schmidt norm
$\braket{|\mathring{A}|^2}^{1/2}$
instead of the Euclidean matrix norm $\|\mathring{A}\|$.  This is done in Section~\ref{sec loclaw proof}
and it will directly imply the improved overlap estimate~\eqref{ethnew}. Once this estimate is available,
both the DBM  and the GFT parts of the proof in the current paper are essentially the same as in~\cite{MR4413210}, hence we will not
give all details, we only point out the differences. While this can be done very concisely for the GFT in Appendix~\ref{sec:GFT},
for the DBM part  we need to recall large part of the necessary setup in Section~\ref{sec:see} for the convenience of the reader. 

As to our main result, the  general scheme to prove single resolvent local laws
has been well established and traditionally it consisted of two parts: (i) the derivation
of an approximate self-consistent equation that $G-m$ satisfies and (ii) estimating the key
fluctuation term in this equation. The proofs of the multi-resolvent local laws
follow the same scheme, but the self-consistent equation is considerably more
complicated and its stability is more delicate, see e.g.~\cite{MR4334253, MR4372147}
where general multi-resolvent local laws were proven. The main complication lies in part (ii)
where a high moment estimate is  needed for the fluctuation term. The corresponding
cumulant expansion results in many terms which have typically  been organized and estimated
by a graphical Feynman diagrammatic scheme. A reasonably manageable power counting handles all diagrams 
for the purpose of proving~\eqref{aveloc} and \eqref{isoloc}. However, in the multi-resolvent setup
or if we aim at some improvement,
the diagrammatic approach becomes very involved since the right number of additional improvement factors
needs to be gained from every single graph. This was the case many times before:
(i) when a small factor (so-called \emph{``sigma-cell''}) was extracted  at the cusp~\cite{MR4134946}, (ii)
when we proved that the correlation between the resolvents of the Hermitization of an i.i.d.\ random matrix shifted by two
different spectral parameters $z_1, z_2$ decays in $1/|z_1-z_2|$~\cite{1912.04100},  
and (iii) more recently 
when the  gain of order $\sqrt{\eta}$  due to the  traceless $A$ in~\eqref{tracelessloc} was obtained in~\cite{MR4334253}.

Extracting $\braket{|\mathring{A}|^2}^{1/2}$ instead of $\|A\|$, especially in the multi-resolvent case,  seems
even more involved in this way since estimating $A$ simply by its norm appears everywhere in  any diagrammatic expansion.
However, very recently in~\cite{MR4479913}
we introduced a new method of \emph{a system of master inequalities}
that circumvents the full diagrammatic expansion. The power of this method
was demonstrated   by fully extracting
the maximal $\sqrt{\eta}$-gain from traceless $A$ even in the multi-resolvent setup;
the same result seemed out of reach with the diagrammatic method used for
the single-resolvent setup in~\cite{MR4334253}.
In the current paper we extend this  technique to obtain the optimal control in terms of
$\braket{|\mathring{A}|^2}^{1/2}$ instead of $\|\mathring{A}\|$ for  single resolvent local laws. However, the master inequalities in this paper
are different from the ones in~\cite{MR4479913}; in fact they are much tighter,
since the effect we extract now is much more delicate. We also obtain  a similar optimal control for the multi-resolvent local laws
needed to prove the Gaussianity of the bulk eigenvector overlaps under
the optimal condition on $A$. 

\subsection*{Notations and conventions} We denote vectors by bold-faced lower case Roman letters \({\bm x}, {\bm y}\in\C ^N\), for some \(N\in\N\). Vector and matrix norms, \(\norm{\vx}\) and \(\norm{A}\), indicate the usual Euclidean norm and the corresponding induced matrix norm. For any \(N\times N\) matrix \(A\) we use the notation \(\braket{ A}:= N^{-1}\Tr  A\) to denote the normalized trace of \(A\). Moreover, for vectors \({\bm x}, {\bm y}\in\C^N\) and matrices  \(A\in\C^{N\times N}\) we define
\[ \braket{ {\bm x},{\bm y}}:= \sum_{i=1}^N \overline{x}_i y_i, \qquad A_{\vx\vy}:=\braket{\vx,A\vy}.\]

We will use the concept of ``with very high probability'' meaning that for any fixed \(D>0\) the probability of an \(N\)-dependent event is bigger than \(1-N^{-D}\) if \(N\ge N_0(D)\). We introduce the notion of \emph{stochastic domination} (see e.g.~\cite{MR3068390}): given two families of non-negative random variables
\[
X=\tuple*{ X^{(N)}(u) \given N\in\N, u\in U^{(N)} }\quad\text{and}\quad Y=\tuple*{ Y^{(N)}(u) \given N\in\N, u\in U^{(N)} }
\] 
indexed by \(N\) (and possibly some parameter \(u\)  in some parameter space $U^{(N)}$), 
we say that \(X\) is stochastically dominated by \(Y\), if for all \(\xi, D>0\) we have \begin{equation}\label{stochdom}
\sup_{u\in U^{(N)}} \Prob\left[X^{(N)}(u)>N^\xi  Y^{(N)}(u)\right]\leq N^{-D}
\end{equation}
for large enough \(N\geq N_0(\xi,D)\). In this case we use the notation \(X\prec Y\) or \(X= \landauOprec*{Y}\).
We also use the convention that \(\xi>0\) denotes an arbitrary small constant which is independent of \(N\).

Finally, for positive quantities \(f,g\) we write \(f\lesssim g\) and \(f\sim g\) if \(f \le C g\) or \(c g\le f\le Cg\), respectively, for some constants \(c,C>0\) which depend only on the constants appearing in the moment condition, see~\eqref{moments} later.

\section{Main results}

\begin{assumption}\label{Wigner def}
    We say that \(W=W^\ast\in\C^{N\times N}\) is a real symmetric/complex hermitian 
    Wigner-matrix if the entries \((w_{ab})_{a\le b}\) in the upper triangular part are independent and satisfy 
    \begin{equation}
        w_{ab}\stackrel{\mathrm{d}}{=} N^{-1/2}\times\begin{cases}
            \chi_\mathrm{od}, & a\ne b\\
            \chi_\mathrm{d}, & a=b,
        \end{cases}
    \end{equation}
    for some real random variable \(\chi_\mathrm{d}\) and
    some real/complex random variable \(\chi_\mathrm{od}\) 
    of mean \(\E \chi_\mathrm{d}=\E \chi_\mathrm{od}=0\) and variances \(\E \abs{\chi_\mathrm{od}}^2=1\), \(\E\chi_\mathrm{od}^2=0\), \(\E \chi_\mathrm{d}^2=1\) in the complex, and \(\E \abs{\chi_\mathrm{od}}^2=\E \chi_\mathrm{od}^2=1\), \(\E \chi_\mathrm{d}^2=2\) in the real case\footnote{We assumed that $\sigma:=\E\chi_{\mathrm{od}}^2=0$, $\E \chi_{\mathrm{d}}^2=1$ in the complex case, and that $\E\chi_\mathrm{d}^2=2$ in the real case only for notational simplicity. All the results presented below hold under the more general assumption $|\sigma|<1$ and general variance for diagonal entries. The necessary modifications in the proofs are straightforward
    and will be omitted.}. We furthermore assume that for every \( n\ge 3\) 
    \begin{equation}\label{moments}
        \E \abs{\chi_\mathrm{d}}^n + \E \abs{\chi_\mathrm{od}}^n \le C_n
    \end{equation}
    for some constant $C_n$, in particular all higher order cumulants \(\kappa_n^\mathrm{d},\kappa_n^\mathrm{od}\) of \(\chi_\mathrm{d}, \chi_\mathrm{od} \) are finite for any \(n\).
\end{assumption}
Our results hold for both symmetry classes, but for definiteness we prove the main results in the real case, the changes for the complex case being minimal.

For a spectral parameter \(z\in\C\) with \(\eta:=\abs{\Im z}\gg N^{-1}\) the resolvent \(G=G(z)=(W-z)^{-1}\) of a \(N\times N\) 
Wigner matrix \(W\) is well approximated by a constant multiple \(m\cdot I\) of the identity matrix, where \(m=m(z)\) 
is the Stieltjes transform of the semicircular distribution \(\sqrt{4-x^2}/(2\pi)\)  and satisfies the equation 
\begin{equation}\label{MDE}
    -\frac{1}{m}=m+z,\qquad \Im m\Im z>0. 
\end{equation}
We set $\rho(z): = |\Im m(z)|$, which approximates the density of eigenvalues near $\Re z$ in a window of size $\eta$. 

We first recall the classical local law for Wigner matrices, both in its tracial and isotropic 
form~\cite{MR2481753,MR2871147,MR3103909,MR3800833}: 
\begin{theorem}\label{theorem G av local law}
    Fix any $\epsilon>0$, then 
     it holds that 
    \begin{equation}\label{oneG}
        \abs{\braket{G-m}}\prec \frac{1}{N\eta}, 
        \qquad \abs{\braket{\vx,(G-m)\vy}}\prec \norm{\vx}\norm{\vy}\Bigl(\sqrt{\frac{\rho}{N\eta}}+\frac{1}{N\eta}\Bigr)
    \end{equation}
 uniformly in any deterministic    vectors  $\vx, \vy$ and spectral parameter $z$ with $\eta=\abs{\Im z}\ge N^{-1+\epsilon}$ and \(\Re z\in\R\),
 where $\rho=\abs{\Im m(z)}$. 
\end{theorem}
Our main result is the following optimal multi-resolvent local law with Hilbert-Schmidt norm error terms.
 Compared to~\cref{theorem G av local law} we formulate the bound only in averaged sense since, due to the Hilbert-Schmidt norm in the error term, the isotropic bound is a special case with one of the traceless matrices being a centred rank-one matrix, see~\cref{corol}.
\begin{theorem}[ Averaged multi-resolvent local law]\label{theorem multi G local law}
    Fix \(\epsilon>0\), let \(k\ge 1\) and consider \(z_1,\ldots,z_{k}\in\C\) with \(N\eta\rho\ge N^\epsilon\), for \(\eta:=\min_i\abs{\Im z_i},\rho:=\max_i\abs{\Im m(z_i)},d:=\min_i\mathrm{dist}(z_i,[-2,2])\), and let \(A_1,\ldots,A_k\) be deterministic traceless matrices, $\braket{A_i}=0$. Set $G_i:= G(z_i)$ and $m_i:= m(z_i)$ for all $i\le k$. Then we have the local law on optimal scale\footnote{The constant \(10\) is arbitrary and can be replaced by any positive constant.}
    \begin{equation}\label{loc1}
        \abs{\braket{G_1A_1\cdots G_k A_k-m_1\cdots m_k A_1\cdots A_k}} \prec N^{k/2-1} \prod_{i=1}^k \braket{\abs{A_i}^2}^{1/2} \times\begin{cases}
            \sqrt{\frac{\rho}{N\eta}}, & d<10\\ 
            \frac{1}{\sqrt{N}d^{k+1}}, & d\ge 10.
        \end{cases}
    \end{equation}
\end{theorem}

\begin{remark}
    We also obtain generalisations of~\cref{theorem multi G local law} where each \(G\) may be replaced by a product of \(G\)'s and \(\abs{G}'s\), see~\cref{lemma Psi G prod} later.
\end{remark}

Due to the Hilbert-Schmidt sense of the error term we obtain an isotropic variant of~\cref{theorem multi G local law} as an immediate corollary by choosing $A_k= N \vy\vx^* -\braket{\vx, \vy}$ in~\eqref{loc1}.  
\begin{corollary}[Isotropic local law]\label{corol} Under the setup and conditions of Theorem~\ref{theorem multi G local law}, 
    for any vectors \(\vx,\vy\) it holds that 
    \begin{equation}
        \abs{\braket{\vx,(G_1A_1\cdots A_{k-1} G_k-m_1\cdots m_k A_1\cdots A_{k-1})\vy}}\prec \norm{\vx}\norm{\vy}
         N^\frac{k-1}{2}\prod_{i=1}^{k-1} \braket{\abs{A_i}^2}^{1/2} \times\begin{cases}
            \sqrt{\frac{\rho}{N\eta}}, & d<10\\ 
            \frac{1}{\sqrt{N}d^{k+1}}, & d\ge 10.
        \end{cases}
    \end{equation}
\end{corollary}

We now compare~\cref{theorem multi G local law} to the previous result~\cite[Theorem~2.5]{MR4479913} where an error term  $N^{-1}\eta^{-k/2}\prod_i \norm{A_i}$ was proven for~\eqref{loc1}. For clarity we focus on the really interesting \(d<10\) regime.
\begin{remark}\label{rmk:compare}
    For \(k=1\) our new estimate for traceless $A$:
    \begin{equation}
   \label{eq:imphs1}
     \big|\braket{(G-m)A}\big|=\big|\braket{GA}\big|\prec \frac{\sqrt{\rho}}{N\sqrt{\eta}}\braket{|A|^2}^{1/2}, 
    \end{equation}
        is strictly better than the one in~\cite[Theorem~2.5]{MR4479913}, 
    since \(\braket{\abs{A}^2}\le \norm{A}^2\) always holds, but
    \(\braket{\abs{A}^2}\) can be much smaller than \(\norm{A}^2\) for small rank \(A\). In addition,~\eqref{eq:imphs1} features
    an additional factor $\sqrt{\rho}\lesssim 1$ that is considerably smaller than 1 near the spectral edges.
          
    For larger \(k\ge 2\) the relationship depends on the relative size of the Hilbert-Schmidt and operator norm of the \(A_i\)'s, 
    as well as on the size of $\eta$. We recall~\cite{MR2351844} that the \emph{numerical rank} of \(A\) is defined as \(r(A):=N\braket{\abs{A}^2}/\norm{A}^2\le\rank(A)\) and say that \(A\) is \(\alpha\)-mesoscopic for some \(\alpha\in[0,1]\) if \(r(A)=N^\alpha\).
    If for some \(k\ge 2\) all \(A_i\) are \(\alpha\)-mesoscopic, then~\cref{theorem multi G local law} improves upon~\cite[Theorem 2.5]{MR4479913} whenever \(\eta\ll N^{( 1  -\alpha k) /(k-1)}\).
\end{remark}

Local laws on optimal scales can give certain  information on eigenvectors as well. Let $\lambda_1\le \lambda_2 \le\ldots
\le \lambda_N$ denote the eigenvalues and $\{ {\bm u}_i\}_{i=1}^N$ 
the corresponding orthonormal eigenvectors of $W$. Already the single-resolvent isotropic local law~\eqref{oneG}
implies the \emph{eigenvector delocalization}, i.e.\ that $\| {\bm u_i} \|_\infty \prec N^{-1/2}$. 
More generally\footnote{Under stronger
decay conditions on the distribution of $\chi_\mathrm{d}, \chi_\mathrm{od}$
even the optimal bound $\| {\bm u_i} \|_\infty \le C \sqrt{\log N/N}$
for the bulk and $\| {\bm u_i} \|_\infty \le C \log N/\sqrt{N}$ for the edge eigenvectors has been proven~\cite{MR3418916},
see also~\cite{MR3534074} for a comprehensive summary of related results. Very recently even 
the optimal constant $C$ has been identified~\cite{MR4370471}.}  $|\langle {\bm x}, {\bm u_i} \rangle|\prec N^{-1/2}\norm{{\bm x}} $, i.e.\ eigenvectors behave as completely random unit vectors in the sense of considering their rank-\(1\) 
projections onto  any deterministic vector ${\bm x}$. 
 This concept can be greatly extended  to arbitrary deterministic 
observable matrix $A$ leading to the following results motivated both by 
thermalisation ideas from physics~\cite{29862983,9897286,9964105,1509.06411}
 as well as by \emph{Quantum (Unique) Ergodicity (QUE)} in 
mathematics~\cite{Shni1974,MR916129,Col1985,MR1361757,MR1266075,MR2757360,MR2680500,MR1810753,MR3961083,MR3322309,MR3962004,MR3688032}.
  
\begin{theorem}[Eigenstate Thermalization Hypothesis]\label{thm:eth}
Let \(W\) be a Wigner matrix satisfying~\cref{Wigner def} and let \(\delta>0\). Then for any deterministic matrix \(A\) and any bulk indices \(i,j\in[\delta N,(1-\delta)N]\) it holds that 
\begin{equation}\label{eq ETH}
    \abs*{\braket{\bm u_i,A\bm u_j}-\delta_{ij}\braket{A}} \prec \frac{\braket{|\mathring A|^2}^{1/2}}{N^{1/2}},
\end{equation}
where \(\mathring A:=A-\braket{A}\) is the traceless part of $A$.
\end{theorem}
\begin{remark}~
\begin{enumerate}
    \item The result~\cref{eq ETH} was established in~\cite{MR4334253} with $\braket{\mathring A\mathring A^\ast}^{1/2}$ 
    replaced by \(\norm{\mathring A}\) uniformly in the spectrum (i.e. also for edge indices). 
    \item For rank-\(1\) matrices \(A=\vx\vx^\ast\) the bound~\cref{eq ETH} immediately implies the complete delocalisation of 
    eigenvectors in the form \(\abs{\braket{\vx,\bm u_i}}\prec N^{-1/2}\norm{\vx}\).
\end{enumerate}    
\end{remark}

Theorem~\ref{thm:eth} directly follows from the bound
\[
  \max_{i,j\in[\delta N,(1-\delta)N] } 
  N  \abs*{\braket{\bm u_i,\mathring{A}\bm u_j}}^2 \le C_\delta (N\eta)^2 \max_{E, E' \in [-2+\epsilon, 2-\epsilon]}
  \braket{ \Im G(E+\ii \eta) \mathring{A} \Im G(E'+\ii \eta)\mathring{A}^*}
\]
that is obtained by the spectral decomposition of both resolvents and the well-known eigenvalue
rigidity, with some explicit $\delta$-dependent constants $C_\delta$ and 
 $\epsilon=\epsilon(\delta)>0$ (see~\cite[Lemma 1]{MR4334253} for more details). The right hand side can be directly estimated 
using~\eqref{loc1} and finally choosing $\eta= N^{-1+\xi}$ for any small $\xi>0$ 
gives~\eqref{eq ETH} and thus proves Theorem~\ref{thm:eth}.

The next question is to establish a central limit theorem for the diagonal overlap in~\cref{eq ETH}. 
\begin{theorem}[Central Limit Theorem in the QUE]\label{theo:flucque}
    Let \(W\)  be a  real symmetric (\(\beta=1\)) or complex Hermitian (\(\beta=2\)) 
    Wigner matrix satisfying~\cref{Wigner def}.
    Fix small \(\delta,\delta'>0\) and let \(A=A^*\)  be a deterministic \(N\times N\) matrix with 
    \(N^{-1+\delta'}\norm{\mathring{A}}^2\lesssim\braket{\mathring{A}^2}\lesssim 1\). In the real symmetric case we also assume that \(A\in\mathbf{R}^{N\times N}\)
    is real.  Then for any bulk index \(i\in [\delta N, (1-\delta) N]\) we have a central limit theorem
    \begin{equation}
        \label{eq:clt}
        \sqrt{\frac{\beta N}{2\braket{\mathring{A}^2}}} \big[\braket{{\bm u}_i,A {\bm u_i}}-\braket{A}\big]\Rightarrow \mathcal{N},\qquad \mbox{as 
        \;\; \(N\to\infty\)}
    \end{equation}
    with \(\mathcal{N}\) being a standard real Gaussian random variable. Moreover, for any moment the speed of convergence is explicit (see~\eqref{eq:firststep}).
\end{theorem}

We require that $\braket{\mathring{A}^2}\gtrsim N^{-1+\delta'}\norm{\mathring{A}}^2$ in order to ensure that the spectral
distribution of $\mathring{A}$ is not concentrated to a finite number eigenvalues, i.e. that $\mathring{A}$ has effective rank $\gg 1$. Indeed, the statement in \eqref{eq:clt} does not hold for finite rank $A$'s, e.g.
if $A=\mathring{A}=|{\bf e}_x\rangle\langle {\bm e}_x|-|{\bm e}_y\rangle\langle {\bf e}_y|$, for some
$x\ne y\in [N]$, then $\braket{{\bm u}_i,\mathring{A} {\bm u_i}}=|{\bm u}_i(x)|^2-|{\bm u}_i(y)|^2$, which is the difference of two asymptotically independent $\chi^2$-distributed random variables (e.g. see \cite[Theorem 1.2]{MR3606475}).
More generally, the joint distribution of finitely many eigenvector overlaps has been identified in~\cite{MR3606475,MR4272266,MR4260468,MR3690289} for various related ensembles. 

\section{Proof of Theorem~\ref{theorem multi G local law}}\label{sec loclaw proof}
Within this section we prove~\cref{theorem multi G local law} in the critical  \(d<10\) regime. 
The \(d\ge 10\) regime is handled similarly but the estimates are much simpler; 
the necessary modifications are outlined in~\cref{appendix large d}.

In the subsequent proof we will often assume that a priori bounds, with some control parameters \(\psi_K^{\av/\iso}\ge 1\), of the form 
\begin{align}\label{a priori 0}
    \Psi_0^\av &=\Psi_0^\av(z_1):=N\eta \abs{\braket{G_1-m_1}}\prec\psi^\av_0\\\label{a priori k}
    \Psi_K^\av &=\Psi_K^\av(\bm A,\bm z):=\frac{N^{(3-K)/2}\eta^{1/2}}{\rho^{1/2} \prod_i \braket{\abs{A_i}^2}^{1/2}} \abs{\braket{[G_1A_1\cdots G_KA_K-m_1\cdots m_K A_1\cdots A_K}} \prec \psi^\av_K, \qquad  K\ge 1,  \\\label{a priori k iso}
    \Psi_K^\iso &=\Psi_K^\iso(\vx,\vy,\bm A,\bm z) :=\frac{N^{(1-K)/2 }\eta^{1/2} \rho^{-1/2}}{\norm{\vx}\norm{\vy} \prod_i \braket{\abs{A_i}^2}^{1/2}} \abs{\braket{\vx,[G_1A_1\cdots G_{K+1}-m_1\cdots m_{K+1} A_1\cdots A_K]\vy}} \prec \psi^\iso_K 
\end{align}
for certain indices \(K\ge 0\) have been established uniformly in deterministic traceless matrices \(\bm A=(A_1,\ldots,A_K)\), deterministic vectors $\vx, \vy$, and spectral parameters \(\bm z=(z_1,\ldots,z_K)\) with\footnote{In some estimates the domain of uniformity for the spectral parameters may shrink a bit. 
The pedantic way to track this effect is to define the concept of $(\epsilon, \ell)$-uniformity (see \cite[Definition 3.1]{MR4479913}) meaning
that an estimate holds uniformly for $N\rho \eta \ge \ell N^\epsilon$ with some $\ell\in \N$. We keep $\epsilon$ fixed but
$\ell$ may increase by one in some steps. However, this happens only finitely many times
and it is inconsequential to our main argument, hence we entirely omit tracking the $\ell$-dependence.\label{uniformity footnote}} \(N\eta\rho\ge N^\epsilon\). We stress that we do \emph{not} assume the estimates to be uniform in $K$.
Note that $\psi_0^\av$ is defined somewhat differently from $\psi_K^\av$, $K\ge 1$, but the definition of $\psi^\iso_K$ is the same for all $K\ge 0$. For intuition, the reader should think of the control parameters as essentially order one quantities, 
in fact our main goal will be to prove this fact. Note that by~\cref{theorem G av local law} we may set \(\psi_0^{\av/\iso}=1\).

As a first step we observe that~\cref{a priori 0,a priori k,a priori k iso} immediately imply estimates on more general averaged resolvent chains and isotropic variants. 
\begin{lemma}\label{lemma Psi G prod}
    (i)  Assuming~\cref{a priori 0} and~\cref{a priori k iso} for \(K=0\) holds uniformly in \(z_1\), then for any \(z_1,\ldots,z_l\) with \(N\eta\rho\ge N^\epsilon\) it holds that
    \begin{equation}
    \begin{split}
        \abs*{\braket{G_1 G_2\cdots G_l-m[z_1,\ldots,z_l]}}&\prec\frac{\psi^\av_0}{N\eta^{l}}, \\
         \abs*{\braket{\vx,(G_1 G_2\cdots G_l-m[z_1,\ldots,z_l])\vy}} &\prec\frac{\norm{\vx}\norm{\vy}\psi^\iso_0}{\eta^{l-1}}\sqrt{\frac{\rho}{N\eta}}
        \end{split}
    \end{equation} 
    where $m[z_1,\ldots,z_l]$ stands for the $l$-th divided difference of the function $m(z)$ from~\eqref{MDE}, explicitly
    \begin{equation}\label{divdiff}
        m[z_1,\ldots,z_l] = \int_{-2}^2\frac{\sqrt{4-x^2}}{2\pi}\prod_{i=1}^l \frac{1}{x-z_{i}}\dif x.
    \end{equation}
    (ii)  Assuming for some \(k\ge 1\) the estimates~\cref{a priori k,a priori k iso} for \(K=k\) have been established uniformly, then for 
    \(\cG_j:=G_{j,1}\cdots G_{j,l_j}\) with \(G_{j,i}\in \set{G(z_{j,i}),\abs{G(z_{j,i})}}\), 
    traceless matrices $A_i$ and \(\eta:=\min_{j,i}\abs{\Im z_{j,i}}\), \(\rho := \max_{j,i} \rho(z_{j,i})\) it holds that
    \begin{equation}\label{multicalG}
        \begin{split}
            \abs*{\braket{\cG_1A_1\cdots \cG_k A_k - m^{(1)}\cdots m^{(k)} A_1\cdots A_k} }&\prec \psi^\av_k N^{k/2-1}\sqrt{\frac{\rho}{N\eta}}  \prod_j \frac{\braket{\abs{A_j}^2}^{1/2}}{\eta^{l_j-1}},  \\
            \abs*{\braket{\vx,[\cG_1A_1\cdots  A_k\cG_{k+1} -m^{(1)}\cdots m^{(k+1)} A_1\cdots A_k]\vy} }&\prec \psi^\iso_k \norm{\vx}\norm{\vy}N^{k/2}\sqrt{\frac{\rho}{N\eta}}  \prod_j \frac{\braket{\abs{A_j}^2}^{1/2}}{\eta^{l_j-1}}, 
        \end{split}
    \end{equation}
    where
    \begin{equation}
        m^{(j)} := \int_{-2}^2\frac{\sqrt{4-x^2}}{2\pi} \prod_i g_{j,i}(x)\dif x
    \end{equation}
    and \(g_{j,i}(x)=(x-z_{j,i})^{-1}\) or \(g_{j,i}(x)=\abs{x-z_{j,i}}^{-1}\), depending on whether \(G_{j,i}=G(z_{j,i})\) or \(G_{j,i}=\abs{G(z_{j,i})}\).
\end{lemma}
\begin{proof}
    Analogous to~\cite[Lemma 3.2]{MR4479913}.
\end{proof}

The main result of this section is the following hierarchy of \emph{master inequalities}.
\begin{proposition}[Hierarchy of master inequalities]\label{prop master}
    Fix \(k\ge 1\), and assume that~\cref{a priori k,a priori k iso} have been established uniformly in \(\bm A\), and \(\bm z\)  with \(N\eta\rho\ge N^\epsilon\),
    for all \( K\le 2k\). Then\footnote{Following\footnoteref{uniformity footnote}, we omit tracking the spectral domains in the main text. We only  mention here that the pedantic formulation of Proposition~\ref{prop master} would assert that  if \eqref{a priori k}--\eqref{a priori k iso} hold $(\epsilon, \ell)$-uniformly, for some $\ell\in \N$, then the conclusions in \eqref{Psi master av}--\eqref{Psi master iso} hold $(\epsilon, \ell+1)$-uniformly.
}  it holds that 
    \begin{align}\label{Psi master av}
        \Psi_k^\av &\prec \Phi_k + \Bigl(\frac{\psi_{2k}^\av}{\sqrt{N\eta\rho}}\Bigr)^{1/2} + \psi^\av_{k-1}+\frac{\psi^\av_k}{\sqrt{N\eta}}+ (\psi_k^\iso)^{2/3}\Phi_{k-1}^{1/3} +\sum_{j=1}^{k-1}\sqrt{\psi^\iso_j\Omega_{k-j}(\psi_k^\iso+\Phi_{k-1})} \\\nonumber
        &\qquad\qquad+\frac{1}{N\eta}\sum_{j=1}^{k-1}\psi^\av_{j}\Bigl(1+\psi^\av_{k-j}\sqrt{\frac{\rho}{N\eta}}\Bigr)
        \\\label{Psi master iso}
        \Psi_k^\iso &\prec \Phi_k + \psi^\iso_{k-1}  +\frac{1}{N\eta}\Bigl[ \sum_{j=1}^k \psi_j^\av\Bigl(1+\sqrt{\frac{\rho}{N\eta}}\psi_{k-j}^\iso\Bigr)+\sum_{j=0}^{2k}\sqrt{\psi_{j}^\iso\psi_{2k-j}^\iso}+\psi_k^\iso\Bigr]
    \end{align}
    with the definitions
    \begin{equation}\label{Phi k def}
        \Omega_k := \sum_{k_1+k_2 +\cdots \le k}\prod_{i\ge 1}\Bigl(1+\psi_{k_i}^\iso\sqrt{\frac{\rho}{N\eta}}\Bigr)\le\Phi_k:=\sum_{k_1+k_2+\cdots\le k}\prod_{i=1}^2\Bigl(1+\frac{\psi^\iso_{2k_i}}{\sqrt{N\eta\rho}}\Bigr)^{1/2}\prod_{i\ge 3}\Bigl(1+\psi^\iso_{k_i}\sqrt{\frac{\rho}{N\eta}}\Bigr),    
    \end{equation}
    where the second sum is taken over an arbitrary number of  non-negative integers \(k_i\), with $k_i\ge 1$
    for $i\ge 3$, under the condition that their sum does not exceed $k$ (in the case of only one non-zero \(k_1\)
    the second factor and product in~\cref{Phi k def} are understood to be one and $\Phi_0=1$). 
\end{proposition}

This hierarchy has the structure that each $\Psi^{\av/\iso}_k$  is estimated partly by $\psi$'s with index higher than $k$,
which potentially is uncontrollable even if the coefficient of the higher order terms is small (recall that 
$1/(N\eta)$, and $1/(N\eta\rho)$ are small quantities). 
Thus the hierarchy must be complemented by another set of inequalities that estimate higher indexed $\Psi$'s
with smaller indexed ones even at the expense of a large constant. The success of this scheme eventually
depends on the  relative size of these small and large constants, so it is very  delicate.
We prove the following reduction inequalities 
to estimate the $\psi_l^{\av/\iso}$ terms with $k+1\le l\le 2k$ in \eqref{Psi master av}--\eqref{Psi master iso}
by $\psi$'s with indices smaller or equal than $k$. %

\begin{lemma}[Reduction lemma]\label{pro:redin}
    Fix \(1\le j\le k\) and assume that~\cref{a priori k}--\cref{a priori k iso} have been established uniformly for \(K\le 2k\). Then it holds that 
    \begin{equation}
        \label{eq:redinav}
        \Psi_{2k}^\av\lesssim \sqrt{\frac{N\eta}{\rho}}+\begin{cases}
           \sqrt{\frac{\rho}{N\eta}}(\psi_k^\av)^2\quad & k \,\, \mathrm{even}, \\
            \psi_{k-1}^\av+\psi_{k+1}^\av+\sqrt{\frac{\rho}{N\eta}}\psi_{k-1}^\av\psi_{k+1}^\av\quad & k \,\, \mathrm{odd},
                 \end{cases}
    \end{equation}
   and for even $k$ also that
    \begin{equation}
        \label{eq:rediniso}
        \Psi_{k+j}^\iso\lesssim \sqrt{\frac{N\eta}{\rho}}+ \left(\frac{N\eta}{\rho}\right)^{1/4}(\psi_{2j}^\av)^{1/2}+\psi_k^\iso+\left(\frac{\rho}{N\eta}\right)^{1/4}(\psi_{2j}^\av)^{1/2}\psi_k^\iso.    \end{equation}
\end{lemma}

The rest of the present section is structured as follows: In~\cref{sec av} we prove~\cref{Psi master av} while in~\cref{sec iso} we prove~\cref{Psi master iso}. Then in~\cref{sec bootstrap} we prove~\cref{pro:redin} and conclude the proof of~\cref{theorem multi G local law}. Before starting the main proof we collect some trivial estimates between Hilbert-Schmidt and operator norms using matrix H\"older inequalities. 
\begin{lemma}\label{trace prod lemma}
    For $N\times N$ matrices \(B_1,\ldots,B_k\) and \(k\ge 2\) it holds that 
    \begin{equation}\label{Bnorms}
        \abs*{\braket*{\prod_{i=1}^k B_i}}\le \prod_{i=1}^k \braket{\abs{B_i}^k}^{1/k} \le N^{k/2-1}\prod_{i=1}^k\braket{\abs{B_i}^2}^{1/2},
    \end{equation}
    and
    \begin{equation}\label{Bnorm}
        \norm{B}=\sqrt{\lambda_{\max}(\abs{B}^2)} \le N^{1/2}\braket{\abs{B}^2}^{1/2}.
    \end{equation}
\end{lemma}    

In the sequel we often drop the indices from \(G,A\), hence we write $(GA)^k$ for $G_1A_1\ldots G_kA_k$, and assume without loss of generality that \(A_i=A_i^\ast\) and \(\braket{A_i^2}=1\). 
We also introduce the convention in this paper that matrices denoted by capital $A$ letter are always traceless.

\subsection{Proof of averaged estimate \eqref{Psi master av} in~\cref{prop master}}\label{sec av}

We now identify the leading contribution of \(\braket{(GA)^k-m^{k}A^k}\). For any matrix-valued function \(f(W)\) we define the \emph{second moment renormalisation}, denoted by underline, as 
\begin{equation}\label{def:underline}
    \un{Wf(W)} := W f(W) - \wt\E_\mathrm{GUE} \wt W (\partial_{\wt W}f)(W)
\end{equation}
in terms of the directional derivative \(\partial_{\wt W}\) in the direction of an independent GUE-matrix \(\wt W\). The motivation for the second moment renormalisation is that by Gaussian integration by parts it holds that \(\E W f(W)=\wt\E\wt W (\partial_{\wt W} f)(W)\) whenever \(W\) is a Gaussian random matrix of zero mean, and \(\wt W\) is an independent copy of \(W\). In particular it holds that \(\E\un{Wf(W)}=0\) whenever \(W\) is a GUE-matrix, while \(\E\un{Wf(W)}\) is small but non-zero for GOE or non-Gaussian matrices.  By concentration and universality we expect that to leading order \(Wf(W)\) may be approximated by \(\wt\E \wt W(\partial_{\wt W}f)(W)\). Here the directional derivative \(\partial_{\wt W}f\) should be understood as
\[(\partial_{\wt W}f)(W):=\lim_{\epsilon\to 0} \frac{f(W+\epsilon \wt W)-f(W)}{\epsilon}.\]

In our application the function \(f(W)\) always is a (product of) matrix resolvents \(G(z)=(W-z)^{-1}\) and possibly deterministic matrices \(A_i\). This time we view the resolvent as a function of \(W\), \(G(W)= (W-z)^{-1}\) for any fixed \(z\). By the resolvent identity it follows that  
\begin{equation}\label{resolvent der}
    (\partial_{\wt W}G)(W) = \lim_{\epsilon\to 0}\frac{(W+\epsilon\wt W-z)^{-1}-(W-z)^{-1}}{\epsilon} = -\lim_{\epsilon\to 0}(W+\epsilon\wt W-z)^{-1} \wt W(W-z)^{-1}= -G(W)\wt W G(W),
\end{equation}
while the expectation of a product of GUE-matrices acts as an averaged trace in the sense 
\[\wt \E_\mathrm{GUE}\wt W A \wt W = \frac{1}{N}\sum_{ab}\Delta^{ab}A\Delta^{ba} = \braket{A} I, \]
where \(I\) denotes the identity matrix and \((\Delta^{ab})_{cd}:=\delta_{ac}\delta_{bd}\). Therefore, for instance, we have the identities 
\[\un{WG}=WG+\braket{G}G, \quad \un{WGAG}=WGAG+\braket{G}GAG + \braket{GAG}G = \un{WG}AG+\braket{GAG}G.\]

Finally, we want to comment on the choice of renormalising with respect to an independent GUE rather than GOE matrix. In fact this is purely a matter of convenience and we could equally have chosen the GOE-renormalisation. Indeed, we have 
\begin{equation*}
    \wt\E_\mathrm{GOE} \wt W A \wt W = \braket{A}I + \frac{A^t}{N} 
\end{equation*}
and therefore, for instance, 
\[\un{WG}_\mathrm{GOE} = \un{WG}_\mathrm{GUE} + \frac{G^t G}{N}\] 
which is a negligible difference. Our formulas below will be slightly simpler with our choice in~\cref{def:underline} even though now $E\un{W f(W)}$ is not exactly zero
even for \(W \sim \mathrm{GOE}\).
\begin{lemma}\label{underline repl lemma}
    We have
    \begin{align}
        \label{Ek}
        \braket*{\prod_{i=1}^k (G_iA_i) -\prod_{i=1}^k m_i A_i}\Bigl(1+\landauOprec*{\frac{\psi^\av_0}{N\eta}}\Bigr)&=-m_1\braket{\un{WG_1A_1\cdots G_kA_k}} + \landauOprec*{\cE_k^\av},
    \end{align}
    where \(\cE_1^\av=0\) and     
    \begin{equation}\label{def:Ek}
        \begin{split}
            \cE_2^\av &:= \sqrt{\frac{\rho}{N\eta}}\Bigl(\psi^\av_1+\frac{\psi^\av_0}{\sqrt{N\eta\rho}}+
            \frac{1}{N\eta}\sqrt{\frac{\rho}{N\eta}} (\psi^\av_1)^2\Bigr),\\
            \cE_k^\av &:= N^{k/2-1}\sqrt{\frac{\rho}{N\eta}}\Bigl(\psi^\av_{k-1}  +\frac{1}{N\eta}\sum_{j=1}^{k-1}\psi^\av_{j}\Bigl(1+\psi^\av_{k-j}\sqrt{\frac{\rho}{N\eta}}\Bigr)\Bigr)
        \end{split} 
    \end{equation}
    for \(k\ge 3\). 
\end{lemma}
\begin{proof}
    We start with the expansion
    \begin{equation}\label{selfcons}
        \begin{split}
            \Bigl(1+\landauOprec*{\frac{\psi^\av_0}{N\eta}}\Bigr)\braket{G_1A_1\cdots G_kA_k} &= m_1\braket{G_2\cdots G_kA_kA_1} - m_1\braket{\un{WG_1}A_1G_2\cdots G_kA_k}\\
            &= m_1\braket{G_2\cdots G_kA_kA_1} +  m_1\sum_{j=2}^{k} \braket{G_1\cdots G_j}\braket{G_j\cdots G_kA_k} \\
            &\quad - m_1\braket{\un{WG_1A_1G_2\cdots G_kA_k}},
        \end{split}
    \end{equation}
    due to
    \begin{equation}\label{G-m un}
        G=m-m\un{WG}+m\braket{G-m}G,
    \end{equation}
    where for \(k=1\) the first two terms in the right hand side of~\cref{selfcons} are not present. In the second step we extended the underline renormalization to the entire product $\un{WG_1A_1G_2\cdots G_kA_k}$
    at the expense of generating additional terms collected in the summation; this identity can directly be obtained
    from the definition~\eqref{def:underline}. Note that in the first line of \eqref{selfcons} we moved the term coming 
    from $m_1\braket{G_1-m_1}G_1$ of \eqref{G-m un} to the left hand side causing the error $\mathcal{O}_\prec(\psi_0^\av/(N\eta))$.
    For \(k\ge 2\), using~\cref{lemma Psi G prod,trace prod lemma} we estimate the second term in the second line of \eqref{selfcons} by
    \begin{equation}\label{GG}
        \begin{split}
           & \abs{\braket{G_1\ldots G_j}\braket{G_j\ldots G_kA_k}} \\
            &\qquad\quad\prec \Bigl(\frac{\rho}{\eta}\abs{\braket{A_1\cdots A_{j-1}}} + \frac{\psi^\av_{j-1}\rho^{1/2} N^{j/2-2}}{\eta^{3/2}}\Bigr)\Bigl(\abs{\braket{A_{j}\cdots A_k}}+\frac{\psi^\av_{k-j+1}\rho^{1/2} N^{(k-j)/2-1}}{\eta^{1/2}}\Bigr)\\
            &\qquad\quad\lesssim \frac{N^{k/2-1}\rho}{N\eta}\Bigl(1 + \frac{\psi^\av_{j-1}}{\sqrt{N\eta\rho}}\Bigr)
            \Bigl(1+\psi^\av_{k-j+1} \sqrt{ \frac{\rho }{N\eta}}\Bigr).
        \end{split}
    \end{equation}
    For the first term in the second line of~\cref{selfcons} we distinguish the cases \(k=2\) and \(k\ge 3\). In the former we write 
    \begin{equation}\label{GAA}
        m_1\braket{G_2A_2A_1}=m_1\braket{G_2}\braket{A_2A_1} + m_1\braket{G_2(A_2A_1)^\circ}=m_1\braket{A_1A_2}\Bigl(m_2+\landauOprec*{\frac{\psi^\av_0}{N\eta}}\Bigr) + \landauOprec*{\frac{\psi^\av_1\rho^{1/2}}{(N\eta)^{1/2}}},
    \end{equation}
    where we used~\cref{trace prod lemma} to estimate 
    \begin{equation}
        \braket{\abs{(A_2A_1)^\circ}^2}^{1/2}=\Bigl(\braket{\abs{A_2A_1}^2}-\abs{\braket{A_2A_1}}^2\Bigr)^{1/2}\le N^{1/2}
    \end{equation}
    
    In case  \(k\ge 3\)  we estimate 
    \begin{equation}\label{2k}
        \begin{split}
            m_1\braket{G_2\cdots G_kA_kA_1} &=  m_1 \braket{G_2\cdots G_k(A_kA_1)^\circ} +
            m_1 \braket{G_2\cdots G_k}\braket{A_kA_1} \\
            &= m_1\cdots m_k \braket{A_2\cdots A_{k-1}(A_kA_1)^\circ} \\
            &\quad+ \landauOprec*{N^{k/2-1}\sqrt{\frac{\rho}{N\eta}}\Bigl[\psi^\av_{k-1}+\sqrt{\frac{\rho}{N\eta}}\Bigl(1+\frac{\psi^\av_{k-2}}{\sqrt{N\eta\rho}}\Bigr)\Bigr]}.
        \end{split}
    \end{equation}
    Note that the leading deterministic term 
    of \(\braket{G_2\cdots G_k}\) was  simply estimated  as
    \begin{equation}\label{leading est}
        \abs*{ m[z_2, z_k] m_3\cdots m_{k-1} \braket*{\prod_{i=2}^{k-1} A_i }  } \lesssim \frac{\rho}{\eta} N^{k/2-2}.
    \end{equation}
    From~\eqref{2k} we write
    \(\braket{A_2\cdots A_{k-1}(A_kA_1)^\circ}=\braket{A_1\cdots A_{k}}-\braket{A_1A_k}\braket{A_2\cdots A_{k-1}}\) where
    the second  term can simply be estimated 
    as \(\abs{\braket{A_1A_k}\braket{A_2\cdots A_{k-1}}}\le N^{k/2-2}\), due to~\cref{trace prod lemma},
    and included in the error term. Collecting all other error terms from~\eqref{GG} and \eqref{2k} 
    and recalling $\psi_j^{\av/\iso}\ge 1\gtrsim \sqrt{\rho/(N\eta)}$ for all $j$, we obtain~\eqref{Ek} with the definition of $\cE_k$ from~\eqref{def:Ek}.  
\end{proof}

Lemma~\ref{underline repl lemma} reduces understanding the local law to the underline term  in~\eqref{selfcons} since  $\cE_k^\av$  will  be treated as  an error term.
For the underline term we use a cumulant expansion when calculating the high moment $\E \abs{ \braket{(GA)^k-m^k A^k}}^p$ for any fixed integer $p$. Here we will make again a notational simplification ignoring different indices
in $G$, $A$ and $m$, and, in particular we may write 
\begin{equation}\label{GAM}
    \abs*{\braket*{\prod_{i=1}^k (G_iA_i)- \prod_{i=1}^k m_i A_i}}^p= \braket{(GA)^k-m^kA^k}^p
\end{equation}
by choosing \(G=G(\ov{z_i})\) for half of the factors.

We set $\partial_{ab}:= \partial/\partial w_{ab}$ as the derivative with respect to the \((a,b)\)-entry of $W$, i.e. 
we consider $w_{ab}$ and $w_{ba}$
as independent variables  in the following cumulant expansion (such expansion 
was first used in the random matrix context in~\cite{MR1411619} and later
 revived in~\cite{MR3678478, MR3405746})
 \[ \E w_{ab} f(W) = \sum_{j=1}^\infty \frac{1}{j!N^{(j+1)/2}} \begin{cases} \kappa^\mathrm{od}_{j+1}\E (\partial_{ab}+\partial_{ba})^{j} f(W),&a\ne b, \\ \kappa^\mathrm{d}_{j+1} \E\partial_{aa}^j f(W),&a=b. \end{cases}  \] 
 Technically we use a truncated version of the expansion above, see e.g.~\cite{MR3941370,MR3678478,MR3941370}.
We thus compute\footnote{The truncation error of the cumulant expansion after \(R=(3+4k)p\) terms can be estimated trivially by the single-\(G\) local law for resolvent entries, and by norm for entries of \(GAG\cdots\) resolvent chains.}
\begin{equation}\label{WGA cum exp}
    \begin{split}
        &\E \braket{\un{W(GA)^k}} \braket{(GA)^k-m^k A^k}^{p-1} \\
        &= \frac{1}{N}\E \sum_{ab} \frac{[(GA)^k]_{ba}}{N} (\partial_{ab}+\partial_{ba}) \braket{(GA)^k-m^k A^k}^{p-1} + \E \sum_{ab} \frac{\partial_{ab} [((GA)^k)_{ba}]}{N^2} \braket{(GA)^k-m^k A^k}^{p-1} \\
        &\quad + \sum_{j=2}^R\frac{\kappa^\mathrm{d}_{j+1}}{j!N^{(j+3)/2}} \E \sum_a  \partial_{aa}^j \Bigl([(GA)^k]_{aa} \braket{(GA)^k-m^k A^k}^{p-1}\Bigr) \\
        &\quad + \sum_{j=2}^R\frac{\kappa^\mathrm{od}_{j+1}}{j!N^{(j+3)/2}}\E \sum_{a\ne b} (\partial_{ab}+\partial_{ba})^j \Bigl([(GA)^k]_{ba} \braket{(GA)^k-m^k A^k}^{p-1}\Bigr) + \landauOprec{\bigl(N^{k/2-3/2}\bigr)^p}
    \end{split}
\end{equation}
recalling~\cref{Wigner def} for the diagonal and off-diagonal cumulants. The  summation runs over all indices $a,b\in [N]$.
The second cumulant calculation in~\eqref{WGA cum exp} used the fact that by definition of the underline renormalisation the \(\partial_{ba}\)-derivative in the first line may not act on its own \((GA)^k\). 

For the first term of~\cref{WGA cum exp} we use \(\partial_{ab}\braket{(GA)^k}=-kN^{-1}((GA)^kG)_{ba}\) due to~\cref{resolvent der} with \(\wt W=\Delta^{ab}\) so that using \(G^t=G\) we can perform the summation and obtain 
\begin{equation}\label{firstline}
    \begin{split}
        &\abs*{\frac{1}{N}\sum_{ab} \frac{[(GA)^k]_{ba}}{N} (\partial_{ab}+\partial_{ba}) \braket{(GA)^k-m^k A^k}^{p-1}} \\
        &\qquad\lesssim \abs*{\frac{\braket{(GA)^{2k} G}}{N^2}} \abs{\braket{(GA)^k -m^kA^k}}^{p-2}\prec \Bigl(N^{k/2-1}\sqrt{\frac{\rho}{N\eta}}\Bigr)^2\Bigl(1+\frac{\psi^\av_{2k}}{\sqrt{N\eta\rho}}\Bigr)\abs{\braket{(GA)^k -m^kA^k}}^{p-2}
    \end{split}
\end{equation}
from~\cref{lemma Psi G prod} with estimating the deterministic leading term
of $ \braket{(GA)^{2k} G}$ by $\abs{ m^{(2)} m^{2k-1} \braket{A^{2k}}}\le N^{k-1}\rho/\eta$
as in~\eqref{leading est}. The first prefactor in the right hand side of~\eqref{firstline} is already written
as the square of the target size $N^{k/2-1}\sqrt{\rho/(N\eta)}$ for $\braket{(GA)^k -m^kA^k}$,
see~\eqref{loc1}.

For the second term of~\cref{WGA cum exp} we estimate 
\[
    \begin{split}
        \sum_{ab} \frac{\partial_{ab} [((GA)^k)_{ba}]}{N^2} &= - \frac{1}{N^2}\sum_{j=0}^{k-1} \sum_{ab} ((GA)^jG)_{ba} ((GA)^{k-j})_{ba} = - \frac{1}{N}\sum_{j=0}^{k-1}  \braket{(GA)^jG (A^t G)^{k-j}}\\
        &= \landauOprec*{N^{k/2-1}\sqrt{\frac{\rho}{N\eta}} \Bigl(\sqrt{\frac{\rho}{N\eta}}+\frac{\sqrt{\rho}\psi^\av_{k}}{N\eta}\Bigr) },
    \end{split}
\]
recalling that \(G=G^t\) since \(W\) is real symmetric\footnote{We recall that we present the proof for the slightly more involved real symmetric case. In the complex Hermitian case the second term on the right hand side of~\cref{WGA cum exp} would not be present.}.

For the second line of~\cref{WGA cum exp} we define the set of multi-indices $\bm l = (l_1, l_2, \ldots, l_n)$
with arbitrary length $n$, denoted by \(\abs{\bm l}:=n\), and total size $k=\sum_i l_i$ as 
\begin{equation}
    \cI_k^\mathrm{d} := \set*{\bm l \in \N_0^{n}\given n\le R,\sum_i l_i=k}, \quad R:= (3+4k)p. 
\end{equation}
Note that the set  $\cI_k^\mathrm{d}$ is a finite set with cardinality depending only on $k,p$.
We distribute the derivatives according to the product rule to estimate 
\begin{equation}\label{Xid}
    \begin{split}
        &\abs*{\sum_{j\ge 2}\frac{\kappa^\mathrm{d}_{j+1}}{j!N^{(j+3)/2}}\sum_a  \partial_{aa}^j \Bigl([(GA)^k]_{aa} \braket{(GA)^k-m^k A^k}^{p-1}\Bigr)} \\
        &\qquad\qquad\qquad\quad \le \sum_{\substack{\bm l\in \cI_k^\mathrm{d}, J\subset\cI_k^\mathrm{d}\\l_{\abs{\bm l}}\ge1, \abs{\bm l}+\sum J\ge 3}} \Xi_k^\mathrm{d}(\bm l,J) \abs{\braket{(GA)^k-m^k A^k}}^{p-1-\abs{J}},
    \end{split}
\end{equation}
where for the multi-set \(J\) we define \(\sum J:=\sum_{\bm j\in J}\abs{\bm j}\) and set
\begin{equation}\label{Iod}
    \Xi_k^\mathrm{d} := \frac{N^{-\frac{\abs{\bm l}+\sum J}{2}}}{N^{1+\abs{J}}} \abs*{\sum_{a} [(GA)^{l_1}G]_{aa}\cdots [(GA)^{l_{\abs{\vl}-1}}G]_{aa} [(GA)^{l_{\abs{\vl}}}]_{aa} \prod_{\bm j\in J}[(GA)^{j_1}G]_{aa}\cdots [(GA)^{j_{\abs{\bm j}}}G]_{aa}}.
\end{equation}
Here for the multi-set $J\subset \cI_k^\mathrm{d}$ we defined its cardinality by $|J|$ and we 
set \(\sum J:=\sum_{\bm j\in J}\abs{\bm j}\).  
Along the product rule, the multi-index $\bm l$ encodes how the first factor  $([(GA)^k]_{aa}$  in~\eqref{Xid} is differentiated, while
each element $\bm j\in J$  is a multi-index that encodes how another factor $\braket{(GA)^k-m^k A^k}$ is differentiated.
Note that $|J|$ is the number of such factors affected by derivatives, the remaining $p-1-|J|$ factors are untouched.

For the third line of~\cref{WGA cum exp} we similarly define the appropriate index set that is needed to
encode the product rule\footnote{In the definition of $\cI_k^\mathrm{od}$ the indices \(ab,ba,aa,bb\) should be understood symbolically, merely indicating the diagonal or off-diagonal character of the term. However, in the formula~\cref{Xi od def} below the concrete summation indices \(a,b\) are substituted for the symbolic expressions. Alternatively, we could have avoided this slight abuse of notation by defining \(\alpha_i\in\set{(1,1),(1,2),(2,1),(2,2)}\), sum over \(a_1,a_2=1,\ldots,N\) in~\cref{Xi od def} and substitute \(a_{(\alpha_i)_1},a_{(\alpha_i)_2}\) for \(\alpha_i\), however this would be an excessive  pedantry.}
\begin{equation}\label{I od def}
    \cI_k^\mathrm{od} := \set*{(\bm l,\bm\alpha) \in \N_0^{\abs{\bm l}}\times \set{ab,ba,aa,bb}^{\abs{\vl}} \given \abs{\bm l}\le R,\sum_i l_i=k, \abs{\set{i\given\alpha_i=aa}}=\abs{\set{i\given\alpha_i=bb}}}.
\end{equation}
Note that in addition to the multi-index $\bm l$ encoding the distribution of the derivatives after the Leibniz rule similarly to the previous diagonal case, the second element $\bm\alpha$ of the new type of indices also keeps track of whether  after the differentiations the corresponding factor is evaluated at $ab, ba, aa$ or $bb$. While a single $\partial_{ab}$ or $\partial_{ba}$ acting on $\braket{(GA)^k-m^k A^k}$ results in an off-diagonal term of the  form $[(GA)^kG]_{ab}$ or $[(GA)^kG]_{ba}$, a second derivative also produces diagonal terms. The derivative action on the first factor $[(GA)^k]_{ba} $ in the third line of~\cref{WGA cum exp} produces diagonal factors already after one derivative. The restriction in~\eqref{Iod} that the number 
of $aa$ and $bb$-type diagonal elements must coincide comes from a simple  counting of diagonal  indices along derivatives: when an additional $\partial_{ab}$ hits an off-diagonal term, then 
either one $aa$ and one $bb$ diagonal is created or none. Similarly, when an additional \(\partial_{ab}\) hits a diagonal \(aa\) term, then one diagonal \(aa\) remains, along with a new off-diagonal \(ab\). In any case the difference of the $aa$ and $bb$ diagonals is unchanged.

Armed with this notation, similarly to~\eqref{Xid} we estimate 
\begin{equation}\label{Xiod}
    \begin{split}
        &\abs*{\sum_{j\ge 2}\frac{\kappa^\mathrm{od}_{j+1}}{j!N^{(j+3)/2}}\sum_{a,b} (\partial_{ab}+\partial_{ba})^j \Bigl([(GA)^k]_{ba} \braket{(GA)^k-m^k A^k}^{p-1}\Bigr)} \\
        &\quad \le \sum_{\substack{(\bm l,\bm\alpha)\in \cI_k^\mathrm{od}, J\subset\cI_k^\mathrm{od}\\ l_{\abs{\bm l}}\ge1,\abs{\bm l}+\sum J\ge 3}} \Xi_k^\mathrm{od}((\bm l,\bm\alpha),J) \abs{\braket{(GA)^k-m^k A^k}}^{p-1-\abs{J}},
    \end{split}
\end{equation}
where for the multi-set \(J\subset\cI_k^\mathrm{od}\) we define \(\sum J:=\sum_{(\bm j,\bm\beta)\in J}\abs{\bm j}\) and set
\begin{equation}\label{Xi od def}
    \Xi_k^\mathrm{od} := \frac{N^{-\frac{\abs{\bm l}+\sum J}{2}}}{N^{1+\abs{J}}} \abs*{\sum_{ab} [(GA)^{l_1}G]_{\alpha_1}\cdots [(GA)^{l_{\abs{\vl}}}]_{\alpha_{\abs{\vl}}} \prod_{(\bm j,\bm\beta)\in J}[(GA)^{j_1}G]_{\beta_1}\cdots [(GA)^{j_{\abs{\bm j}}}G]_{\beta_{\abs{\vj}}}}.
\end{equation}
Note that~\eqref{Xiod} is an overestimate: not all  indices $(\bm j,\bm\beta)$ indicated in~\eqref{Xi od def} can actually occur
after the Leibniz rule. 

\begin{lemma}\label{Xi lemma}
    For any \(k\ge 1\) it holds that 
    \begin{equation}\label{Xi k claim}
        \Xi_k^\mathrm{d} + \Xi_k^\mathrm{od} \prec \Biggl(N^{k/2-1}\sqrt{\frac{\rho}{N\eta}}\Bigl(\Phi_k+(\psi_k^\iso)^{2/3}\Phi_{k-1}^{1/3} + \sum_{j=1}^{k-1}\sqrt{(\Phi_{k-1}+\psi_{k}^\iso)\psi_{j}^\iso \Omega_{k-j}}\Bigr)\Biggr)^{1+\abs{J}}
    \end{equation}
\end{lemma}
By combining~\cref{underline repl lemma} and~\cref{WGA cum exp,firstline,Xid,Xiod} with~\cref{Xi lemma} 
and using a simple H\"older inequality, we obtain, for any fixed $\xi>0$,
that  
\begin{equation}\label{finb}
    \begin{split}
        &\Bigl(\E \abs{\braket{(GA)^k-m^k A^k}}^p\Bigr)^{1/p} \\
        &\lesssim N^\xi N^{k/2-1}\sqrt{\frac{\rho}{N\eta}} \biggl(\Phi_k + \Bigl(\frac{\psi_{2k}^\av}{\sqrt{N\eta\rho}}\Bigr)^{1/2}+\psi^\av_{k-1}+\frac{\psi^\av_k}{\sqrt{N\eta}}+(\psi_k^\iso)^{2/3}\Phi_{k-1}^{1/3} \\
        &\qquad\qquad\qquad\qquad\qquad+ \sum_{j=1}^{k-1}\sqrt{(\Phi_{k-1}+\psi_{k}^\iso)\psi_{j}^\iso \Omega_{k-j}}+\frac{1}{N\eta}\sum_{j=1}^{k-1}\psi^\av_{j}\Bigl(1+\psi^\av_{k-j}\sqrt{\frac{\rho}{N\eta}} \Bigr)\biggr),
    \end{split}
\end{equation}
where we used the \(\Xi_k^\mathrm{d}\) term to add back the \(a=b\) part of the summation in~\cref{Xiod} compared to~\cref{WGA cum exp}. By taking \(p\) large enough, \(\xi\) arbitrarily small, 
and using the definition of $\prec$  and the fact that the bound~\eqref{finb} holds uniformly in the spectral parameters and the deterministic matrices, we conclude the proof of~\cref{Psi master av}. 

\begin{proof}[Proof of~\cref{Xi lemma}]
    The proof repeatedly uses~\cref{a priori k iso} in the form
    \begin{align}\label{eq naive iso GAG} 
        ((GA)^kG)_{ab}&\prec N^{k/2-1/2}\Bigl(\norm{A\bm e_a}\wedge\norm{A\bm e_b}+\psi_k^\iso\sqrt{\frac{\rho}{\eta}}\Bigr)\lesssim N^{k/2}\Bigl(1+\psi_k^\iso\sqrt{\frac{\rho}{N\eta}}\Bigr),
        \\ \label{eq naive iso GA}  
        ((GA)^k)_{ab}&\prec N^{k/2-1/2}\norm{A\bm e_b}\Bigl(1+\psi_{k-1}^\iso\sqrt{\frac{\rho}{N\eta}}\Bigr)\lesssim N^{k/2}\Bigl(1+\psi_{k-1}^\iso\sqrt{\frac{\rho}{N\eta}}\Bigr)
    \end{align}
    with $\bm e_b$ being the $b$-th coordinate vector, where we
    estimated the deterministic leading term  $m^k(A^k)_{ab}$ by
    \(\abs{(A^k)_{ab}} \le \norm{A}^{k-1} \norm{A\bm e_b} \le N^{(k-1)/2} \norm{A\bm e_b}\) using~\eqref{Bnorm}. Recalling the normalization $\braket{\abs{A}^2}=1$,  the best available bound on \(\norm{A\bm e_b}\) is \(\norm{A\bm e_b}\le N^{1/2}\), however this can be substantially improved under a summation  over the index $b$: 
    \begin{equation}\label{eq A sum}
        \sum_b \norm{A\bm e_b}^2 = N \braket{\abs{A}^2} \le N, \quad \sum_{b}\norm{A\bm e_b}\le 
        \sqrt{N}\sqrt{\sum_b \norm{A\bm e_b}^2} \le N.
    \end{equation}

    Using~\cref{eq naive iso GAG,eq naive iso GA} for each entry of~\eqref{Iod} and~\eqref{Xi od def}, 
    we obtain the following \emph{naive (or a priori) estimates} on \(\Xi_k^{\mathrm{d/od}}\)   
    \begin{equation}\label{Xi k naive}
        \Xi_k^{\mathrm{d/od}} \prec \bigl(N^{k/2-1}\frac{\Omega_k}{\sqrt{N}} \bigr)^{1+\abs{J}} 
        N^{1+\bm1(\mathrm{od})+(\abs{J}-\abs{\vl}-\sum J)/2}
    \end{equation}
    where we recall the definition of $\Omega_k$ from~\eqref{Phi k def}.
    Using \(\Omega_k/\sqrt{N}\lesssim \Phi_k \sqrt{\rho/(N\eta)}\) due to \(1\lesssim \rho/\eta\) the claim~\eqref{Xi k claim} 
    follows trivially from~\eqref{Xi k naive}  for \(\Xi_k^\mathrm{d}\) and \(\Xi_k^\mathrm{od}\)
    whenever \(\abs{\vl}+\sum J\ge 2+\abs{J}\) or \(\abs{\vl}+\sum J\ge 4+\abs{J}\), respectively, 
    i.e.\ when the exponent of $N$ in~\eqref{Xi k naive} is non-positive. 
    
    In the rest of the proof we consider the remaining diagonal~\cref{cased 1 gain} and 
    off-diagonal cases~\cref{caseod 1 gain}--\cref{caseod 3 gains} that we will define below. The cases are organised according to
    the quantity \(\abs{\vl}+\sum J -\abs{J}\) which captures by how many factors of \(N^{1/2}\) the naive estimate~\cref{Xi k naive} 
    exceeds the target~\cref{Xi k claim} when all $\Phi$'s and $\psi$'s are set to be order one. 
    Within case~\cref{caseod 1 gain} we further differentiate whether an off-diagonal index pair
    \(ab\) or \(ba\) appears at least once in the tuple \(\bm\alpha\) or in one of the tuples \(\bm\beta\). Within case~\cref{caseod 2 gains} we distinguish according to the length of \(\abs{\vl}\) and \(\abs{J}\) as follows:
    \begin{enumerate}[label=\textbf{D\(\mathbf{\arabic*}\)},start=1]
        \item\label{cased 1 gain} \(\abs{\vl}+\sum J=\abs{J}+1\)
    \end{enumerate}
    \begin{enumerate}[label=\textbf{O\(\mathbf{\arabic*}\)},start=1]
        \item\label{caseod 1 gain} \(\abs{\vl}+\sum J=\abs{J}+3\)
        \begin{enumerate}[label=\textbf{\theenumi\alph*}]
            \item\label{caseod 1 gain od}\(ab\vee ba \in\bm \alpha\cup\bigcup_{(\vj,\bm\beta)\in J} \bm\beta\)
            \item\label{caseod 1 gain dd}\(J\in\set{\set{(\vj,(aa,bb))},\set{(\vj,(bb,aa))}}\)  
            and \(\bm\alpha\in\set{(aa,bb),(bb,aa)}\), i.e.\ \(\sum J=\abs{\vl}=2\) and \(\abs{J}=1\)
        \end{enumerate}
        \item\label{caseod 2 gains} \(\abs{\vl}+\sum J=\abs{J}+2\)
        \begin{enumerate}[label=\textbf{\theenumi\alph*}]
            \item\label{l1} \(\abs{\vl}=1\),
            \item\label{l2 J2} \(\abs{\vl}=2\), \(\abs{J}\ge 2\),
            \item\label{l2 J1 l11} \(\abs{\vl}=2\), \(\abs{J}=1\), \(l_1\ge 1\),
            \item\label{l2 J1 l10} \(\abs{\vl}=2\), \(\abs{J}=1\), \(l_1= 0\).
        \end{enumerate}
        \item\label{caseod 3 gains} \(\abs{\vl}+\sum J=\abs{J}+1\)
    \end{enumerate} 
    The list of four cases  above is exhaustive since \(\sum J+\abs{\vl}\ge \abs{J}+1\) by definition, and the subcases of~\cref{caseod 2 gains} are obviously exhaustive. Within case~\cref{caseod 1 gain} either some off-diagonal element appears in \(\bm\alpha\) or some \(\bm\beta\)
    (hence we are in case~\cref{caseod 1 gain od}), or the number of elements in
    \(\bm\alpha\) and all \(\bm\beta\) is even, c.f.\ the constraint on the number of diagonal elements in~\cref{I od def}.
    The latter case is only possible if \(\abs{J}=1\), \(\abs{\vl}=\sum J=2\) which is case~\cref{caseod 1 gain dd} 
    (note that \(\abs{\vl}\ge 2\) implies \(\abs{J}\le 1\), and  \(\abs{J}=0\) is impossible as it
    would imply  \(\abs{\vl}=3\), the number of elements in  \(\bm\alpha\), is odd.

    Now we give the estimates for each case separately. For case~\cref{cased 1 gain},  using the restriction in the summation in~\eqref{Xiod} to get
    \(3\le\abs{\vl}+\sum J=1+\abs{J}\), 
    we estimate
    \begin{equation}\label{Xi d est}
        \begin{split}
            \Xi_k^\mathrm{d}&=N^{-3(1+\abs{J})/2} \abs*{\sum_a [(GA)^k]_{aa} [(GA)^k G]_{aa}^{\abs{J}} } \\
            &\prec \frac{(N^{k/2-1})^{\abs{J}+1}}{N^{\abs{J}/2+1}}\Omega_k^{\abs{J}-1}\Omega_{k-1}\sum_a \norm{A\bm e_a}\Bigl(\frac{\norm{A\bm e_a}}{N^{1/2}}+ \sqrt{\frac{\rho}{N\eta}}\psi_k^\iso\Bigr)\\
            &\lesssim \Bigl(N^{k/2-1}\sqrt{\frac{\rho}{N\eta}}\Bigr)^{1+\abs{J}} \Phi_k^{\abs{J}-1}\Phi_{k-1} \psi_k^\iso%
        \end{split}
    \end{equation}
    where we used the first inequalities of~\cref{eq naive iso GAG,eq naive iso GA} for the \((GA)^k\) and one of the \((GA)^kG\) factors, and the second inequality of~\cref{eq naive iso GAG} for the remaining factors, and in the last step we used~\cref{eq A sum} and \(\psi_k^\iso\sqrt{\rho/\eta}\gtrsim 1\).  Finally we use Young's inequality \( \Phi_k^{\abs{J}-1}\Phi_{k-1} \psi_k^\iso\le
    \Phi_k^{\abs{J}+1}+ (\Phi_{k-1} \psi_k^\iso)^{(\abs{J}+1)/2}\).  
    This 
    confirms~\cref{Xi k claim}  in case~\cref{cased 1 gain}. 
    
    For the offdiagonal cases we will use the following so-called \emph{Ward-improvements}:   
    \begin{enumerate}[label=\textbf{I\(\mathbf{\arabic*}\)}]
        \item\label{IWard} Averaging over 
        \(a\) or \(b\) in \(\abs{((GA)^kG)_{ab}}\) gains a factor of \(\sqrt{\rho/(N\eta)}\) compared to~\cref{eq naive iso GAG},
        \item\label{IWardA}  Averaging over \(a\) in \(\abs{((GA)^k)_{ab}}\) gains a factor of \(\sqrt{\rho/(N\eta)}\) compared to~\cref{eq naive iso GA},
    \end{enumerate}
    at the expense of replacing a factor of \((1+\psi_k^\iso\sqrt{\rho/(N\eta)})\) in the definition of $\Omega_k$ 
    by a factor of \((1+\psi^\iso_{2k}/\sqrt{N\eta\rho})^{1/2}\). 
    These latter replacements necessitate 
    changing $\Omega_k$ to the larger $\Phi_k$ as a main control parameter in the estimates after Ward improvements. Indeed,~\cref{IWard,IWardA} follow directly from~\eqref{multicalG} of~\cref{lemma Psi G prod} and \(\abs{m^{(2)}}\lesssim \rho/\eta\), more precisely
    \begin{equation}\label{ward}
        \begin{split}
            \frac{1}{N}  \sum_a\abs{[(GA)^kG]_{ab}} &\le \frac{\sqrt{[(G^* A)^{k} G^\ast G (AG)^{k} ]_{bb}}}{\sqrt{N}}
            \prec  N^{k/2}\sqrt{\frac{\rho}{N\eta}} \Bigl(1+\psi^\iso_{2k}\sqrt{\frac{1}{N\eta \rho}}\Bigr)^{1/2}\\
            \frac{1}{N}  \sum_a\abs{[(GA)^k]_{ab}} &\le \frac{\sqrt{[(AG^\ast)^{k} (GA)^{k} ]_{bb}}}{\sqrt{N}}
            \prec  N^{k/2-1/2}\norm{A\bm e_b}\sqrt{\frac{\rho}{N\eta}} \Bigl(1+\psi^\iso_{2(k-1)}\sqrt{\frac{1}{N\eta \rho}}\Bigr)^{1/2}\\
            \frac{1}{N}  \sum_a\abs{[(GA)^k]_{ab}}^2 &=\frac{ [(AG^\ast)^{k} (GA)^{k} ]_{bb}}{N}
            \prec  N^{k-1}\norm{A\bm e_b}^2 \frac{\rho}{N\eta} \Bigl(1+\psi^\iso_{2(k-1)}\sqrt{\frac{1}{N\eta \rho}}\Bigr),
        \end{split}
    \end{equation}
    where the first step in each case  followed from a Schwarz inequality and summing up the indices explicitly.  
    This improvement is essentially equivalent to using the \emph{Ward-identity} $GG^*= \Im G/\eta$ in~\eqref{ward}.
    
    Now we collect these gains over the naive bound given in~\eqref{Xi k naive} for each case.  
    Note that whenever a factor $\sqrt{\rho/(N\eta)}$ is gained,  the additional $1/\sqrt{N}$ is freed up along the second inequality
    in~\eqref{Xi k naive} which can be used to compensate the positive $N$-powers. 
    
    For case~\cref{caseod 3 gains} we have \(\abs{J}\ge 2\) and estimate all but the first two \((\bm j,\bm\beta)\) 
    factors in~\cref{Xi od def} trivially using the last inequality in~\cref{eq naive iso GAG} to obtain 
    \begin{equation}
        \Xi_k^\mathrm{od}\prec N^{-3(1+\abs{J})/2} (N^{k/2}\Omega_k)^{\abs{J}-2} \sum_{ab}\abs*{[(GA)^k]_{ba}}\abs*{[(GA)^kG]_{ab}}\abs*{[(GA)^kG]_{ab}}.
    \end{equation}
    For the last two factors we use first inequality in~\cref{eq naive iso GAG} 
    and then estimate  as
    \begin{equation}\label{o3}
        \begin{split}
            &\sum_{ab}\abs*{[(GA)^k]_{ba}}\abs*{[(GA)^kG]_{ab}}\abs*{[(GA)^kG]_{ab}}\\
            &\quad\lesssim N^{k-1}\sum_{ab}\abs*{[(GA)^k]_{ba}}\Bigl(\norm{A \bm e_a}\norm{A \bm e_b}
            +(\psi_k^\iso)^2 \frac{\rho}{\eta}\Bigr)\prec \Bigl(N^{k/2}\sqrt{\frac{\rho}{\eta}}\Bigr)^3  \Phi_{k-1} (\psi_k^\iso)^2,
        \end{split}
    \end{equation}
    where in the second step we performed a Schwarz inequality for the double  \(a, b\) summation
    and used the last bound in~\cref{ward},~\cref{eq A sum} and \(1\lesssim \psi_k^\iso\sqrt{ \rho/\eta}\). Thus, we conclude 
    \begin{equation}
        \Xi_k^\mathrm{od}\prec\Bigl(N^{k/2-1}\sqrt{\frac{\rho}{N\eta}}\Bigr)^{\abs{J}+1}\Phi_k^{\abs{J}-2}\Phi_{k-1}(\psi_k^\iso)^2. 
    \end{equation}

    In case~\cref{l1} there exists some \(\vj\) with \(\abs{\vj}=2\) (recall that \(\sum J =\abs{J}+1\)). By estimating the remaining \(J\)-terms trivially by~\eqref{eq naive iso GAG}, we obtain 
    \begin{equation}\label{o2a}
        \begin{split}
            \Xi_k^\mathrm{od} &\prec N^{-3(1+\abs{J})/2-1/2} (N^{k/2}\Omega_k)^{\abs{J}-1}\sum_{ab}\abs{[(GA)^k]_{ab}} 
            \abs{[(GA)^{j_1}G]_{\beta_1}}\abs{[(GA)^{j_2}G]_{\beta_2}}\\
            &\prec N^{-3(1+\abs{J})/2-1/2} (N^{k/2}\Omega_k)^{\abs{J}-1}N^{k/2-1/2}\Omega_{j_2}
            \sum_{ab} \abs{[(GA)^k]_{ab}}\Bigl(\norm{A\bm e_a}+\norm{A\bm e_b}+\psi_{j_1}^\iso\sqrt{\frac{\rho}{\eta}}\Bigr) \\
            &\lesssim  \Bigl(N^{k/2-1}\frac{\Omega_k}{\sqrt{N}}\Bigr)^{\abs{J}-1} 
            \Bigl(N^{k/2-1}\sqrt{\frac{\rho}{N\eta}}\Bigr)^2\Phi_{k-1}\psi_{j_1}^\iso \Omega_{j_2}
        \end{split}
    \end{equation}
    for some \(j_1+j_2=k\) and double indices $\beta_1,\beta_2 \in \{ aa, bb, ab, ba\}$.
    Here in the second step we assumed without loss of generality \(j_1\ge 1\) (the case \(j_2\ge 1\) being completely analogous) and used the first inequality in~\eqref{eq naive iso GAG} for $\abs{[(GA)^{j_1}G]_{\beta_1}}$ and the second inequality in \eqref{eq naive iso GAG} for $\abs{[(GA)^{j_2}G]_{\beta_2}}$. Finally, in the last step we performed an \(a,b\)-Schwarz inequality,
    used the last bound in~\cref{ward} and~\cref{eq A sum}.
    
    In case~\cref{l2 J2} we have \(\abs{\vj}=1\) for all \(\vj\) since \(\sum J+\abs{\vl}=\abs{J}+2\) 
    implies $\sum J = \abs{J}$, and we estimate all but two \(J\)-factors trivially by the last inequality in~\eqref{eq naive iso GAG}, the other two \(J\)-factors (which are necessarily offdiagonal) by the first inequality in~\cref{eq naive iso GAG}, the \(l_1\)-factor by the last inequality in~\cref{eq naive iso GAG} and the \(l_2\) factor by the first inequality in~\cref{eq naive iso GA} (note that \(l_2\ge 1\)) to obtain 
    \begin{equation}
        \begin{split}
            \Xi_k^\mathrm{od} &\prec N^{-3(1+\abs{J})/2-1/2} (N^{k/2}\Omega_k)^{\abs{J}-2}\sum_{ab}\abs{[(GA)^{l_1}G]_{\alpha_1}}\abs{[(GA)^{l_2}]_{\alpha_2}} \abs{[(GA)^{k}G]_{ab}}^2\\
            &\prec N^{-3(1+\abs{J})/2-1/2} (N^{k/2}\Omega_k)^{\abs{J}-2}N^{3k/2-3/2}\Omega_{k-1}\sum_{ab}(\norm{A\bm e_a}+\norm{A \bm e_b}) \Bigl(\norm{A \bm e_a}\norm{A \bm e_b} + \frac{\rho}{\eta}(\psi_k^\iso)^2 \Bigr)\\
            &\lesssim  \Bigl(N^{k/2-1}\frac{\Omega_k}{\sqrt{N}}\Bigr)^{\abs{J}-2}N^{k/2-3/2}\Bigl(N^{k/2-1}\sqrt{\frac{\rho}{N\eta}}\Bigr)^2\Omega_{k-1}(\psi_k^\iso)^2,
        \end{split}
    \end{equation}
    where the last step used~\cref{eq A sum} and \(\psi_k^\iso \sqrt{\rho/\eta}\gtrsim 1\).  
    
    In case~\cref{l2 J1 l11} we use the first inequalities of~\cref{eq naive iso GAG,eq naive iso GA} for the \(l_1,l_2\)-terms (since \(l_1,l_2\ge 1\)) and the first inequality of~\cref{eq naive iso GAG} for the \((GA)^kG\) factor to obtain 
    \begin{equation}
        \begin{split}
            \Xi_k^\mathrm{od} &\lesssim N^{-7/2} \sum_{ab} \abs{[(GA)^{l_1}G]_{\alpha_1}}\abs{[(GA)^{l_2}]_{\alpha_2}} \abs{[(GA)^{k}G]_{ab}} \\
            &\prec N^{k-5}\Omega_{l_2-1}\sum_{ab} \Bigl(\norm{A\bm e_b}+\norm{A\bm e_a}+\sqrt{\frac{\rho}{\eta}}\psi_{l_1}^\iso\Bigr)(\norm{A\bm e_a}+\norm{A\bm e_b}) \Bigl(\norm{A \bm e_a}\wedge \norm{A \bm e_b} + \sqrt{\frac{\rho}{\eta}}\psi_{k}^\iso \Bigr)\\
            &\lesssim \Bigl(N^{k/2-1}\sqrt{\frac{\rho}{N\eta}}\Bigr)^2 \Omega_{l_2-1} \psi_{l_1}^\iso\psi_k^\iso
        \end{split}
    \end{equation}
    by~\cref{eq A sum}. 
    
    In case~\cref{l2 J1 l10} we write the single-\( G\) diagonal as \(G_{aa}=m + \landauOprec*{\sqrt{\rho/(N\eta)}}\) and 
    use \emph{isotropic resummation} for the leading $m$ term 
    into the \(\bm 1=(1,1,\ldots)\) vector of norm \(\norm{\bm 1}=\sqrt{N}\), i.e. 
    \[\sum_{a} G_{aa} [(GA)^k G]_{ab} = m[(GA)^k G]_{\bm 1b} + \landauOprec*{\sqrt{\frac{\rho}{N\eta}}}\sum_{a} \abs*{[(GA)^k G]_{ab}},\]
     and estimate
    \begin{equation}
        \begin{split}
            \Xi_k^\mathrm{od} &\lesssim N^{-7/2} \abs*{\sum_{ab} G_{aa} [(GA)^{k}]_{bb} [(GA)^{k}G]_{ab}} +  N^{-7/2} \sum_{ab} \abs*{G_{ab} [(GA)^{k}]_{ab} [(GA)^{k}G]_{ab}}\\
            & \prec N^{-7/2} \abs*{\sum_b [(GA)^{k}]_{bb} [(GA)^{k}G]_{\bm 1b}} + N^{-7/2}\sqrt{\frac{\rho}{N\eta}}\sum_{ab} \abs*{[(GA)^{k}]_{bb} [(GA)^{k}G]_{ab}} \\
            &\prec \sqrt{\frac{\rho}{\eta}}N^{k-4}\Omega_{k-1}\sum_{b}\norm{A\bm e_b} \Bigl(\norm{A\bm e_b}+\sqrt{\frac{\rho}{\eta}}\psi_k^\iso\Bigr)\lesssim (N^{k/2-1}\sqrt{\frac{\rho}{N\eta}}\Bigr)^2\Omega_{k-1}\psi_k^\iso
        \end{split}
    \end{equation}
    using the first inequalities of~\cref{eq naive iso GAG,eq naive iso GA}.

    In case~\cref{caseod 1 gain od} we use either~\cref{IWard} or~\cref{IWardA} 
    depending on whether the off-diagonal matrix 
    is of the form \((GA)^lG\) or \((GA)^l\) to gain one factor of \(\sqrt{\rho/(N\eta)}\) in either case and conclude~\cref{Xi k claim}. 

    Finally we consider case~\cref{caseod 1 gain dd} where there is no off-diagonal element 
    to perform Ward-improvement, but for which, using~\cref{eq A sum}, we estimate
    \begin{equation} 
        \begin{split}
            & N^{-4} \abs*{\sum_{ab} [(GA)^{k_1}G]_{aa}[(GA)^{k_2}]_{bb} [(GA)^{k_3}G]_{aa}[(GA)^{k_4}G]_{bb}} \\
            & \quad\prec N^{k-5}\Omega_{k-1}\Omega_{k_3}\sum_{ab}\norm{A\bm e_b}\Bigl(\norm{A\bm e_b}+\psi_{k_4}^\iso \sqrt{\frac{\rho}{\eta}}\Bigr)\le N^{k-3}\sqrt{\frac{\rho}{\eta}}\Omega_{k-1}\Bigl(1+\psi_{k_3}^\iso \sqrt{\frac{\rho}{N\eta}}\Bigr)(\psi_{k_4}^\iso+1)\\
            &\quad \lesssim \Bigl(N^{k/2-1}\sqrt{\frac{\rho}{N\eta}}\Bigr)^2\Omega_{k-1}\sum_{j=0}^{k}\psi_j^\iso\Omega_{k-j}
        \end{split}
    \end{equation}
    for any  exponents with \(k_1+k_2=k_3+k_4=k\). Here in case \(k_4>0\) we used the second inequalities of~\cref{eq naive iso GAG,eq naive iso GA} for the \(k_2,k_4\) factors and the first inequality of~\cref{eq naive iso GAG} for the \(k_1,k_3\) factors.
    The case \(k_4=0\) is handled similarly, with the same result, by estimating \([(GA)^{k_3}G]_{aa}\) instead of \([(GA)^{k_4}G]_{bb}\) using the first inequality of~\cref{eq naive iso GAG}. 
\end{proof}

\subsection{Proof of isotropic estimate \eqref{Psi master iso} in~\cref{prop master}}\label{sec iso}
First we state the isotropic version 
of~\cref{underline repl lemma}:
\begin{lemma}\label{underline repl lemma iso}
    For any deterministic unit vectors $\vx, \vy$ and \(k\ge 0\) we have 
    \begin{align}
        \label{Ekiso}
        \braket*{\vx,[(GA)^kG-m^{k+1}A^k]\vy}\Bigl(1+\landauOprec*{\frac{\psi^\av_0}{N\eta}}\Bigr)&=-m\braket{\vx,\un{W(GA)^kG}\vy} + \landauOprec*{\cE_k^\iso},
    \end{align}
    where \(\cE_0^\iso=0\) and for \(k\ge 1\)
    \begin{equation}\label{def:Ekiso}
        \begin{split}
            \cE_k^\iso &:= N^{k/2}\sqrt{\frac{\rho}{N\eta}}\Bigl(\psi^\iso_{k-1} + \frac{1}{N\eta}\sum_{j=1}^{k}\Bigl(\psi_j^\av+\psi_{k-j}^\iso+\sqrt{\frac{\rho}{N\eta}}\psi_j^\av \psi_{k-j}^\iso\Bigr)\Bigr).
        \end{split} 
    \end{equation}
\end{lemma}
\begin{proof}
    From~\cref{G-m un} applied to the first factor $G=G_1$, similarly  to~\eqref{selfcons},  we obtain 
    \begin{equation}\label{selfconsiso}
        \begin{split}
            \Bigl(1+\landauOprec*{\frac{\psi^\av_0}{N\eta}}\Bigr)\braket{\vx,(GA)^k G\vy} &= m\braket{\vx,(AG)^k\vy} - m\braket{\vx,\un{WG}(AG)^k\vy}\\
            &= m^{k+1}\braket{\vx,A^k\vy}-m\braket{\vx,\un{WG(AG)^k}\vy} \\
            &\quad+  m_1\sum_{j=1}^{k}\braket{(GA)^jG}\braket{\vx,(GA)^{k-j}G\vy} \\
            &\quad + \landauOprec*{N^{(k-1)/2}\sqrt{\frac{\rho}{N\eta}} \norm{A\vx }  \psi_{k-1}^\iso}, 
        \end{split}
    \end{equation}
    where we used the definition~\eqref{a priori k iso} for the first term and the definition~\eqref{def:underline}. An estimate analogous to~\eqref{GG} handles the sum and is incorporated in~\eqref{def:Ekiso}. This concludes the proof together with~\cref{lemma Psi G prod} and \(\norm{A\vx} \le \norm{A} \le N^{1/2}\). 
\end{proof}
Exactly as in~\cref{WGA cum exp} we perform a cumulant expansion 
\begin{equation}\label{WGA cum exp iso}
    \begin{split}
        &\E \braket{\vx,\un{W(GA)^kG}\vy} \braket{\vx,[(GA)^kG-m^{k+1} A^k]\vy}^{p-1} \\
        &\quad= \E \sum_{ab} \frac{x_a[(GA)^kG]_{b\vy}}{N} (\partial_{ab}+\partial_{ba}) [(GA)^kG-m^{k+1} A^k]_{\vx\vy}^{p-1} \\
        &\qquad+\E\sum_{ab} \frac{x_a \partial_{ab} [(GA)^kG]_{b\vy}}{N}[(GA)^kG-m^{k+1} A^k]_{\vx\vy}^{p-1}\\
        &\qquad + \sum_{j\ge 2}\frac{\kappa^\mathrm{d}_{j+1}}{j!N^{(j+1)/2}} \E \sum_a  \partial_{aa}^j \Bigl(x_a[(GA)^kG]_{a\vy} [(GA)^kG-m^{k+1} A^k]_{\vx\vy}^{p-1}\Bigr) \\
        &\qquad + \sum_{j\ge 2}\frac{\kappa^\mathrm{od}_{j+1}}{j!N^{(j+1)/2}}\E \sum_{a\ne b} (\partial_{ab}+\partial_{ba})^j \Bigl(x_a[(GA)^kG]_{b\vy} [(GA)^kG-m^{k+1} A^k]_{\vx\vy}^{p-1}\Bigr),
    \end{split}
\end{equation}
recalling~\cref{Wigner def} for the diagonal and off-diagonal cumulants. In fact, the formula~\eqref{WGA cum exp iso} is identical  to~\cref{WGA cum exp}  for $k+1$ instead of $k$
if the last $A=A_{k+1}$ in 
the product $(GA)^{k+1}= G_1 A_1G_2 A_2\ldots G_{k+1} A_{k+1}$ is chosen specifically $A_{k+1}= \vy\vx^*$. 

For the first line of~\cref{WGA cum exp iso}, after performing the derivative, we can also perform the summations
and estimate the resulting isotropic resolvent chains by using the last inequality of~\eqref{eq naive iso GAG} as well as~\cref{lemma Psi G prod}  to obtain 
\begin{equation}\label{firstline iso}
    \begin{split}
        &\sum_{ab} \frac{x_a[(GA)^kG]_{b\vy}}{N} (\partial_{ab}+\partial_{ba}) [(GA)^kG-m^{k+1} A^k]_{\vx\vy}^{p-1}\\
        & =  \sum_{j=0}^{2k} \frac{[(GA)^jG]_{\vx \vx}[(GA)^kG(GA)^{k-j}G]_{\vy\vy} + [(GA)^jG(GA)^kG]_{\vx \vy}[(GA)^{k-j}G]_{\vx\vy} }{N}  [(GA)^kG-m^{k+1} A^k]_{\vx\vy}^{p-2} \\
               &\prec \Bigl(N^{k/2}\sqrt{\frac{\rho}{N\eta}}\Bigr)^2\Bigl(1+\sum_{j=0}^{2k}\frac{\psi_j^\iso}{\sqrt{N\eta\rho}}
        \Bigl(1+\sqrt{\frac{\rho}{N\eta}}\psi_{2k-j}^\iso\Bigr)\Bigr)\abs*{[(GA)^kG-m^{k+1} A^k]_{\vx\vy}}^{p-2}.
    \end{split}
\end{equation}

For the second line of~\cref{WGA cum exp iso} we estimate 
\[
    \begin{split}
        \sum_{ab} \frac{x_a \partial_{ab} [(GA)^kG]_{b\vy}}{N} &= - \sum_{j=0}^{k}\sum_{ab} \frac{x_a [(GA)^jG]_{ba} [(GA)^{k-j}G]_{b\vy}}{N} \\
        &= - \sum_{j=0}^{k} \frac{[(GA^t)^jG(GA)^{k-j}G]_{\vx\vy}}{N} =\landauOprec*{N^{k/2}\frac{\rho}{N\eta}\Bigl(1+\frac{\psi_k^\iso}{\sqrt{N\eta\rho}}\Bigr)}.
    \end{split}
\]

For the third and fourth line of~\cref{WGA cum exp iso} we distribute the derivatives according to the product rule to estimate (with absolute value inside the summation to address both diagonal and off-diagonal terms)
\begin{equation}\label{Xidiso}
    \begin{split}
        &\sum_{j\ge 2}\frac{1}{N^{(j+1)/2}}
        \sum_{a, b}\abs*{ (\partial_{ab}+\partial_{ba})^j \Bigl(x_a[(GA)^kG]_{b\vy} [(GA)^kG-m^{k+1} A^k]_{\vx\vy}^{p-1}\Bigr)} \\
        &\qquad \le \sum_{\substack{\sum\bm j\ge 2\\1\le\abs{\bm j}\le p}} \Lambda_k(\bm j) \abs*{[(GA)^kG-m^{k+1} A^k]_{\vx\vy}}^{p-\abs{\bm j}}
    \end{split}
\end{equation}
where  
\begin{equation}\label{Lambda def}
    \Lambda_k(\bm j):=N^{(n-\sum \bm j)/2}\sum_{ab}\abs*{\Bigl((\partial_{ab}+\partial_{ba})^{j_0}\frac{x_a[(GA)^kG]_{b\vy}}{\sqrt{N}}\Bigr)\prod_{i=1}^n\Bigl((\partial_{ab}+\partial_{ba})^{j_i}\frac{[(GA)^kG]_{\vx\vy}}{\sqrt{N}}\Bigr) }
\end{equation}
and  the summation  in~\eqref{Xidiso} 
is performed over all \(\bm j=(j_0,\ldots,j_n) \in \N_0^n\) with \(j_0\ge 0\), \(j_1,\ldots,j_n\ge 1\) and \(\abs{\bm j}=n+1\).
Recall that $\sum {\bm j}=j_0+ j_1+ j_2+\ldots +j_n$.  
\begin{lemma}\label{Lambda lemma}
    For any admissible \(\bm j\) in the summation~\cref{Xidiso} it holds that 
    \begin{equation}\label{Lambda eq}
        \Lambda_k(\bm j) \prec \Bigl(N^{k/2}\sqrt{\frac{\rho}{N\eta}}\Phi_k\Bigr)^{\abs{\bm j}}.
    \end{equation}
\end{lemma}
By combining~\cref{underline repl lemma iso,Lambda lemma,firstline iso,Xidiso} and~\cref{Lambda def}  we obtain 
\begin{equation}
    \abs*{\braket{\vx,[(GA)^k-m^{k+1}A^k]\vy}} \prec \cE_k^\iso + N^{k/2}\sqrt{\frac{\rho}{N\eta}}\Bigl(\Phi_k+\frac{1}{N\eta}\sum_{j=0}^{2k}\sqrt{\psi_{j}^\iso\psi_{2k-j}^\iso}+\frac{\psi_k^\iso}{N\eta}\Bigr),
\end{equation}
concluding the proof of~\cref{Psi master iso}.
\begin{proof}[Proof of~\cref{Lambda lemma}] 
    We recall the notations \(\Omega_k,\Phi_k\) from~\cref{Phi k def}.
    For a naive bound we estimate all but the first factor trivially in~\eqref{Lambda def} with 
    \begin{equation}\label{trivi}
        \abs*{(\partial_{ab}+\partial_{ba})^{j_i}\frac{[(GA)^kG]_{\vx\vy}}{\sqrt{N}}}\prec \frac{N^{k/2}}{N^{1/2}}\Omega_k.
    \end{equation}
    Note that the estimate is independent of the number
    of derivatives. For the first factor in~\eqref{Lambda def} we estimate, after performing the derivatives, all but the last \([(GA)^{k_i}G]\)-factor (involving \(\bm y\)) trivially by \eqref{eq naive iso GAG} as
    \begin{equation}\label{trivFirst}
        \abs*{(\partial_{ab}+\partial_{ba})^{j_0}\frac{x_a[(GA)^kG]_{b\vy}}{\sqrt{N}}}\prec \sum_{j=0}^k N^{(k-j)/2}\Omega_{k-j}\abs{x_a}\frac{\abs{[(GA)^jG]_{a\vy}}+\abs{[(GA)^jG]_{b\vy}}}{\sqrt{N}}.
    \end{equation}
    By combining~\cref{trivi,trivFirst} and the Schwarz-inequality 
    \begin{equation}
    \begin{split}
        \sum_{ab} \abs{x_a}\frac{\abs{[(GA)^jG]_{a\vy}}+\abs{[(GA)^jG]_{b\vy}}}{\sqrt{N}}& \le \sqrt{N}\norm{\vx} \sqrt{[(G^\ast A)^j G^\ast G (AG)^j]_{\vy\vy}}\\
        &\prec N^{j/2+1}\sqrt{\frac{\rho}{N\eta}}\Bigl(1+\frac{\psi_{2j}^\iso}{\sqrt{N\eta\rho}}\Bigr)^{1/2}
        \end{split}
    \end{equation}
    we conclude 
    \begin{equation}
        \Lambda_k(\bm j) \prec N^{(n-\sum\bm j)/2+1} N^{k/2}\sqrt{\frac{\rho}{N\eta}}\Phi_k \Bigl(N^{k/2}\frac{1}{\sqrt{N}}\Omega_k\Bigr)^{\abs{\bm j}-1},
    \end{equation}
    which implies~\cref{Lambda eq} in the case when \(\sum\bm j\ge n+2\) using that $\Omega_k\le \Phi_k$ and $\rho/\eta\gtrsim 1$. It thus only remains to consider the cases \(\sum \bm j=n\) and \(\sum \bm j=n+1\). 
    
    If \(\sum \bm j=n\), then \(n\ge 2\) and \(j_0=0\), \(j_1=j_2=\cdots=1\). By estimating the \(j_2,j_3,\ldots\) factors in~\cref{Lambda def} using~\cref{trivi} we then bound 
    \begin{equation}
        \begin{split}
           & \Lambda_k(\bm j) \prec \Bigl(N^{k/2}\frac{\Omega_k}{\sqrt{N}}\Bigr)^{\abs{\bm j}-2}\sum_{ab} \frac{\abs{x_a}\abs*{[(GA)^kG]_{b\vy}}}{\sqrt{N}} \sum_{j=0}^k \frac{\abs*{[(GA)^jG]_{\vx a}}\abs*{[(GA)^{k-j}G]_{b \vy}}}{\sqrt{N}}\\
            &\lesssim \Bigl(N^{k/2}\frac{\Omega_k}{\sqrt{N}}\Bigr)^{\abs{\bm j}-2} \frac{\sqrt{[(G^\ast A)^k G^\ast G (AG)^k]_{\vy\vy}}}{\sqrt{N}} \sum_{j=0}^k\frac{\sqrt{[(G^\ast A)^j G^\ast G (AG)^j]_{\vy\vy}[(G^\ast A)^{k-j} G^\ast G (AG)^{k-j}]_{\vx\vx}}}{\sqrt{N}}\\
            &\prec \Bigl(N^{k/2}\frac{\Omega_k}{\sqrt{N}}\Bigr)^{\abs{\bm j}-2} \Bigl(N^{k/2}\sqrt{\frac{\rho}{N\eta}}\Bigr)^2 \sqrt{\frac{\rho}{\eta}}\Phi_k \sum_{j=0}^k \Bigl(1+\frac{\psi_{2j}^\iso}{\sqrt{N\eta\rho}}\Bigr)^{1/2}\Bigl(1+\frac{\psi_{2(k-j)}^\iso}{\sqrt{N\eta\rho}}\Bigr)^{1/2}\lesssim \Bigl(N^{k/2}\sqrt{\frac{\rho}{N\eta}}\Phi_k\Bigr)^{\abs{\bm j}}
        \end{split}
    \end{equation}
    using \(\abs{\bm j}\ge 3\) and \(\Omega_k\le\Phi_k\), \(1\lesssim \rho/\eta\) in the last step.
    
    Finally, if \(\sum\bm j=n+1\), then \(n\ge 1\) by admissibility and either \(j_0=0\) or \(j_1=1\). In the first case we estimate the \(j_2,j_3,\ldots\) factors in~\cref{Lambda def} using~\cref{trivi}, and all but the first \([(GA)^jG]_{\vx\cdot}\) in the \(j_1\)-factor after differentiation trivially to obtain 
    \begin{equation}
        \begin{split}
            \Lambda_k(\bm j) &\prec N^{-1/2}\Bigl(N^{k/2}\frac{\Omega_k}{\sqrt{N}}\Bigr)^{\abs{\bm j}-2}\sum_{ab} \frac{\abs{x_a}\abs*{[(GA)^kG]_{b\vy}}}{\sqrt{N}} \sum_{j=0}^k N^{(k-j)/2}\Omega_{k-j}\frac{\abs*{[(GA)^jG]_{\vx a}}+\abs*{[(GA)^jG]_{\vx b}}}{\sqrt{N}}\\
            &\prec \Bigl(N^{k/2}\frac{\Omega_k}{\sqrt{N}}\Bigr)^{\abs{\bm j}-2} \Bigl(N^{k/2}\sqrt{\frac{\rho}{N\eta}}\Phi_k\Bigr)^2,
        \end{split}
    \end{equation}
    again using a Schwarz inequality. Finally, in the \(j_1=1\) case we estimate two \(j_0\)-factor using~\cref{trivFirst}, the \(j_2,j_3,\ldots\) factors trivially, and  to bound 
    \begin{equation}
        \begin{split}
            \Lambda_k(\bm j) &\prec N^{-1/2}\Bigl(N^{k/2}\frac{\Omega_k}{\sqrt{N}}\Bigr)^{\abs{\bm j}-2} \\
            &\qquad \times\sum_{ab} \sum_{j,l=0}^k N^{(k-l)/2}\Omega_{k-l}\abs{x_a}\frac{\abs{[(GA)^lG]_{a\vy}}+\abs{[(GA)^lG]_{b\vy}}}{\sqrt{N}} \frac{\abs*{[(GA)^{j}G]_{\vx a}}\abs*{[(GA)^{k-j}G]_{b\vy}}}{\sqrt{N}}\\
            &\prec \Bigl(N^{k/2}\frac{\Omega_k}{\sqrt{N}}\Bigr)^{\abs{\bm j}-2} \Bigl(N^{k/2}\sqrt{\frac{\rho}{N\eta}}\Phi_k\Bigr)^2,
        \end{split}
    \end{equation}
    where we used the trivial bound for the \(\abs*{[(GA)^{j}G]_{\vx a}}\) in order to estimate the remaining terms by a Schwarz inequality. This completes the proof of the lemma.
\end{proof}

\subsection{Reduction inequalities and bootstrap}\label{sec bootstrap}

In this section we prove the reduction inequalities in Lemma~\ref{pro:redin} and conclude the proof of our main result Theorem~\ref{theorem multi G local law} showing that $\psi_k^{\av/\iso}\lesssim 1$ for any $k\ge 0$.

\begin{proof}[Proof of Lemma~\ref{pro:redin}]
   The proof of this proposition is very similar to \cite[Lemma 3.6]{MR4479913}, we thus present only the proof in the averaged case. Additionally, we only prove the case when $k$ is even, if $k$ is odd the proof is completely analogous.
       
    Define $T=T_k:=A(GA)^{k/2-1}$, write $(GA)^{2k} = GTGTGTGT$ and use the spectral theorem
    for these four intermediate resolvents.
    Then, using that $|m_i|\lesssim 1$ and that $|\braket{A^k}|\lesssim N^{k/2-1}\braket{|A|^2}^{k/2}$,
    after a Schwarz inequality in the third line, we conclude that
    \[
    \begin{split}
        &\Psi_{2k}^\av=\frac{N^{(3-2k)/2}\eta^{1/2}}{\rho^{1/2}\braket{|A|^2}^k}\big|\braket{(GA)^{2k}-m_1\dots m_{2k}A^{2k}}\big| \\
        &\lesssim \sqrt{\frac{N\eta}{\rho}}+\frac{N^{(3-2k)/2}\eta^{1/2}}{N\rho^{1/2}\braket{|A|^2}^k}\left|\sum_{ijml}\frac{\braket{{\bm u}_i,T{\bm u}_j} \braket{{\bm u}_j,T{\bm u}_m}\braket{{\bm u}_m,T{\bm u}_l}\braket{{\bm u}_l,T{\bm u}_i}}{(\lambda_i-z_1)(\lambda_j-z_{k/2+1})(\lambda_m-z_{k+1})(\lambda_l-z_{3k/2+1})}\right| \\
        &\lesssim \sqrt{\frac{N\eta}{\rho}}+\frac{N^{(3-2k)/2+1}\eta^{1/2}}{\rho^{1/2}\braket{|A|^2}^k}\braket{|G|A(GA)^{k/2-1}|G|A(G^*A)^{k/2-1}}\braket{|G|A(GA)^{k/2-1}|G|A(G^*A)^{k/2-1}} \\
        &\lesssim \sqrt{\frac{N\eta}{\rho}}+\frac{N^{(3-2k)/2+1}\eta^{1/2}}{\rho^{1/2}\braket{|A|^2}^k}\left(N^{k/2-1}\braket{|A|^2}^{k/2}+\frac{\rho^{1/2}\braket{|A|^2}^{k/2}}{N^{(3-k)/2}\eta^{1/2}}\psi_k^\av\right)^2 \\
        &\lesssim \sqrt{\frac{N\eta}{\rho}}+\sqrt{\frac{\rho}{N\eta}}(\psi_k^\av)^2.
    \end{split}
    \]
    We remark that to bound $\braket{|G|A(GA)^{k/2-1}|G|A(G^*A)^{k/2-1}}$ in terms of $\psi_k^\av$ we used (ii) 
    of~\cref{lemma Psi G prod} together with $G^*(z) = G(\bar z)$.
\end{proof}

We are now ready to conclude the proof of our main result.

\begin{proof}[Proof of Theorem~\ref{theorem multi G local law}]

The proof repeatedly uses
 a simple argument called \emph{iteration}. By this we mean the following observation:
    whenever we know that \(X\prec x\)  implies
    \begin{equation}\label{X iter}
        X\prec A + \frac{x}{B} + x^{1-\alpha} C^{\alpha},
    \end{equation}
    for some constants \(B\ge N^\delta\), \(A,C>0\), and 
    exponent \(0<\alpha<1\), and we know that  \(X\prec N^D\) initially (here $\delta, \alpha$ and $D$ are $N$-independent 
    positive constants,
    other quantities  may depend on $N$)
    then  we also know that \(X\prec x\)  implies
     \begin{equation}\label{XAC}
        X\prec A + C. 
    \end{equation} 
The proof is simply to  
  iterate~\eqref{X iter} finitely many times  (depending only  on $\delta, \alpha$ and $D$).
 The fact that $\Psi_k^{\av/\iso}\prec N^D$ follows by a simple norm bound
    on the resolvents and $A$, so the condition \(X\prec N^D\) is always satisfied in our applications.

    By the standard single resolvent local laws in \eqref{oneG} we know that $\psi_0^\av=\psi_0^\iso=1$. 
    Using the master inequalities in~\cref{prop master} and the reduction bounds from~\cref{pro:redin}, 
    in the first step we will show that $\Psi_k^{\av/\iso}\prec \rho^{-1}$ for any $k\ge 1$ as an a priori bound. 
    Then, in the second step, we  feed this bound into the tandem of the master inequalities and 
    the reduction bounds to 
     improve the estimate to  $\Psi_k^{\av/\iso}\prec 1$.
     The first step is the critical stage of the proof, 
     here we need to show that our bounds are sufficiently strong to close the hierarchy of our estimates
     to yield a better bound on $\Psi_k^{\av/\iso}$ than the trivial $\Psi_k^{\av/\iso}\le N^{k/2}\eta^{-k-1}$
     estimate obtained by using the norm bounds $\norm{G}\le \eta^{-1}$ and $\norm{A}\le N^{1/2}$. 
     Once some improvement is achieved, it can be relatively easily iterated.

    The proof of $\Psi_k^{\av/\iso}\prec \rho^{-1}$ proceeds by  a step-two induction, we first prove that 
    $\Psi_k^{\av/\iso}\prec \rho^{-1}$ for $k=1,2$ and then show that if $\Psi_n^{\av/\iso}\prec \rho^{-1}$ 
    holds for all $n\le k-2$, for some $k\ge 4$, then it also holds for $\Psi_{k-1}^{\av/\iso}$ and $\Psi_{k}^{\av/\iso}$.
    
    Using~\eqref{Psi master av}--\eqref{Psi master iso}  we have
    \begin{equation}
        \label{eq:aib1}
        \begin{split}
            \Psi_1^\av&\prec 1+\frac{\sqrt{\psi_2^\iso}}{(N\eta\rho)^{1/4}}+\frac{\sqrt{\psi_2^\av}}{(N\eta\rho)^{1/4}}+(\psi_1^\iso)^{2/3}+\psi_1^\iso\sqrt{\frac{\rho}{N\eta}} +\frac{\psi_1^\av}{\sqrt{N\eta}} \\
            \Psi_1^\iso&\prec 1+\frac{\sqrt{\psi_2^\iso}}{(N\eta\rho)^{1/4}}+\frac{\psi_1^\av}{N\eta}+\psi_1^\iso\sqrt{\frac{\rho}{N\eta}}+\frac{\psi_1^\iso}{N\eta}
        \end{split}
    \end{equation}
    for $k=1$, using 
    $$
     \Phi_1\lesssim 1+ \psi_1^\iso \sqrt{\frac{\rho}{N\eta}} + \frac{\sqrt{ \psi_2^\iso}}{(N\eta\rho)^{1/4}}.
     $$
     Similarly, for $k=2$, using that $\Omega_1\le \Phi_1$, and estimating explicitly 
    $$
      \Phi_2 \lesssim 1%
      +(\psi_1^\iso)^2 \frac{\rho}{N\eta}
      + \frac{\psi_2^\iso}{(N\eta\rho)^{1/2}} 
      +  \frac{ (\psi_4^\iso)^{1/2}}{(N\eta\rho)^{1/4}} %
   $$
  by Schwarz inequalities and plugging it into~\eqref{Psi master av}--\eqref{Psi master iso}
   we have  
        \begin{equation}
        \label{eq:2aib}
        \begin{split}
            \Psi_2^\av&\prec 1+\psi_1^\av+\frac{\psi_2^\iso}{(N\eta\rho)^{1/12}}+\frac{\sqrt{\psi_4^\iso}}{(N\eta\rho)^{1/4}}+\frac{\sqrt{\psi_4^\av}}{(N\eta\rho)^{1/4}}+(\psi_2^\iso)^{2/3}+\sqrt{\psi_1^\iso\psi_2^\iso} \\
            &\quad+\frac{(\psi_2^\iso)^{3/4}(\psi_1^\iso)^{1/2}}{(N\eta\rho)^{1/8}}+\frac{\rho^{1/3}(\psi_2^\iso)^{2/3}(\psi_1^\iso)^{1/3}}{(N\eta\rho)^{1/6}}+\frac{\sqrt{\rho}(\psi_1^\av)^2}{(N\eta)^{3/2}}+(\psi_1^\iso)^2\frac{\rho}{N\eta} \\
            &\quad+(\psi_1^\iso)^{3/2}\sqrt{\frac{\rho}{N\eta}}+\frac{\rho^{1/2}\psi_1^\iso(\psi_2^\iso)^{1/2}}{(N\eta\rho)^{1/4}}+\frac{\psi_2^\av}{\sqrt{N\eta}},\\
            \Psi_2^\iso&\prec 1+\psi_1^\iso+\frac{\psi_1^\av}{N\eta}+\frac{\psi_2^\iso+\psi_2^\av}{(N\eta\rho)^{1/2}}+\frac{\sqrt{\psi_4^\iso}}{(N\eta\rho)^{1/4}} +\frac{\sqrt{\psi_1^\iso\psi_3^\iso}}{N\eta}+\frac{\psi_1^\av\psi_1^\iso}{(N\eta)^{3/2}}+(\psi_1^\iso)^2\frac{\rho}{N\eta}+\frac{\psi_2^\iso}{N\eta}.
        \end{split}
    \end{equation}
     In these estimates we frequently used that $\psi_k^{\av/\iso}\ge 1$, $\rho\lesssim 1$, $\rho/N\eta\le 1$  and $N\eta\rho\ge 1$ to 
    simplify the formulas.
        
    By \eqref{eq:aib1}-\eqref{eq:2aib}, using iteration for the sum $\Psi_1^\av+\Psi_1^\iso$, we readily conclude
    \begin{equation}
        \label{eq:eas1}
        \Psi_1^\av+\Psi_1^\iso\prec 1+\frac{\sqrt{\psi_2^\iso}}{(N\eta\rho)^{1/4}}+\frac{\sqrt{\psi_2^\av}}{(N\eta\rho)^{1/4}}.
    \end{equation}
   Note that since \eqref{eq:eas1} holds uniformly in the hidden parameters $A, z, \vx, \vy$  in $\Psi_1^{\av/\iso}$,  this bound serves as
   an upper bound on $\psi_1^\av+\psi_1^\iso$ (in the sequel, we will frequently use an already proven upper bound on $\Psi_k$ as
   an effective  upper bound on $\psi_k$ in the next steps without  explicitly mentioning it).
   Next, using this upper bound
   together with an iteration for $\Psi_2^\av+\Psi_2^\iso$ we have from~\eqref{eq:2aib} 
    \begin{equation}
        \label{eq:eas2}
        \Psi_2^\av+\Psi_2^\iso\prec 1+\frac{\sqrt{\psi_4^\iso}}{(N\eta\rho)^{1/4}}+\frac{\sqrt{\psi_4^\av}}{(N\eta\rho)^{1/4}}+\frac{\sqrt{\psi_1^\iso\psi_3^\iso}}{N\eta},
    \end{equation}
  again after several simplifications by Young's inequality and the basic inequalities 
  $\psi_k^{\av/\iso}\ge 1$, $\rho\lesssim 1$ and $N\eta\rho\ge 1$.

We now apply the reduction inequalities from Lemma~\ref{pro:redin} in the form
    \begin{equation}
        \label{eq:red}
        \begin{split}
            \Psi_4^\av&\prec \sqrt{\frac{N\eta}{\rho}}+\sqrt{\frac{\rho}{N\eta}}(\psi_2^\av)^2 \\
            \Psi_4^\iso&\prec  \sqrt{\frac{N\eta}{\rho}}+\psi_2^\av+\psi_2^\iso+\sqrt{\frac{\rho}{N\eta}}\psi_2^\av\psi_2^\iso \\
            \Psi_3^\iso&\prec  \sqrt{\frac{N\eta}{\rho}}+\left(\frac{N\eta}{\rho}\right)^{1/4}\sqrt{\psi_2^\av}+\psi_2^\iso+\left(\frac{\rho}{N\eta}\right)^{1/4}\psi_2^\iso\sqrt{\psi_2^\av},
        \end{split}
    \end{equation}
where the first inequality was inserted already  into the right hand side of~\eqref{eq:rediniso} to get  the second inequality
in~\eqref{eq:red}. 

    Then,  inserting~\eqref{eq:red} and~\eqref{eq:eas1} into~\eqref{eq:eas2}  and 
    using iteration, we conclude
    \begin{equation}
        \Psi_2^\av+\Psi_2^\iso\prec \frac{1}{\sqrt{\rho}}+\frac{\sqrt{\psi_2^\iso}+\sqrt{\psi_2^\av}}{(N\eta\rho)^{1/4}}+\frac{\psi_2^\av+\psi_2^\iso}{(N\eta)^{1/2}}
    \end{equation}
    which together with~\eqref{eq:eas1} implies
    \begin{equation}\label{Psi12}
    \Psi_1^\iso+\Psi_1^\av\prec \rho^{-1/4}, \qquad  \Psi_2^\iso+\Psi_2^\av\prec\rho^{-1/2}.
    \end{equation}

  We now proceed with a step-two induction on $k$. The initial step of the induction is~\eqref{Psi12}.
  Fix an even $k\ge 4$ and assume that
  \begin{equation}
  \label{eq:indhyp}
  \Psi_n^{\av/\iso}\prec \frac{1}{\rho} \qquad \mathrm{for}\quad n\le k-2.
  \end{equation}
  First of all we notice that using the reduction inequality \eqref{eq:redinav} for $k/2$ we obtain (assuming that $k$ is divisible by four)
  \begin{equation}
  \label{eq:impinf}
  \Psi_k^\av\prec \sqrt{\frac{N\eta}{\rho}}+\sqrt{\frac{\rho}{N\eta}}(\psi_{k/2}^\av)^2\prec  \sqrt{\frac{N\eta}{\rho}},
  \end{equation}
  where in the last inequality we used \eqref{eq:indhyp}. If $k$ is not divisible by four we conclude the same bound using the second inequality in \eqref{eq:redinav} instead of the first one. Next, using again the reduction inequality \eqref{eq:redinav} and the induction hypothesis \eqref{eq:indhyp} we obtain
  \begin{equation}
  \label{eq:apbav}
  \Psi_{2n}^\av \prec \sqrt{\frac{N\eta}{\rho}} \qquad \mathrm{for} \quad n\le k.
  \end{equation}
  We point out that for $n=k$ we used
  \[
  \Psi_{2k}^\av\prec \sqrt{\frac{N\eta}{\rho}}+\sqrt{\frac{\rho}{N\eta}}(\psi_k^\av)^2\prec \sqrt{\frac{N\eta}{\rho}},
  \]
  where in the last inequality we used \eqref{eq:impinf}, and a similar bound for $n=k-1$. Similarly, in the isotropic case, using \eqref{eq:rediniso}, we obtain
  \begin{equation}
  \label{eq:apbiso}
  \Psi_{k+j}^\iso\prec  \sqrt{\frac{N\eta}{\rho}}+  \left(\frac{N\eta}{\rho}\right)^{1/4}\sqrt{\psi_{2j}^\av}+\psi_k^\iso+\left(\frac{\rho}{N\eta}\right)^{1/4}\sqrt{\psi_{2j}^\av}\psi_k^\iso\prec  \sqrt{\frac{N\eta}{\rho}} \qquad \mathrm{for} \quad j\le k,
  \end{equation}
  where in the last inequality we used \eqref{eq:apbav} and that by \eqref{eq:rediniso} it follows
  \begin{equation}
  \label{eq:impinf2}
  \psi_k^\iso=\psi_{(k-2)+2}^\iso\prec \sqrt{\frac{N\eta}{\rho}} +\left(\frac{N\eta}{\rho}\right)^{1/4}\sqrt{\psi_4^\av}+\psi_{k-2}^\iso+\left(\frac{\rho}{N\eta}\right)^{1/4} \sqrt{\psi_4^\av}\psi_{k-2}^\iso\prec \sqrt{\frac{N\eta}{\rho}} .
  \end{equation}
  We point out that similarly, writing $k-1=(k-2)+1$, we also conclude that $\psi_{k-1}^\iso\prec \sqrt{N\eta/\rho}$. Then, by using \eqref{eq:apbav} and the induction hypothesis \eqref{eq:indhyp} in the definition of $\Phi_n$ in \eqref{Phi k def}, we readily conclude that
  \begin{equation}
  \label{eq:bPhis}
  \Phi_n\prec \frac{1}{\rho} \qquad \mathrm{for}\quad n\le k.
  \end{equation}
  Furthermore, we notice that, by using  \eqref{eq:indhyp} and \eqref{eq:impinf2} (as well as the similar bound for $\psi_{k-1}^\iso$ stated below it) in the definition of $\Omega_n$ in \eqref{Phi k def}, we also have
   \begin{equation}
  \label{eq:bOmes}
  \Omega_n\prec 1 \qquad \mathrm{for}\quad n\le k.
  \end{equation}
  
  We are now ready to consider the master inequalities for $\Psi_{k-1}^{\av/\iso}$ and $\Psi_k^{\av/\iso}$. Using \emph{iteration} as in \eqref{X iter}--\eqref{XAC}, together with the bounds \eqref{eq:apbav}, \eqref{eq:apbiso}, \eqref{eq:bPhis}--\eqref{eq:bOmes}, by \eqref{Psi master av}--\eqref{Psi master iso}, we obtain (recall that $N\eta\rho\gg 1$ and $\rho\lesssim 1$)
  \begin{equation}
  \label{eq:masterk-1}
  \begin{split}
  \Psi_{k-1}^\av&\prec \frac{1}{\rho}+(\psi_{k-1}^\iso)^{2/3}\rho^{-1/3}+\rho^{-1/2}\sqrt{\psi_{k-1}^\iso} \\
  \Psi_{k-1}^\iso&\prec\frac{1}{\rho}+\frac{\psi_{k-1}^\av}{N\eta},
  \end{split}
  \end{equation}
  and
    \begin{equation}
  \label{eq:masterk}
  \begin{split}
  \Psi_k^\av&\prec \frac{1}{\rho}+(\psi_k^\iso)^{2/3}\rho^{-1/3}+\rho^{-1/2}\big(\sqrt{\psi_k^\iso}+\sqrt{\psi_{k-1}^\iso}\big)+\sqrt{\psi_{k-1}^\iso\psi_k^\iso} +\frac{\psi_{k-1}^\av}{N\eta}\\
  \Psi_k^\iso&\prec\frac{1}{\rho}+\frac{\psi_k^\av+\psi_{k-1}^\av}{N\eta}.
  \end{split}
  \end{equation}
 Then, using \emph{iteration} in \eqref{eq:masterk-1} for $\Psi_{k-1}^\av+\Psi_{k-1}^\iso$, we immediately conclude that $\Psi_{k-1}^{\av/\iso}\prec \rho^{-1}$. Finally, plugging this information into \eqref{eq:masterk}, and using iteration once again for $\Psi_k^\av+\Psi_k^\iso$ we conclude that $\Psi_k^{\av/\iso}\prec \rho^{-1}$ as well. This completes the step-two induction hence the  first and the pivotal step of the proof. 
 
 In the second step we improve $ \Psi_k^{\av/\iso}\prec  \rho^{-1}$   to $ \Psi_k^{\av/\iso}\prec 1$ for all $k$. By plugging the bound $ \Psi_k^{\av/\iso}\prec  \rho^{-1}$  into the master inequalities
     in Proposition~\ref{prop master} and noticing that $\Phi_k \le 1 + \rho^{-1} (N\eta\rho)^{-1/4}$,  we directly conclude that 
   \begin{equation}
     \Psi_k^{\av/\iso}\prec 1+\frac{\rho^{-1}}{(N\eta\rho)^{1/12}}
    \end{equation}
     for any $k\ge 0$. We point that that the exponent $1/12$ comes from the fifth term in the first line of \eqref{Psi master av}.  Now we can use this improved inequality by plugging it again in the master inequalities
     to achieve
        \begin{equation}
     \Psi_k^{\av/\iso}\prec 1+\rho^{-1}\left(\frac{1}{(N\eta\rho)^{1/12}}\right)^2,
    \end{equation}
    and so on.    Recalling the assumption that $N\eta\rho\ge N^\epsilon$,
      we need to iterate this process
     finitely many times (depending on $k$, $\xi, K, \epsilon$) to achieve  $\Psi_k^{\av/\iso}\prec1$ also in the second regime. This concludes
 the proof of the theorem.

     \end{proof}

\section{Stochastic Eigenstate Equation and proof of Theorem~\ref{theo:flucque}}\label{sec:see}

Armed with the new local law (Theorem~\ref{theorem multi G local law}) and its direct corollary on the eigenvector overlaps (Theorem~\ref{thm:eth}),
the rest of the
proof of Theorem~\ref{theo:flucque} is very similar to the proof of \cite[Theorem 2.2]{MR4413210}, which is presented in \cite[Sections 3-4]{MR4413210}. For this reason we only explain the differences and refer to \cite{MR4413210} for a fully detailed proof. We mention that the proof in \cite{MR4413210} heavily relies on the theory of the stochastic eigenstate equation
initiated in~\cite{MR3606475} and then further developed in~\cite{MR4156609, MR4272266}.

Similarly to \cite[Sections 3-4]{MR4413210}, we present the proof only in the real case (the complex case is completely analogous and so omitted). We will prove Theorem~\ref{theo:flucque} dynamically, i.e. we consider the Dyson Brownian motion (DBM)
with initial condition $W$ and 
we will show that the overlaps of the eigenvectors
have Gaussian fluctuations after a time $t$ slightly bigger than $N^{-1}$.  With a separate
argument  then in Appendix~\ref{sec:GFT} we show that the (small) 
Gaussian component added along the DBM  flow
can be removed at the price of a negligible error.

More precisely, we consider the matrix flow
\begin{equation}\label{eq:matdbm}
    \dif W_t=\frac{\dif \widetilde{B}_t}{\sqrt{N}}, \qquad W_0=W,
\end{equation}
where \(\widetilde{B}_t\) is a standard real symmetric matrix Brownian motion (see e.g.~\cite[Definition 2.1]{MR3606475}). We denote the resolvent of \(W_t\) by \(G=G_t(z):=(W_t-z)^{-1}\), for \(z\in\mathbf{C}\setminus\mathbf{R}\). It is well know that
in the limit $N\to\infty$ the resolvent $G_t(z):=(W_t-z)^{-1}$, for $z\in\mathbf{C}\setminus\mathbf{R}$, becomes 
approximately  deterministic, and that its 
deterministic approximation is given by the scalar matrix $m_t\cdot I$. The function  $m_t=m_t(z)$ is the unique solution of the 
complex Burgers equation
\begin{equation}
    \partial_tm_t(z)=-m_t\partial_zm_t(z), \qquad m_0(z)=m(z),
\end{equation}
with initial condition $m(z)=m_{\mathrm{sc}}(z)$ being the Stieltjes transform of the semicircular law. Denote $\rho_t=\rho_t(z):=\pi^{-1}\Im m_t(z)$, then it is easy to see that $\rho_t(x+\ii 0)$ is a rescaling of $\rho_0=\rho_{\mathrm{sc}}$ by a factor $1+t$. In fact, $W_t$ is a Wigner matrix itself, with a normalization $\E | (W_t)_{ab}|^2 = N^{-1}(1+t)$ 
with  a Gaussian component.

Denote by $\lambda_1(t)\le\lambda_2(t)\le\dots\le\lambda_N(t)$ the eigenvalues of $W_t$, and let $\{{\bm u}_i(t)\}_{i\in [N]}$ be the corresponding eigenvectors. Then, it is known \cite[Theorem 2.3]{MR3606475} that $\lambda_i=\lambda_i(t)$, ${\bm u}_i={\bm u}_i(t)$ are the unique strong solutions of the following system of stochastic differential equations:
\begin{align}
    \label{eq:evaluflow}
    \dif \lambda_i&=\frac{\dif B_{ii}}{\sqrt{N}}+\frac{1}{N}\sum_{j\ne i} \frac{1}{\lambda_i-\lambda_j} \dif t \\\label{eq:evectorflow}
    \dif {\bm u}_i&=\frac{1}{\sqrt{N}}\sum_{j\ne i} \frac{\dif B_{ij}}{\lambda_i-\lambda_j}{\bm u}_j-\frac{1}{2N}\sum_{j\ne i} \frac{{\bm u}_i}{(\lambda_i-\lambda_j)^2}\dif t,
\end{align}
where \(B_t=(B_{ij})_{i,j\in [N]}\) is a standard real symmetric matrix Brownian motion (see e.g.~\cite[Definition 2.1]{MR3606475}).

Note that the flow for the diagonal overlaps \(\braket{{\bm u}_i, A {\bm u}_i}\), by \eqref{eq:evectorflow},
naturally also depends  on the off-diagonal overlap \(\braket{{\bm u}_i, A {\bm u}_j}\). Hence, even if we are only interested in diagonal overlaps, our analysis must also handle  off-diagonal overlaps.
In particular, this implies that there is no closed differential equation for only diagonal or only off-diagonal overlaps. However, in \cite{MR4156609} Bourgade, Yau, and Yin proved that the \emph{perfect matching observable}  $f_{{\bm \lambda},t}$, which is presented in \eqref{eq:deff} below, satisfies a parabolic PDE (see \eqref{eq:1dequa} below). We now describe how the observable $f_{{\bm \lambda},t}$ is constructed.

\subsection{Perfect matching observables}

Without loss of generality for the rest of the paper we assume that \(A\) is traceless, \(\braket{A}=0\), i.e.\ \(A=\mathring{A}\). 
We introduce the short-hand notation for the \emph{eigenvector overlaps}
\begin{equation}
    p_{ij}=p_{ij}(t):=\braket{{\bm u}_i(t), A {\bm u}_j(t)},
    \quad i, j\in [N]. 
\end{equation}
To compute the moments, we will consider monomials of eigenvector overlaps of the form $\prod_k p_{i_k j_k}$
where each index occurs an even number of times. 
We start  by introducing a particle picture and a certain graph that encode such  monomials:  each particle on the set of integers \([N]\) corresponds to two occurrences of an index \(i\) in the monomial product.
This particle picture was introduced in~\cite{MR3606475} and
heavily used also
in~\cite{MR4156609, MR4272266}. Each particle configuration
is encoded by a function  \({\bm \eta}:[N] \to \mathbf{N}_0\), where \(\eta_j:={\bm \eta}(j)\) denotes the number of particles at the site \(j\), and \(n({\bm \eta}):=\sum_j \eta_j= n\) is the total number of particles.  We denote the space of \(n\)-particle configurations by \(\Omega^n\).
Moreover, for any index pair \(i\ne j\in[N]\), we define \({\bm \eta}^{ij}\) to be the configuration obtained moving a particle from the site \(i\) to the site \(j\), if there is no particle in \(i\) then  \({\bm \eta}^{ij}:={\bm \eta}\).

We now define the \emph{perfect matching observable} (introduced in \cite{MR4156609}) for any given configuration \({\bm \eta}\):
\begin{equation}
    \label{eq:deff}
    f_{{\bm \lambda},t}({\bm \eta}):= \frac{N^{n/2}}{ [2\braket{{A}^2}]^{n/2}} \frac{1}{(n-1)!!}\frac{1}{\mathcal{M}({\bm \eta}) }\E\left[\sum_{G\in\mathcal{G}_{\bm \eta}} P(G)\Bigg| 
    {\bm \lambda}\right], \quad \mathcal{M}({\bm \eta}):=\prod_{i=1}^N (2\eta_i-1)!!,
\end{equation}
with \(n\) being the number of particles in the configuration \({\bm \eta}\). Here $\mathcal{G}_{\bm \eta}$ denotes the set of perfect matchings on the complete graph with vertex set
\[
\mathcal{V}_{\bm \eta}:=\{(i,a): 1\le i\le n, 1\le a\le 2\eta_i\},
\]
and
\begin{equation}
    \label{eq:pg}
    P(G):=\prod_{e\in\mathcal{E}(G)}p(e), \qquad p(e):=p_{i_1i_2},
\end{equation}
where  \(e=\{(i_1,a_1),(i_2,a_2)\}\in \mathcal{V}_{\bm \eta}^2\), and \(\mathcal{E}(G)\) denotes the edges of \(G\). Note that in \eqref{eq:deff} we took the conditioning on the entire flow of eigenvalues, \({\bm \lambda} =\{\bm \lambda(t)\}_{t\in [0,T]}\) for some  fixed  \(T>0\). From now on we will always assume that $T\ll 1$ (even if not stated explicitly).

We always assume that the entire eigenvalue trajectory 
\(\{\bm \lambda(t)\}_{t\in [0,T]}\) satisfies the usual rigidity estimate asserting 
that the eigenvalues are very close to the deterministic quantiles of the semicircle law with very high probability. 
To formalize it, we define \begin{equation}\label{def:Omega}
    \widetilde{\Omega}=\widetilde{\Omega}_\xi:= \Big\{ \sup_{0\le t \le T} \max_{i\in[N]} N^{2/3}\widehat{i}^{1/3} | \lambda_i(t)-\gamma_i(t)| \le N^\xi\Big\}
\end{equation}
for any $\xi>0$, where \(\widehat{i}:=i\wedge (N+1-i)\). 
Here \(\gamma_i(t)\) denote the \emph{quantiles} of $\rho_t$, defined by
\begin{equation}\label{eq:quantin}
    \int_{-\infty}^{\gamma_i(t)} \rho_t(x)\, \dif x=\frac{i}{N}, \qquad i\in [N],
\end{equation}
where \(\rho_t(x)= \frac{1}{2(1+t)\pi}\sqrt{(4(1+t)^2-x^2)_+}\) is the semicircle law corresponding to \(W_t\).
Note that \(|\gamma_i(t)-\gamma_i(s)|\lesssim |t-s|\) for any  bulk index $i$ and
for any \(t,s\ge 0\).

The well known rigidity estimate  (see e.g.~\cite[Theorem 7.6]{MR3068390} or~\cite{MR2871147}) asserts
that
\[
\mathbf{P} (\widetilde{\Omega}_\xi)\ge 1-  C(\xi, D) N^{-D}
\]
for any (small) \(\xi>0\) and (large) \(D>0\).  This was proven for any fixed $t$ e.g. in~\cite[Theorem 7.6]{MR3068390} or~\cite{MR2871147}, the extension to all $t$ follows by a grid argument together with the fact that ${\bm \lambda}(t)$ is stochastically 
$1/2$-H\"older in $t$, which follows by Weyl's inequality
\[
\norm{{\bm \lambda}(t)-{\bm \lambda}(s)}_\infty\lesssim \norm{W_t-W_s}\stackrel{\dif}{=}\norm{W+\sqrt{s}U_1+\sqrt{t-s}U_2-W-\sqrt{s}U_1}\lesssim \sqrt{t-s},
\]
with $s\le t$ and $U_1,U_2$ being independent GUE/GOE matrices, which are also independent of $W$.

By~\cite[Theorem 2.6]{MR4156609} we know that the perfect matching observable
$f_{{\bm \lambda},t}$ is a solution of the following parabolic discrete PDE
\begin{align}\label{eq:1dequa}
    \partial_t f_{{\bm \lambda},t}&=\mathcal{B}(t)f_{{\bm \lambda},t}, \\\label{eq:1dkernel}
    \mathcal{B}(t)f_{{\bm \lambda},t}&=\sum_{i\ne j} c_{ij}(t) 2\eta_i(1+2\eta_j)\big(f_{{\bm \lambda},t}({\bm \eta}^{kl})-f_{{\bm \lambda},t}({\bm \eta})\big).
\end{align}
where 
\begin{equation}\label{eq:defc}
    c_{ij}(t):= \frac{1}{N(\lambda_i(t) -  \lambda_j(t))^2}.
\end{equation}
Note that the number of particles $n=n({\bm \eta})$ is preserved under the flow \eqref{eq:1dequa}. The eigenvalue trajectories are fixed in this proof, hence we will often omit ${\bm\lambda}$ from the notation, e.g. we will use $f_t=f_{{\bm \lambda}, t}$, and so on.

The main technical input in the proof of Theorem~\ref{theo:flucque} is  the following result (cf. \cite[Proposition 3.2]{MR4413210}):

\begin{proposition}\label{pro:flucque}
    For any \(n\in\mathbf{N}\) there exists \(c(n)>0\) such that for any \(\epsilon>0\), and for any \(T\ge N^{-1+\epsilon}\) it holds
    \begin{equation}\label{eq:mainbthissec}
        \sup_{{\bm \eta}}\big|f_T({\bm\eta})-\bm1(n\,\, \mathrm{even})\big|\lesssim N^{-c(n)},
    \end{equation}
    with very high probability, where the supremum is taken over configurations \({\bm \eta} \in \Omega^n \) 
    supported in the bulk, i.e. such that
    \(\eta_i=0\) for \(i\notin [\delta N, (1-\delta) N]\), with \(\delta>0\) from Theorem~\ref{theo:flucque}. The implicit constant in~\eqref{eq:mainbthissec} depends on \(n\), \(\epsilon\), \(\delta\).
\end{proposition}

We are now ready to prove Theorem~\ref{theo:flucque}.

\begin{proof}[Proof of Theorem~\ref{theo:flucque}]
    
    Fix  \(i\in [\delta N,(1-\delta) N]\), then the convergence in \eqref{eq:clt} follows immediately from \eqref{eq:mainbthissec} choosing \({\bm \eta}\) to be the configuration with \(\eta_i=n\) and all other \(\eta_j=0\), together with a standard application of the Green function comparison theorem (GFT), relating the eigenvectors/eigenvalues of \(W_T\) to those of \(W\); see Appendix~\ref{sec:GFT} where we recall the GFT argument for completeness. We defer the interested reader to \cite[Proof of Theorem 2.2]{MR4413210} for a more detailed proof.
    
\end{proof}

\subsection{DBM analysis}

Since the current DBM analysis of~ \eqref{eq:1dequa}
heavily relies on \cite[Section 4]{MR4413210},  before starting it we introduce an equivalent representation of \eqref{eq:deff}
used in \cite{MR4413210} (which  itself is based on the particles representation from \cite{MR4272266}).

Fix $n\in\mathbf{N}$, and consider configurations ${\bm \eta}\in\Omega^n$, i.e. such that $\sum_j\eta_j=n$. We now give an equivalent representation of \eqref{eq:1dequa}--\eqref{eq:1dkernel} which is defined on the $2n$-dimensional lattice $[N]^{2n}$ instead of configurations of $n$ particles (see \cite[Section 4.1]{MR4413210} for a more detailed description).
Let ${\bm x}\in [N]^{2n}$ and define the configuration space
\begin{equation}
    \Lambda^n:= \{ {\bm x}\in [N]^{2n} \, : \,\mbox{\(n_i({\bm x})\) is even for every \(i\in [N]\)} \big\},
\end{equation}
where
\begin{equation}
    n_i({\bm x}):=|\{a\in [2n]:x_a=i\}|
\end{equation}
for all $i\in \mathbf{N}$. 

The correspondence between these two representations is given by 
\begin{equation}\label{xeta}
    {\bm \eta} \leftrightarrow {\bm x}\qquad \eta_i=\frac{n_i( {\bm x})}{2}.
\end{equation}
Note that \({\bm x}\) uniquely determines \({\bm \eta}\), but \({\bm \eta}\) determines only the coordinates
of \({\bm x}\) as a multi-set and not its ordering. Let \(\phi\colon\Lambda^n\to \Omega^n\), \(\phi({\bm x})={\bm \eta}\) 
be the projection from the \({\bm x}\)-configuration space to the \({\bm \eta}\)-configuration space using~\eqref{xeta}. We will then always consider
functions \(g\) on \([N]^{2n}\) that are push-forwards of some function \(f\) on \(\Omega^n\), 
\(g= f\circ \phi\), i.e.
they correspond to functions on the configurations
\[
f({\bm \eta})= f(\phi({\bm x}))= g({\bm x}).
\]
In particular \(g\) is supported on \(\Lambda^n\) and it is equivariant under permutation of the 
arguments, i.e.\ it
depends on \({\bm x}\) only as a multiset. 
We  thus consider the observable
\begin{equation}\label{eq:defg}
    g_t({\bm x})=g_{{\bm \lambda},t}({\bm x}):= f_{{\bm \lambda},t}( \phi({\bm x}))
\end{equation}
where \( f_{{\bm \lambda},t}\) was defined in~\eqref{eq:deff}.

Using the ${\bm x}$-representation space, we can now write the flow~\eqref{eq:1dequa}--\eqref{eq:1dkernel} as follows:
\begin{align}\label{eq:g1deq}
    \partial_t g_t({\bm x})&=\mathcal{L}(t)g_t({\bm x}) \\\label{eq:g1dker}
    \mathcal{L}(t):=\sum_{j\ne i}\mathcal{L}_{ij}(t), \quad \mathcal{L}_{ij}(t)g({\bm x}):&= c_{ij}(t) \frac{n_j({\bm x})+1}{n_i({\bm x})-1}\sum_{a\ne b\in[2 n]}\big(g({\bm x}_{ab}^{ij})-g({\bm x})\big),
\end{align}
where
\begin{equation}
    \label{eq:jumpop}
    {\bm x}_{ab}^{ij}:={\bm x}+\delta_{x_a i}\delta_{x_b i} (j-i) ({\bm e}_a+{\bm e}_b),
\end{equation} 
with \({\bm e}_a(c)=\delta_{ac}\), \(a,c\in [2n]\). This flow is  map  on functions 
defined on \(\Lambda^n\subset [N]^{2n}\) and it preserves equivariance.

We now define the scalar product and the natural measure on $\Lambda^n$:
\begin{equation}\label{eq:scalpro}
    \braket{f, g}_{\Lambda^n}=\braket{f, g}_{\Lambda^n, \pi}:=\sum_{{\bm x}\in \Lambda^n}\pi({\bm x})\bar f({\bm x})g({\bm x}),
    \qquad \pi({\bm x}):=\prod_{i=1}^N ((n_i({\bm x})-1)!!)^2,
\end{equation}
as well as the norm on \(L^p(\Lambda^n)\):
\begin{equation}
    \norm{f}_p=\norm{f}_{L^p(\Lambda^n,\pi)}:=\left(\sum_{{\bm x}\in \Lambda^n}\pi({\bm x})|f({\bm x})|^p\right)^{1/p}.
\end{equation}

By~\cite[Appendix A.2]{MR4272266}
it follows that the operator \(\mathcal{L}=\mathcal{L}(t)\)  is symmetric with respect to the measure \(\pi\)
and it is a negative operator on \(L^2(\Lambda^n)\) with Dirichlet form
\[
D(g)=\braket{g, (-\mathcal{L}) g}_{\Lambda^n} = \frac{1}{2}  \sum_{{\bm x}\in \Lambda^n}\pi({\bm x})
\sum_{i\ne j} c_{ij}(t) \frac{n_j({\bm x})+1}{n_i({\bm x})-1}
\sum_{a\ne b\in[2 n]}\big|g({\bm x}_{ab}^{ij})-g({\bm x})\big|^2.
\]
Let \(\mathcal{U}(s,t)\) be
the semigroup associated to \(\mathcal{L}\), i.e.\ for any \(0\le s\le t\) it holds
\[
\partial_t\mathcal{U}(s,t)=\mathcal{L}(t)\mathcal{U}(s,t), \quad \mathcal{U}(s,s)=I.
\]

\subsubsection{Short range approximation}

Most of our DBM analysis will be completely local, hence we will introduce a short range approximation $h_t$ (see its definition in \eqref{g-1} below) of
$g_t$ that will be exponentially small evaluated on ${\bm x}$'s which are not fully supported in the bulk.

Recall the definition of the quantiles $\gamma_i(0)$ from \eqref{eq:quantin}, then we define the sets
\begin{equation}\label{eq:defintJ}
    \mathcal{J}=\mathcal{J}_\delta:=\{ i\in [N]:\, \gamma_i(0)\in \mathcal{I}_\delta\}, \qquad \mathcal{I}_\delta:=(-2+\delta,2-\delta),
\end{equation}
which correspond to indices and spectral range in the bulk, respectively. From now on we fix a point ${\bm y}\in \mathcal{J}$, and an $N$-dependent parameter $K$ such that $1\ll K\le \sqrt{N}$. 
Next, we define the  \emph{averaging operator} as a simple multiplication operator by a ``smooth'' cut-off function:
\begin{equation}
    \Av(K,{\bm y})h({\bm x}):=\Av({\bm x};K,{\bm y})h({\bm x}), \qquad \Av({\bm x}; K, {\bm y}):=\frac{1}{K}\sum_{j=K}^{2K-1} \bm1(\norm{{\bm x}-{\bm y}}_1<j),
\end{equation}
with \(\norm{{\bm x}-{\bm y}}_1:=\sum_{a=1}^{2n} |x_a-y_a|\). Additionally, fix an  integer \(\ell\) with \(1\ll\ell\ll K\),
and define the short range coefficients
\begin{equation}\label{eq:ccutoff}
    c_{ij}^{\mathcal{S}}(t):=\begin{cases}
        c_{ij}(t) &\mathrm{if}\,\, i,j\in \mathcal{J} \,\, \mathrm{and}\,\, |i-j|\le \ell \\
        0 & \mathrm{otherwise},
    \end{cases}
\end{equation}
where \(c_{ij}(t)\) is defined in~\eqref{eq:defc}. The parameter $\ell$ is the length of the short range interaction.

The short range approximation $h_t=h_t({\bm x})$ of $g_t$  is defined as the unique solution of the parabolic equation
\begin{equation}\label{g-1}
    \begin{split}
        \partial_t h_t({\bm x}; \ell, K,{\bm y})&=\mathcal{S}(t) h_t({\bm x}; \ell, K,{\bm y})\\
        h_0({\bm x};\ell, K,{\bm y})=h_0({\bm x};K,{\bm y}):&=\Av({\bm x}; K,{\bm y})(g_0({\bm x})-\bm1(n \,\, \mathrm{even})), 
    \end{split}
\end{equation}
where
\begin{equation}\label{g-2}
    \mathcal{S}(t):=\sum_{j\ne i}\mathcal{S}_{ij}(t), \quad \mathcal{S}_{ij}(t)h({\bm x}):=c_{ij}^{\mathcal{S}}(t)\frac{n_j({\bm x})+1}{n_i({\bm x})-1}\sum_{a\ne b\in [2n]}\big(h({\bm x}_{ab}^{ij})-h({\bm x})\big).
\end{equation}
Since $K$, ${\bm y}$ and $\ell$  are fixed for the rest of this section we will often omit them from the notation. We conclude this section defining the transition semigroup \(\mathcal{U}_{\mathcal{S}}(s,t)=\mathcal{U}_{\mathcal{S}}(s,t;\ell)\) associated to the short range generator $\mathcal{S}(t)$.

\subsubsection{$L^2$-bound}

By standard finite speed propagation estimates (see \cite[Proposition 4.2, Lemmata 4.3--4.4]{MR4413210}), we conclude that
\begin{lemma}\label{lem:shortlongapprox}
    Let \(0\le s_1\le s_2\le s_1+\ell N^{-1}\), and \(f\) be a function on \(\Lambda^n\), 
    then for any \({\bm x}\in \Lambda^n\) supported on \(\mathcal{J}\) it holds
    \begin{equation}\label{eq:shortlong}
        \Big| (\mathcal{U}(s_1,s_2)-\mathcal{U}_{\mathcal{S}}(s_1,s_2;\ell) ) f({\bm x}) \Big|\lesssim N^{1+n\xi}\frac{s_2-s_1}{\ell} \| f\|_\infty,
    \end{equation}
    for any small \(\xi>0\). The implicit constant in~\eqref{eq:mainbthissec} depends on \(n\), \(\epsilon\), \(\delta\).
\end{lemma}
In particular, this lemma shows that the observable $g_t$ and its short-range approximation $h_t$ are close to each other up to times $t\ll \ell/N$, hence to prove Proposition~\ref{pro:flucque} will be enough to estimate $h_t$. First in Proposition~\ref{prop:mainimprov} below will prove a bound in $L^2$-sense that will be enhanced to
an $L^\infty$ bound by standard parabolic regularity arguments.

Define the event \(\widehat{\Omega}\) on which the local laws for certain products of resolvents and traceless matrices \(A\) hold, i.e.\ for a small \(\omega>2\xi>0\) we define
\begin{equation}  
    \label{eq:hatomega}
    \begin{split}
       & \widehat{\Omega}=\widehat{\Omega}_{\omega, \xi} \\
       :&=\bigcap_{\substack{z_i: \Re z_i\in [-3,3], \atop |\Im z_i|\in [N^{-1+\omega},10]}}\Bigg[\bigcap_{k=2}^n \left\{\sup_{0\le t \le T}(\rho_t^*)^{-1/2}\left|\braket{G_t(z_1)A\dots G_t(z_k)A}-\braket{A^k}\prod_{i=1}^km_t(z_i)\right|\le \frac{N^{\xi+k/2-1}\braket{A^2}^{k/2}}{\sqrt{N\eta_*}}\right\} \\
        &\quad\cap \left\{\sup_{0\le t \le T}(\rho_{1,t})^{-{1/2}}\big|\braket{G_t(z_1)A}\big|\le \frac{N^\xi\braket{A^2}^{1/2}}{N\sqrt{|\Im z_1|}}\right\}\Bigg],
    \end{split}
\end{equation}
where \(\eta_*:=\min\set[\big]{|\Im z_i|\given i\in[k]}\), $\rho_{i,t}:=|\Im m_t(z_i)|$, and $\rho_t^*:=\max_i\rho_{i,t}$. Theorem~\ref{theorem multi G local law} shows that \(\widehat{\Omega}\) is a very high probability event, by using standard grid argument
for the spectral parameters and stochastic continuity in the time parameter. Note that by rigidity \eqref{def:Omega} and the spectral theorem we have (recall the definition of $\gamma_i(0)$ from \eqref{eq:quantin}):
\begin{equation}
\begin{split}
\label{eq:specdec}
&(\rho_t^*)^{-1} \braket{\Im G_t(\gamma_{i_1}(t)+\ii\eta_1) A\Im G_t(\gamma_{i_2}(t)+\ii\eta_2) A} \\
&\qquad\quad=\frac{1}{N\rho_t^*}\sum_{i,j=1}^N\frac{\eta^2|\braket{{\bm u}_i(t), A {\bm u}_j(t)}|^2}{((\lambda_i(t)-\gamma_{i_1}(t))^2+\eta_1^2)((\lambda_i(t)-\gamma_{i_2}(t))^2+\eta_2^2)}\\
&\qquad\quad\ge \frac{|\braket{{\bm u}_{i_1}(t), A {\bm u}_{i_2}(t)}|^2}{N\eta_1\eta_2\rho_t^*}\\
&\qquad\quad= \frac{N\big[\rho(\gamma_{i_1}(t)+\ii N^{-2/3})\wedge\rho(\gamma_{i_2}(t)+\ii N^{-2/3})\big]\cdot|\braket{{\bm u}_{i_1}(t), A {\bm u}_{i_2}(t)}|^2}{N^2\eta_1\eta_2\rho_t^*\big[\rho(\gamma_{i_1}(t)+\ii N^{-2/3})\wedge\rho(\gamma_{i_2}(t)+\ii N^{-2/3})\big]} \\
&\qquad\quad= N^{1-2\omega}\big[\rho(\gamma_{i_1}(t)+\ii N^{-2/3})\wedge\rho(\gamma_{i_2}(t)+\ii N^{-2/3})\big]\cdot|\braket{{\bm u}_{i_1}(t), A {\bm u}_{i_2}(t)}|^2
\end{split}
\end{equation}
with $\eta_k=\eta_k(t)$ defined by $N\eta_k\rho(\gamma_{i_k}(t)+\ii N^{-2/3})=N^\omega$. In particular, since \( |\Im m_t(z_1) \Im m_t(z_2)|\lesssim \rho(z_1)\rho(z_2)\), by the first line of \eqref{eq:hatomega} for $k=2$ we have
\[
\sup_{0\le t\le T}\sup_{z_1, z_2}(\rho_t^*)^{-1} \braket{\Im G_t(z_1) A\Im G_t(z_2) A}\lesssim \braket{A^2},
\] 
on  \(\widehat{\Omega}_{\omega,\xi}\), which by \eqref{eq:specdec}, choosing $z_k=\gamma_{i_k}(t)+\ii\eta$, implies
\begin{equation}\label{eq:apriori}
    |\braket{{\bm u}_i(t), A {\bm u}_j(t)}|^2\le \frac{N^{2\omega} \braket{A^2}}{N[\rho(\gamma_i(t)+\ii N^{-2/3})\wedge \rho(\gamma_j(t)+\ii N^{-2/3})]} \qquad \,\,
    \mbox{on \(\;\widehat{\Omega}_{\omega,\xi}\)}\cap\widetilde{\Omega}_\xi,
\end{equation}
simultaneously for all $i,j\in [N]$ and $0\le t\le T$. We recall that the quantiles $\gamma_i(t)$ are defined in \eqref{eq:quantin}.

\begin{remark}\label{rmk:Omega}
    The set $\widehat{\Omega}$ defined in \eqref{eq:hatomega} is slightly different from its analogue\footnote{ 
    The definition of $\widehat\Omega$ in the published version of~\cite[Eq. (4.20)]{MR4413210}
    contained a small error; the constraints were formally 
    restricted only to spectral parameters in the bulk, even though the necessary
    bounds were directly available at the edge as well. This slightly imprecise formulation
    is corrected in the latest arXiv version of~\cite{MR4413210}; Remark~\ref{rmk:Omega} refers to the corrected version.}
    in
    \cite[Eq. (4.20)]{MR4413210}. First, all the error terms now explicitly depend on $\braket{A^2}$, whilst in \cite[Eq. (4.20)]{MR4413210} we just bounded the error terms using the operator norm of $A$ (which was smaller than $1$ in \cite[Eq. (4.20)]{MR4413210}). Second, we now have a slightly weaker bound (compared to \cite[Eq. (4.20)]{MR4413210}) for $\braket{ \Im G_t(z_1)A\Im G_t(z_2)A}-\Im m_t(z_1) \Im m_t(z_2)\braket{A^2}$, since we now do not carry the dependence on the $\rho_{i,t}$'s optimally; as a consequence of this slightly worse bound close to the edges we get the overlap bound \eqref{eq:apriori}, instead of the optimal bound \cite[Eq. (4.21)]{MR4413210}, however this difference will not cause any change in the result. We remark that the bound \eqref{eq:apriori} is optimal for bulk indices.
\end{remark}

\begin{proposition}\label{prop:mainimprov}
    For any parameters satisfying \(N^{-1}\ll \eta\ll T_1\ll \ell N^{-1}\ll K N^{-1}\), and any small \(\epsilon, \xi>0\) it holds
    \begin{equation}\label{eq:l2b}
        \norm{h_{T_1}(\cdot; \ell, K, {\bm y})}_2\lesssim K^{n/2}\mathcal{E},
    \end{equation}
    with 
    \begin{equation}\label{eq:basimpr}
        \mathcal{E}:= N^{n\xi}\left(\frac{N^\epsilon\ell}{K}+\frac{NT_1}{\ell}+\frac{N\eta}{\ell}+\frac{N^\epsilon}{\sqrt{N\eta}}+\frac{1}{\sqrt{K}}\right),
    \end{equation}
    uniformly for particle configuration \({\bm y}\in \Lambda^n\) supported on \(\mathcal{J}\) and eigenvalue trajectory \({\bm \lambda}\) 
    in the high probability event \(\widetilde{\Omega}_\xi \cap \widehat{\Omega}_{\omega,\xi}\).
\end{proposition}
\begin{proof}
    
    This proof  is very similar to that of \cite[Proposition 4.5]{MR4413210}, hence
    we will only explain the main differences. The reader should consult with \cite{MR4413210} for  a fully detailed proof.
    The key idea  is to replace the operator $\mathcal{S}(t)$ in \eqref{g-1}--\eqref{g-2},
    by the following operator 
    \begin{equation}\label{eq:defAgen}
        \mathcal{A}(t):=\sum_{{\bm i}, {\bm j}\in [N]^n}^*\mathcal{A}_{ {\bm i}{\bm j}}(t), \quad \mathcal{A}_{{\bm i}{\bm j}}(t)h({\bm x}):=\frac{1}{\eta}\left(\prod_{r=1}^n a_{i_r,j_r}^\mathcal{S}(t)\right)\sum_{{\bm a}, {\bm b}\in [2n]^n}^*(h({\bm x}_{{\bm a}{\bm b}}^{{\bm i}{\bm j}})-h({\bm x})),
    \end{equation}
    where
    \begin{equation}
        a_{ij}=a_{ij}(t):=\frac{\eta}{N((\lambda_i(t)-\lambda_j(t))^2+\eta^2)},
    \end{equation}
    and \(a_{ij}^\mathcal{S}\) are their short range version defined as in~\eqref{eq:ccutoff}, and
    \begin{equation}
        \label{eq:lotsofjumpsop}
        {\bm x}_{{\bm a}{\bm b}}^{{\bm i}{\bm j}}:={\bm x}+\left(\prod_{r=1}^n \delta_{x_{a_r}i_r}\delta_{x_{b_r}i_r}\right)\sum_{r=1}^n (j_r-i_r) ({\bm e}_{a_r}+{\bm e}_{b_r}).
    \end{equation}
    We remark that  ${\bm x}^{ij}_{ab}$ from~\eqref{eq:jumpop} 
    changes two entries of ${\bm x}$ per time, instead  ${\bm x}_{{\bm a}{\bm b}}^{{\bm i}{\bm j}}$ 
    changes all the coordinates of ${\bm x}$ at the same time, i.e. let ${\bm i}:=(i_1,\dots, i_n), {\bm j}:=(j_1,\dots, j_n)\in [N]^n$, with $\{i_1,\dots,i_n\}\cap\{j_1,\dots, j_n\}=\emptyset$, then ${\bm x}_{{\bm a}{\bm b}}^{{\bm i}{\bm j}}\ne {\bm x}$ iff for all $r\in [n]$ it holds that $x_{a_r}=x_{b_r}=i_r$. This means that   $\mathcal{S}(t)$ makes a jump only in one direction at a time, while
    $\mathcal{A}(t)$ jumps in all directions simultaneously. 
    Technically, the replacement of $\mathcal{S}(t)$ by $\mathcal{A}(t)$ is done on the level of Dirichlet forms:
    \begin{lemma}[Lemma 4.6 of \cite{MR4413210}]
        \label{lem:replacement}
        Let \(\mathcal{S}(t)\), \(\mathcal{A}(t)\) be the generators defined in~\eqref{g-2} and~\eqref{eq:defAgen}, respectively,
        and let $\mu$ denote the uniform measure on $\Lambda^n$ for which  \(\mathcal{A}(t)\) is reversible.
        Then there exists a constant \(C(n)>0\) such that
        \begin{equation}\label{eq:fundbound}
            \braket{h, \mathcal{S}(t) h}_{\Lambda^n, \pi}\le C(n) \braket{h, \mathcal{A}(t) h}_{\Lambda^n,\mu}\le 0,
        \end{equation}
        for any \(h\in L^2(\Lambda^n)\), on the very high probability set \(\widetilde{\Omega}_\xi \cap \widehat{\Omega}_{\omega,\xi}\).
    \end{lemma}
    
    Next, combining
    \begin{equation}\label{eq:l2der}
        \partial_t \norm{h_t}_2^2=2\braket{h_t, \mathcal{S}(t) h_t}_{\Lambda^n},
    \end{equation}
    which follows from~\eqref{g-1},
    with \eqref{eq:fundbound}, and using that \({\bm x}_{{\bm a}{\bm b}}^{{\bm i}{\bm j}}={\bm x}\) unless \({\bm x}_{a_r}={\bm x}_{b_r}=i_r\) for all \(r\in [n]\), we conclude that
    \begin{equation}\label{eq:boundderl2}
        \begin{split}
            \partial_t \norm{h_t}_2^2&\le C(n) \braket{h_t,\mathcal{A}(t) h_t}_{\Lambda^n,\mu} \\
            &=\frac{C(n) }{2\eta}\sum_{{\bm x}\in \Lambda^n}\sum_{{\bm i},{\bm j}\in [N]^n}^* \left(\prod_{r=1}^n a_{i_r j_r}^\mathcal{S}(t)\right)\sum_{{\bm a}, {\bm b}\in [2n]^n}^*\overline{h_t}({\bm x})\big(h_t({\bm x}_{{\bm a}{\bm b}}^{{\bm i} {\bm j}})-h_t({\bm x})\big)\left(\prod_{r=1}^n \delta_{x_{a_r}i_r}\delta_{x_{b_r}i_r}\right).
        \end{split}
    \end{equation}
    The star over $\sum$ means summation over two $n$-tuples of fully distinct indices.
    Then, proceeding as in the proof of \cite[Proposition 4.5]{MR4413210}, we conclude that
    \begin{equation}
        \label{eq:finb}
        \partial_t\norm{h_t}_2^2\le -\frac{C_1(n)}{2\eta}\norm{h_t}_2^2+\frac{C_3(n)}{\eta}\mathcal{E}^2K^n,
    \end{equation}
    which implies \(\norm{h_{T_1}}_2^2\le C(n) \mathcal{E}^2 K^n\), by a simple Gronwall inequality, using that \(T_1\gg \eta\).

    We point out that to go from \eqref{eq:boundderl2} to \eqref{eq:finb} we proceed exactly as in the proof of \cite[Proposition 4.5]{MR4413210} (with the  additional $\braket{A^2}^{k/2}$, $\braket{A^2}^{n/2}$  factors in \cite[Eq. (4.47)]{MR4413210} and \cite[Eq. (4.48)]{MR4413210}, respectively) except for the estimate in \cite[Eq. (4.43)]{MR4413210}. The error terms in this estimate
    used that  $|P(G)|\le N^{n\xi-n/2}$ uniformly in the spectrum, a fact that  we cannot establish near the edges
    as a consequence of the weaker bound \eqref{eq:apriori}.
    We now explain how we can still prove \cite[Eq. (4.43)]{MR4413210} in the current case. 
    The main mechanism is that  the strong bound $|P(G)|\le N^{n\xi-n/2} \braket{A^2}^{n/2} $
    holds for bulk indices and when an edge index $j$ is involved together with a bulk index $i$, 
    then the kernel  $a_{ij}\lesssim \eta/N$  is
    very small which balances the weaker estimate on the overlap.  Note that \eqref{eq:apriori} still 
    provides a nontrivial bound of order $N^{-1/3}$ for $|\langle {\bm u}_i, A {\bm u}_j\rangle|$
    since $\rho(\gamma_i(t)+\ii N^{-2/3})\gtrsim N^{-1/3}$ uniformly in $0\le t\le T$.
    
    We start with removing the short range cutoff from the kernel $a_{ij}^\mathcal{S}(t)$
    in the left hand side of~\cite[Eq. (4.43)]{MR4413210}:
    \begin{equation}
        \begin{split}
            \label{eq:newb}
            &\sum_{{\bm j}}^*\left(\prod_{r=1}^n a_{i_r j_r}^\mathcal{S}(t)\right)\big(g_t({\bm x}_{{\bm a}{\bm b}}^{{\bm i} {\bm j}})-\bm1(n\,\, \mathrm{even})\big) \\
            &=\sum_{{\bm j}}^*\left(\prod_{r=1}^n a_{i_r j_r}(t)\right)\left(\frac{N^{n/2}}{\braket{A^2}^{n/2} 2^{n/2}(n-1)!!}\sum_{G\in \mathcal{G}_{{\bm \eta}^{{\bm j}}}}P(G)-\bm1(n\,\, \mathrm{even})\right) \\
            &\quad -\sum_{{\bm j}}^{**}\left(\prod_{r=1}^n a_{i_r j_r}(t)\right)\left(\frac{N^{n/2}}{\braket{A^2}^{n/2} 2^{n/2}(n-1)!!}\sum_{G\in \mathcal{G}_{{\bm \eta}^{{\bm j}}}}P(G)-\bm1(n\,\, \mathrm{even})\right).
        \end{split}
    \end{equation}
    Here $\sum_{{\bm j}}^{**}$ denotes the sum over distinct $j_1,\dots,j_n$ such that at least one $|i_r-j_r|$ is bigger than $\ell$.
    
    Here the indices $i_1,\dots,i_n$ are fixed and such that $i_l\in [\delta N,(1-\delta)N]$, for any $l\in [n]$. We will now show that the second line in \eqref{eq:newb} is estimated by $N^{1+n\xi}\eta\ell^{-1}$. This is clear for the terms containing $\bm1(n\,\,\mathrm{even})$, hence we now show that this bound is also valid for the terms containing $P(G)$. We present this bound only for the case when $|j_1-i_1|>\ell$ and $|j_r-i_r|\le \ell$ for any $r\in\{2,\dots,n\}$. The proof in the other cases is completely analogous and so omitted. Additionally, to make our presentation easier we assume that $n=2$:
    \begin{equation}
        \begin{split}
            \label{eq:firststep11}
            &\sum_{ |j_1-i_1|>\ell, |j_2-i_2|\le \ell, \atop j_1\ne j_2} a_{i_1 j_2}(t)a_{i_1 j_2}(t)\left(\frac{N}{2\braket{A^2}}\sum_{G\in \mathcal{G}_{{\bm \eta}^{{\bm j}}}}P(G)\right) \\
            &=\left(\sum_{ cN\ge |j_1-i_1|>\ell, |j_2-i_2|\le \ell, \atop j_1\ne j_2}+\sum_{ |j_1-i_1|>cN, |j_2-i_2|\le \ell, \atop j_1\ne j_2}\right) a_{i_1 j_2}(t)a_{i_1 j_2}(t)\left(\frac{N}{2\braket{A^2}}\sum_{G\in \mathcal{G}_{{\bm \eta}^{{\bm j}}}}P(G)\right).
        \end{split}
    \end{equation}
    Here $c\le \delta/2$ is a small fixed constant  so that $j_1$ is still a bulk index if 
    $|i_1-j_1|\le cN$.
    The fact that the first summation in the second line of \eqref{eq:firststep11} is bounded by $N^{1+n\xi}\eta\ell^{-1}$ follows from~\eqref{eq:apriori}, i.e.  that $|\braket{{\bm u}_i,A{\bm u}_j}|\le N^{-1/2+\omega} \braket{A^2}^{1/2} $, with very high probability, for any bulk indices $i,j$, 
    in particular the bound $|P(G)|\le N^{n\xi-n/2} \braket{A^2}^{n/2} $ holds for this term.
    For the second summation we have that  
    \begin{equation}
        \label{eq:secondstep11}
        \begin{split}
            \sum_{ |j_1-i_1|>cN, |j_2-i_2|\le \ell, \atop j_1\ne j_2} a_{i_1 j_1}(t)a_{i_2 j_2}(t)\left(\frac{N}{2\braket{A^2}}\sum_{G\in \mathcal{G}_{{\bm \eta}^{{\bm j}}}}P(G)\right)&\lesssim \frac{N^{1+\xi} \eta}{N^{2/3}} \sum_{ |j_2-i_2|\le \ell} a_{i_2 j_2}(t) \\
            &\lesssim \frac{N^{1+\xi}\eta}{N^{2/3}}\le \frac{N\eta}{\ell},
        \end{split}
    \end{equation}
    where we used that $a_{i_1 j_1}(t)\lesssim \eta N^{-1}$, $\ell\ll K\ll \sqrt{N}$, and that
    \[
    |P(G)|=\big|\braket{{\bm u}_{j_1},A{\bm u}_{j_1}}\braket{{\bm u}_{j_2},A{\bm u}_{j_2}}+2|\braket{{\bm u}_{j_1},A{\bm u}_{j_2}}|^2\big|\lesssim \frac{N^\xi}{N^{2/3}} \braket{A^2}
    \]
    by \eqref{eq:apriori}.
    We point out that to go from the first to the second line of \eqref{eq:secondstep11} we also used that $\sum_{j_2}a_{i_2 j_2}(t)\lesssim 1$ on $\widehat{\Omega}$. This concludes the proof that the last line of \eqref{eq:newb} is bounded by $N^{1+n\xi}\eta\ell^{-1}$. We thus conclude that
    \begin{equation}
        \begin{split}
            \label{eq:1}
            &\sum_{{\bm j}}^*\left(\prod_{r=1}^n a_{i_r j_r}^\mathcal{S}(t)\right)\big(g_t({\bm x}_{{\bm a}{\bm b}}^{{\bm i} {\bm j}})-\bm1(n\,\, \mathrm{even})\big) \\
            &=\sum_{{\bm j}}^*\left(\prod_{r=1}^n a_{i_r j_r}(t)\right)\left(\frac{N^{n/2}}{\braket{A^2}^{n/2} 2^{n/2}(n-1)!!}\sum_{G\in \mathcal{G}_{{\bm \eta}^{{\bm j}}}}P(G)-\bm1(n\,\, \mathrm{even})\right)+\mathcal{O}\left(\frac{N^{1+n\xi}\eta}{\ell}\right).
        \end{split}
    \end{equation}
    
    Proceeding in a similar way, i.e. splitting bulk and edge regimes and using the corresponding bounds for the overlaps, we then add back the missing indices in the summation in the second line of \eqref{eq:1}:
    \begin{equation}
        \begin{split}
            \label{eq:2}
            &\sum_{{\bm j}}^*\left(\prod_{r=1}^n a_{i_r j_r}(t)\right)\left(\frac{N^{n/2}}{\braket{A^2}^{n/2} 2^{n/2}(n-1)!!}\sum_{G\in \mathcal{G}_{{\bm \eta}^{{\bm j}}}}P(G)-\bm1(n\,\, \mathrm{even})\right) \\
            &\quad= \sum_{{\bm j}}\left(\prod_{r=1}^n a_{i_r j_r}(t)\right)\left(\frac{N^{n/2}}{\braket{A^2}^{n/2} 2^{n/2}(n-1)!!}\sum_{G\in \mathcal{G}_{{\bm \eta}^{{\bm j}}}}P(G)-\bm1(n\,\, \mathrm{even})\right) +\mathcal{O}\left(\frac{N^{n\xi}}{N\eta}\right).
        \end{split}
    \end{equation}

    Finally, by \eqref{eq:1}--\eqref{eq:2}, we conclude
    \begin{equation}
        \begin{split}
            \label{eq:3}
            &\sum_{{\bm j}}^*\left(\prod_{r=1}^n a_{i_r j_r}^\mathcal{S}(t)\right)\big(g_t({\bm x}_{{\bm a}{\bm b}}^{{\bm i} {\bm j}})-\bm1(n\,\, \mathrm{even})\big) \\
            &= \sum_{{\bm j}}\left(\prod_{r=1}^n a_{i_r j_r}(t)\right)\left(\frac{N^{n/2}}{\braket{A^2}^{n/2} 2^{n/2}(n-1)!!}\sum_{G\in \mathcal{G}_{{\bm \eta}^{{\bm j}}}}P(G)-\bm1(n\,\, \mathrm{even})\right) +\mathcal{O}\left(\frac{N^{n\xi}}{N\eta}+\frac{N^{1+n\xi}\eta}{\ell}\right),
        \end{split}
    \end{equation}
    which is exactly the same of \cite[Eq. (4.43)]{MR4413210}. Given \eqref{eq:3}, the remaining part of the proof of this proposition is completely analogous to the proof of \cite[Proposition 4.5]{MR4413210}, the only difference is that now in \cite[Eq. (4.48)]{MR4413210}, using that $|m_t(z_i)|\lesssim 1$ uniformly in $0\le t\le T$, we will have an additional error term
    \[
    \begin{split}
        \frac{N^{n/2}}{\braket{A^2}^{n/2}}\sum_{r=1}^n\sum_{k_1+\dots+k_r=n}^*\prod_{i=1}^r N^{1-k_i}\braket{A^{k_i}}\lesssim \frac{N^{n/2}}{\braket{A^2}^{n/2}}\sum_{r=1}^n\sum_{k_1+\dots+k_r=n}^*\prod_{i=1}^r N^{-k_i/2}N^{-\delta'(k_i/2-1)}\braket{A^2}^{k_i/2}\lesssim N^{-\delta'}
    \end{split}
    \]
    coming from the deterministic term in \eqref{eq:hatomega} (the mixed terms when we use the error term in  \eqref{eq:hatomega} for some terms and the leading term for the remaining terms are estimated in the same way). We remark that in the first inequality we used that
    \[
    \braket{A^{k_i}}\le \norm{A}^{k_i-2}\braket{A^2}\lesssim \big(N^{1-\delta'}\big)^{(k_i-2)/2}\braket{A^2}^{k_i/2}
    \]
    by our assumption $\braket{A^2}\gtrsim N^{-1+\delta'}\norm{A}^2$ from Theorem~\ref{theo:flucque}. Here $\sum_{k_1+\dots+k_r=n}^*$ denotes the summation over all $k_1,\dots, k_r\ge 2$ such that there exists at least one $r_0$ such that $k_{r_0}\ge 3$.
\end{proof}

\subsubsection{Proof of Proposition~\ref{pro:flucque}}

Given the finite speed of propagation estimates in Lemma~\ref{lem:shortlongapprox} and the $L^2$-bound on $h_t$ from Proposition~\ref{prop:mainimprov} as an input, enhancing this bound to an $L^\infty$-bound
and hence proving Proposition~\ref{pro:flucque} is completely analogous to the proof of \cite[Proposition 3.2 ]{MR4413210} presented in \cite[Section 4.4]{MR4413210} and so omitted.

\appendix

\section{Proof of Theorem~\ref{theorem multi G local law}  in the large \(d\) regime}
\label{appendix large d}
The \(d\ge 10\) regime is much simpler mainly because the trivial norm bound $\| G(z)\|\le 1/d$ on every 
resolvent is affordable. In particular,  no system of master inequalities and their meticulously bootstrapped
analysis  are necessary; a simple induction on $k$ is sufficient.
We remark that the argument using these  drastic simplifications is completely analogous\footnote{We point out that the
$N$-scaling here is naturally different from that in~\cite[Appendix B]{MR4479913} simply due to the fact that here we chose
the normalization $\braket{|A_i|^2}=1$ instead of $\norm{A_i}=1$.}
  to~\cite[Appendix B]{MR4479913},
hence we will be very brief.

We now assume that~\eqref{loc1} has been proven up to some $k-1$ in the $d\ge 10$ regime.
Using~\eqref{selfcons} and estimating all resolvent chains in the right hand side of~\eqref{selfcons} by the induction hypotheses
(after splitting $A_kA_1 = \braket{A_kA_1} + (A_kA_1)^\circ$), using the analogue of Lemma~\ref{lemma Psi G prod}
to estimate $\braket{(GA)^{j-1}G}$ in terms of the induction hypothesis, 
 we easily obtain 
    \begin{equation}\label{GAd}
    \braket{(GA)^k-m^kA^k}\Bigl(1+\landauOprec*{\frac{1}{Nd^2}}\Bigr) = -m \braket{\un{W(GA)^k}} + \landauOprec*{\frac{N^{k/2-1}}{d^k}\frac{1}{Nd^2}}  
    \end{equation}
    in place of~\cref{underline repl lemma}.  In estimating the leading terms in~\eqref{selfcons}
     we used that $|m[z_1, z_k] - m(z_1)m(z_k) |\lesssim d^{-4}$.
      Note that  $N^{k/2-1}/d^k$ is the natural size of the leading deterministic term
    $\braket{m^k A^k}$ under the normalization $\braket{  |A|^2}=1$ and the small factor $1/Nd^2$ represents the
    smallness of the negligible error term.
        We now follow the argument in~\cref{sec loclaw proof} starting from~\eqref{GAM}.
    For the Gaussian term~\cref{firstline} we simply bound
    \begin{equation}
        \abs*{m\frac{\braket{(GA)^{2k}G}}{N^2}}\prec \frac{N^{k-3}}{d^{2k+2}} = \Bigl(\frac{N^{k/2-1}}{d^k} \frac{1}{\sqrt{N}d}\Bigr)^2
    \end{equation}
    indicating a gain of order $1/(\sqrt{N}d)$ over the natural size of the leading term in~\eqref{GAd}; this gives
    the main error term in~\eqref{loc1}.
    The modifications to the non-Gaussian terms~\cref{WGA cum exp}, i.e. the estimates
    of~\eqref{Xid} and~\eqref{Xiod} are similarly straightforward and left to the reader. 
    This completes the proof in the remaining $d\ge 10$ regime.

\section{Green function comparison}
\label{sec:GFT}

The Green function comparison argument is very similar to the one presented in \cite[Appendix A]{MR4413210}, hence we only explain the minor differences.

Consider the Ornstein-Uhlenbeck flow
\begin{equation}\label{eq:OUmain}
    \dif \widehat{W}_t=-\frac{1}{2}\widehat{W}_t \dif t+\frac{\dif \widehat{B}_t}{\sqrt{N}}, \qquad \widehat{W}_0=W,
\end{equation}
with \(\widehat{B}_t\) a real symmetric Brownian motion. Along the OU-flow~\eqref{eq:OUmain} the moments of the entries of $\widehat{W}_t$ remain constant, additionally, this flow adds a small Gaussian component to \(W\), so that for any fixed \(T\) we have
\begin{equation}\label{eq:imprel}
    \widehat{W}_T\stackrel{\mathrm{d}}{=} \sqrt{1-cT}\widetilde{W}+\sqrt{c T}U,
\end{equation}
with \(c=c(T)>0\) a constant very close to one as long as \(T\ll 1\), and \(U,\wt W\) being independent GOE/Wigner matrices. Now consider the solution of the flow~\eqref{eq:matdbm} \(W_t\) with initial condition \(W_0=\sqrt{1-cT}\wt W\), so that 
\begin{equation}\label{eq dist GFT}
    W_{cT}\stackrel{\mathrm{d}}{=}\widehat{W}_T. 
\end{equation}

\begin{lemma}\label{lem:GFT}
    Let \(\widehat{W}_t\) be the solution of~\eqref{eq:OUmain}, and let \(\widehat{{\bm u}}_i(t)\) be its eigenvectors. Then for any smooth test function \(\theta\) of at most polynomial growth, and any fixed \(\epsilon\in(0,1/2)\) there exists an \(\omega=\omega(\theta,\epsilon)>0\) 
    such that for any bulk index \(i\in [\delta N, (1-\delta)N]\) (with \(\delta>0\) from Theorem~\ref{theo:flucque}) 
    and \(t=N^{-1+\epsilon}\) it holds that 
    \begin{equation}\label{eq:GFT}
        \E\theta\left(\sqrt{\frac{N}{2\braket{A^2}}}\braket{\widehat{\bm u}_i(t),A\widehat{\bm u}_i(t)}\right)=\E\theta\left(\sqrt{\frac{N}{2\braket{A^2}}}\braket{\widehat{\bm u}_i(0),A\widehat{\bm u}_i(0)}\right)+\mathcal{O}\left(N^{-\omega}\right).
    \end{equation}
\end{lemma}

We now show how to conclude Theorem~\ref{theo:flucque} using the GFT result from Lemma~\ref{lem:GFT}. Choose \(T=N^{-1+\epsilon}\) and \(\theta(x)=x^n\)  for some integer \(n\in\N\), then we have
\begin{equation}\label{eq:firststep}
    \begin{split}
        \E\left[\sqrt{\frac{N}{2\braket{A^2}}}\braket{{\bm u}_i,A {\bm u}_i}\right]^n&=\E\left[\sqrt{\frac{N}{2\braket{A^2}}}\braket{\widehat{\bm u}_i(T),A\widehat{\bm u}_i(T)}\right]^n+\mathcal{O}\left(N^{-c}\right)\\
        &= \E\left[\sqrt{\frac{N}{2\braket{A^2}}}\braket{{\bm u}_i(cT),A{\bm u}_i(cT)}\right]^n +\mathcal{O}\left(N^{-c}\right)\\
        &= \bm1(n\,\, \mathrm{even})(n-1)!!+\mathcal{O}\left(N^{-c}\right),
    \end{split}
\end{equation} 
for some small \(c=c(n,\epsilon)>0\), with \({\bm u}_i,\wh{\bm u}_i(t),\bm u_i(t)\) being the eigenvectors of \(W,\wh W_t,W_t\), respectively. This concludes the proof of Theorem~\ref{theo:flucque}. Note that in~\eqref{eq:firststep} we used Lemma~\ref{lem:GFT} in the first 
step,~\eqref{eq dist GFT}
in the second step and~\eqref{eq:mainbthissec} for  ${\bm \eta}$ such that $\eta_i=n$ and $\eta_j=0$ for any $j\ne i$ in the third step, using that in distribution the eigenvectors of \(W_{cT}\) are equal to those of \(\wt W_{cT/(1-cT)}\) with \(\wt W_t\) being the solution to the DBM flow with initial condition \(\wt W_0=\wt W\). 

\begin{proof}[Proof of Lemma~\ref{lem:GFT}]
    
    The proof of this lemma is very similar to the proof of \cite[Appendix A]{MR4413210}.
    The differences come from the somewhat different local law. First, we now systematically carry the
    factor $\braket{A^2}$ instead of $\|A\|^2=1$ as in~\cite[Appendix A]{MR4413210}, but this is automatic. Second, since the current overlap
    bound~\eqref{eq:apriori} is somewhat weaker near the edge, we need to check that for  resolvents 
    with spectral parameters in the bulk this will make no essential difference. This is the main purpose
    of repeating the standard proof from~\cite[Appendix A]{MR4413210} in some details.
    
    As a consequence of the repulsion of the eigenvalues (level repulsion), as in~\cite[Lemma 5.2]{MR3034787}, to understand the overlap \(\sqrt{N}\braket{A^2}^{-1/2}\braket{{\bm u}_i,A {\bm u}_i}\) it is enough to understand functions of \(\sqrt{N}\braket{A^2}^{-1/2}\braket{\Im G(z)A}\) with \(\Im z\) slightly  below \(N^{-1}\), i.e.\ the local eigenvalue spacing. In particular, to prove~\eqref{eq:GFT} it is enough to show that
    \begin{equation}\label{eq:gftres}
        \sup_{E\in(-2+\delta,2-\delta)}\abs*{\E\theta(\sqrt{N}\braket{A^2}^{-1/2}\braket{\Im G_t(z)A})-\E\theta(\sqrt{N}\braket{A^2}^{-1/2}\braket{\Im G_0(z)A})}\lesssim N^{-\omega},
    \end{equation}
    for $t=N^{-1+\epsilon}$, \(z=E+\ii \eta\) for some \(\zeta>0,\omega>0\) and all \(\eta\ge N^{-1-\zeta}\), c.f.~\cite[Section 4]{MR4164858} and~\cite[Appendix A]{MR3606475}.
    
    To prove this we define
    \begin{equation}\label{eq:R}
        R_t:=\theta(\sqrt{N}\braket{A^2}^{-1/2}\braket{\Im G_t(z)A}),
    \end{equation}
    and then use It\^o's formula:
    \begin{equation}\label{eq:ito}
        \E \frac{\dif R_t}{\dif t}=\E\left[-\frac{1}{2}\sum_\alpha w_\alpha(t)\partial_\alpha R_t+\frac{1}{2}\sum_{\alpha,\beta}\kappa_t(\alpha,\beta)\partial_\alpha\partial_\beta R_t\right],
    \end{equation}
    where \(\alpha,\beta\in [N]^2\) are double indices, \(w_\alpha(t)\) are the entries of \(W_t\), and \(\partial_\alpha:=\partial_{w_\alpha}\). Here 
    \begin{equation}
        \kappa_t(\alpha_1,\dots,\alpha_l):=\kappa(w_{\alpha_1}(t), \dots, w_{\alpha_l}(t))
    \end{equation}
    denotes the joint cumulant  of \(w_{\alpha_1}(t), \dots, w_{\alpha_l}(t)\), with \(l\in \mathbf{N}\). Note that by~\eqref{moments} it follows that \(|\kappa_t(\alpha_1,\dots,\alpha_l)|\lesssim N^{-l/2}\) uniformly in \(t \ge 0\).
    
    By cumulant expansion we get 
    \begin{equation}
        \label{eq:thirdhigher}
        \E\frac{\dif R_t}{\dif t}=\sum_{l= 3}^R\sum_{\alpha_1,\ldots,\alpha_l}\kappa_t(\alpha_1,\ldots,\alpha_l)\E[\partial_{\alpha_1}\cdots\partial_{\alpha_l}R_t]+\Omega(R),
    \end{equation}
    where \(\Omega(R)\) is an error term, easily seen to be negligible as every additional derivative gains a further factor of \(N^{-1/2}\). Then to estimate \eqref{eq:thirdhigher} we realize that $\partial_{ab}$-derivatives of $\braket{\Im G A}$ result in 
    factors of the form $(GAG)_{ab}$, $(GAG)_{aa}$. For such factors we
    use that
    \begin{equation}
        \label{eq:impr}
        \begin{split}
            \big|(G_t(z_1)AG_t(z_2))_{ab}\big|&=\left|\sum_{ij}\frac{{\bm u}_i(a)
            \braket{{\bm u}_i, A{\bm u}_j}{\bm u}_j(b)}{(\lambda_i-z_1)(\lambda_j-z_2)}\right| \\
            &\lesssim N^{2/3+\xi}\braket{A^2}^{1/2}\left(\frac{1}{N}\sum_i \frac{1}{|\lambda_i-z_1|}\right)\left(\frac{1}{N}\sum_i \frac{1}{|\lambda_i-z_2|}\right)\\
            &\lesssim N^{2/3+\xi+2\zeta}\braket{A^2}^{1/2},
        \end{split}
    \end{equation}
    where we used that $\norm{\bm u_i}_\infty\lesssim N^{-1/2+\xi}$, $|\braket{{\bm u}_i, A{\bm u}_j}|\le N^{-1/3+\xi}$, for any $\xi>0$, uniformly in the spectrum by~\cite{MR2871147}, and Theorem~\ref{thm:eth},
     respectively. We remark that in \cite[Eq. (A.11)]{MR4413210} we could bound $(G_t(z_1)AG_t(z_2))_{ab}$ by $N^{1/2+\xi+2\zeta}$ as a consequence of the better bound on $|\braket{{\bm u}_i, A{\bm u}_j}|$ for indices close to the edge (however in \cite[Eq. (A.11)]{MR4413210} we did not have $\braket{A^2}^{1/2}$). While our estimate on $(GAG)_{ab}$ is now by a factor $N^{1/6}$ weaker, this is still sufficient to complete
    the Green function comparison argument.

    Indeed, using \eqref{eq:impr} and that $|(G_t)_{ab}|\le N^\zeta$, for any $\zeta>0$, we conclude that
    \begin{equation}\label{eq:derboundgft}
        \left|\partial_{\alpha_1}\dots \partial_{\alpha_l}\frac{\sqrt{N}}{\braket{A^2}}\braket{\Im G_t A}\right|\le N^{1/3+(l+3)(\zeta+\xi)},
    \end{equation}
    and so, together with
    \[
    \sum_{\alpha_1,\ldots,\alpha_l} \abs{\kappa_t(\alpha_1,\ldots,\alpha_l)}\lesssim N^{2-l/2},
    \]
    by \eqref{eq:thirdhigher}, we conclude \eqref{eq:gftres}.
\end{proof}

\printbibliography%
\end{document}